\documentclass[11pt]{amsart}

\usepackage{amsmath,amsthm,amsfonts,amssymb,mathrsfs}
\usepackage{mathtools}
\usepackage{float}
\usepackage[margin=1.2in]{geometry}

\input xy
\xyoption{all}

\usepackage{tikz}
\usetikzlibrary{arrows,matrix}

\def\mini{\mathrm{min}}
\def\special{\mathrm{sp}}

\newcommand{\mg}{\mathfrak{g}}

\newcommand{\0}{\mathcal O}
\newcommand{\OC}{{\mathcal{O}}}
\newcommand{\C}{\mathbb{C}}
\DeclareMathOperator{\codim}{codim}
\DeclareMathOperator{\Or}{O}

    \def\OC{{\mathcal{O}}}
\def\PB{{\mathbf P}}

  \def\cg{{\mathfrak c}}

  \def\gg{{\mathfrak g}}  
  \def\hg{{\mathfrak h}}

  \def\mg{{\mathfrak m}}

\def\SG{{\mathfrak S}}  \def\sg{{\mathfrak s}}  
  \def\tg{{\mathfrak t}}

  \def\zg{{\mathfrak z}}  

 \DeclareMathOperator{\ad}{ad}

\newcommand{\nilrad}{{\mathfrak n}}
\newcommand{\nilcone}{{\mathcal N}}

\newcommand{\parti}{\mathcal P}
\newcommand{\g}{{\mathfrak g}}

\DeclareMathOperator{\SL}{SL}
\DeclareMathOperator{\Sp}{Sp}
\DeclareMathOperator{\SO}{SO}
\DeclareMathOperator{\so}{\mathfrak{so}}
\newcommand{\diag}{\rm diag}     
\newcommand{\tr}{\rm tr}     
\newcommand{\Se}{\mathcal{S}_e}
\newcommand{\SOe}{\mathcal{S}_{\0,e}}

 \newcommand{\renc}{\renewcommand}
\renc{\sl}{{\mathfrak{sl}}}
\renc{\sp}{{\mathfrak{sp}}}   

\newtheorem{theorem}{Theorem}[section]
\newtheorem{lemma}[theorem]{Lemma}
\newtheorem{proposition}[theorem]{Proposition}
\newtheorem{corollary}[theorem]{Corollary}

\theoremstyle{definition}

\theoremstyle{remark}
\newtheorem{remark}[theorem]{Remark}

\RequirePackage{color}
\definecolor{myred}{rgb}{0.75,0,0}
\definecolor{mygreen}{rgb}{0,0.5,0}
\definecolor{myblue}{rgb}{0,0,0.65}

\RequirePackage{ifpdf} \ifpdf
\IfFileExists{pdfsync.sty}{\RequirePackage{pdfsync}}{}
\RequirePackage[pdftex,
colorlinks = true,
urlcolor = myblue, 
citecolor = mygreen, 
linkcolor = myred, 
pagebackref,
bookmarksopen=true]{hyperref}
\else
\RequirePackage[hypertex]{hyperref}
\fi

\RequirePackage{ae, aecompl, aeguill} 

\begin{document}
	
	\title{Minimal special degenerations and duality}
	
	\author{Daniel Juteau, Paul Levy and Eric Sommers}
	
	
	\date{\today}
	
\begin{abstract}
		This paper includes the classification, in a simple Lie algebra $\gg$, of the singularities of Slodowy slices 
		between special nilpotent orbits that are adjacent in the partial order on nilpotent orbits.  
		The irreducible components of most singularities are (up to normalization) either 
		a simple surface singularity or the closure of a minimal special nilpotent orbit 
		in a smaller rank Lie algebra.   
		Besides those cases, there are some 
		exceptional cases that arise as quotients of the closure of a minimal  orbit in types $D_n$ by $V_4$, in type $A_2$ 
		by $\SG_2$ or in type $D_4$ by $\SG_4$.
		We also consider the action on the slice of the fundamental group of the smaller orbit.  With this action,
		we observe that under Lusztig-Spaltenstein duality, in most cases, a singularity of 
		simple surface singularity is interchanged with the closure of a minimal special orbit of Langlands dual type (or a cover of it with action).  Lusztig's canonical quotient helps explain when this duality fails.
		This empirical observation generalizes an observation of Kraft and Procesi in type $A_n$, where all nilpotent orbits are special.  We also resolve a conjecture of Lusztig that concerns the intersection cohomology of slices between special nilpotent orbits.
\end{abstract}
	
	\maketitle

\section{Introduction}

\subsection{Minimal degenerations} Let $G$ be a simple algebraic group over $\mathbb C$ and $\g$ its Lie algebra.  Let $\nilcone_{o}:={\mathcal N}(\g)/G$ be the set of nilpotent orbits in $\g$.  The partial order on $\nilcone_{o}$ is defined so that $\0'{<}\0$ whenever $\0' \subsetneq \overline{\0}$ for $\0,\0 \in \nilcone_{o}$, 
where $\overline{\0}$ is the closure of $\0$.  A pair $\0'{<}\0$  is called a {\it degeneration}.  If $\0$ and $\0'$ are adjacent in the partial order
(that is, there is no orbit strictly between them), then the pair is called a {\it minimal degeneration}.
There are two minimal degenerations at either extreme of the poset $\nilcone_{o}$:  the regular and subregular
nilpotent orbits give a minimal degeneration, as does the minimal nilpotent orbit and the zero orbit. 

Given $e \in \nilcone_{o}$, let $\sg \subset \g$ be an $\sl_2$-triple $\{e,h,f\}$ through $e$.   Then 
$\Se:= e + \g^f$, where $\g^f$ is the centralizer of $f$ in $\g$, is called a Slodowy slice.
Associated to any degeneration $\0'{<}\0$ is a smooth equivalence class of singularities ${\rm Sing}(\0,\0')$ \cite{Kraft-Procesi:classical}, which can be represented by the intersection $\SOe:={\mathcal S}_e\cap\overline\0$, where $e \in \0'$.
We call $\SOe$ a {\it Slodowy slice singularity}.

The singularities $\SOe$ of minimal degenerations are known in the classical types by \cite{Kraft-Procesi:GLn} and \cite{Kraft-Procesi:classical} and in the exceptional types by \cite{Fu-Juteau-Levy-Sommers:GeomSP} and \cite{FJLS}, up to normalization for a few cases in $E_7$ and $E_8$.  
These results can be summarized as:
\begin{itemize}
	\item the irreducible components of $\SOe$ are pairwise isomorphic;
	\item if $\dim(\SOe)=2$, then the normalization of an irreducible component of $\SOe$ is isomorphic to 
	$\C^2/\Gamma$ where $\Gamma \subset \SL_2(\C)$ is a finite subgroup, possibly trivial.  
	Such a variety is called a {\it simple surface singularity} when $\Gamma$ is non-trivial. 
	\item if $\dim(\SOe)\geq 4$, then an irreducible component of $\SOe$ is isomorphic to the closure of a minimal nilpotent orbit in some simple Lie algebra, or else is one of four exceptional cases, denoted $m'$, $\tau$, $\chi$, or $a_2/\SG_2$ in 
	\cite{FJLS} and each appearing exactly one time. 
\end{itemize}

\subsection{Action on slices}\label{intro:action on slices} A simple surface singularity $X=\C^2/\Gamma$ corresponds to the Dynkin diagram of a simply-laced Lie algebra (e.g., $A_n$, $D_n$, $E_n$) either by using the irreducible representations of $\Gamma$ as done by McKay, or by looking at the exceptional fiber of the minimal resolution of $X$, which is union of projective lines, whose arrangement yields the Dynkin diagram.   Slodowy defined an action on $X$ by using a normalizing subgroup $\Gamma'$ of $\Gamma$ in $\SL_2(\C)$ \cite[III.6]{Slodowy:book}.
Looking at the image of the action of $\Gamma'$ on the Dynkin diagram, he introduced the notation 
$B_n$ (resp. $C_n$, $F_4$, $G_2$) to denote a simple surface singularity $A_{2n-1}$ (resp. $D_{n+1}$, $E_6$, $D_4$) singularity with an ``outer'' action of $\mathfrak{S}_2$ (resp. $\mathfrak{S}_2$, $\mathfrak{S}_2$, $\mathfrak{S}_3$).  Here, ``outer" refers to the fact that on the corresponding Lie algebra these come from outer automorphisms.
It is also possible to do the same thing for the simple surface singularity $A_{2n}$, where we 
used the notation $A_{2n}^+$ in \cite{FJLS}, when the outer action is included.
Note, however, that this arises from a cyclic group of order four acting on $X$.  

The centralizer $G^e$ of $e$ in $G$ has a reductive part $C(\sg)$, given by the centralizer of $\sg$ in $G$.  
Then $C(\sg)$ acts on $\SOe$ and we are interested in the image of $C(\sg)$ in $\rm Aut(\SOe)$.
Slodowy \cite[IV.8]{Slodowy:book} showed for the regular/subregular minimal degeneration, that $\SOe$ with the action induced from $C(\sg)$ is
exactly the simple surface singularity denoted by the type of $\g$.  This explains his
choice of notation.  

Let $a_n, b_n, \dots, g_2$ denote the closure of the minimal nilpotent orbit according to the type of $\g$.
In \cite{FJLS}, we introduced the notation 
$a_n^+$, $d_n^+$, $e_6^+$, $d_4^{++}$ to denote these varieties with the outer action of 
$\mathfrak{S}_2$, $\mathfrak{S}_2$, $\mathfrak{S}_2$, $\mathfrak{S}_3$, respectively, coming from the outer automorphisms of $\g$.
In {\it op.\!\! cit.}, using these two notions of action, we studied the action of $C(\sg)$ on $\SOe$ for all minimal degenerations, where we found that $C(\sg)$ acts transitively on the irreducible components of $\SOe$ and in some
sense acts as non-trivially as possible on $\SOe$ given the size of
the component group $A(e):=C(\sg)/C^\circ(\sg)$.
 In this paper one of our results is to repeat this calculation for the classical groups (see \S \ref{section:A_action}).


\subsection{Minimal Special Degenerations}\label{subsect:min_sp_intro} Lusztig defined the notion of special representations of the Weyl group $W$ of $G$ \cite{Lusztig:special}, which led him to define the special nilpotent orbits, denoted $\nilcone^{sp}_{o}$, via the Springer correspondence.
The regular, subregular, and zero nilpotent orbits are always special, but the minimal nilpotent orbit is only special 
when $\g$ is simply-laced (types $A_n$, $D_n$, or $E_n$).   In the other types, there is always a unique minimal (nonzero) special nilpotent orbit.   
We denote the closure of the minimal special nilpotent orbits (which are not minimal nilpotent) by 
$b_n^{sp}, c_n^{sp}, f_4^{sp},$ and $g_2^{sp}$, according to the type of $\g$.

In this paper, we classify the Slodowy slice singularities $\SOe$ when $\0$ and $\0'$ are adjacent special orbits, i.e., there is no special orbit strictly between them.    We call these {\it minimal special degenerations}.
Since $\dim(\SOe)=2$ implies the degeneration is already a minimal degeneration, we are left only 
to classify the cases where  $\dim(\SOe)\geq4$.  Our main result on the classification of 
{\bf minimal special degenerations} is summarized as:
\begin{itemize}
	\item the irreducible components of $\SOe$ are pairwise isomorphic;
	\item if $\dim(\SOe)=2$, then the normalization of an irreducible component of $\SOe$ is isomorphic to 
	$\C^2/\Gamma$ where $\Gamma \subset \SL_2(\C)$ is a finite, non-trivial subgroup.
	\item if $\dim(\SOe)\geq 4$, then an irreducible component of $\SOe$ is isomorphic to the closure of a minimal special nilpotent orbit in some simple Lie algebra, or else is isomorphic to one of the following quotients of the closure of a minimal (special) nilpotent orbit:  $a_2/\SG_2$, $d_{n+1}/V_4$ or $d_4/\SG_4$.
\end{itemize}

The singularities $a_2/\SG_2$ and $d_4/\SG_4$ arose in \cite{Fu-Juteau-Levy-Sommers:GeomSP} and 
along with  $d_{n+1}/V_4$, they also appear in the physics literature \cite{Hanany:Coulomb_branches}.

In the case where $\dim(\SOe) \geq 4$, 
the singularities of $\SOe$ are mostly controlled by the simple factors of $\cg(\sg)$ (see Corollary \ref{Corollary:simple_factors_control_singularity} and the remarks after it), just as occurs for most of the minimal degenerations of dimension four or more.

For dimension two, there is a single slice where one of its irreducible components 
is known not to be normal, namely the $\mu$ singularity from \cite{Fu-Juteau-Levy-Sommers:GeomSP}, which occurs once in $E_8$ (it is irreducible).  
We expect the other components of slices of dimension two all to be normal in the case of minimal special degenerations, unlike the case of minimal degenerations.  The components of slices of dimension at least four are all known to be normal.  

The irreducible minimal special degenerations in the classical types $B$, $C$, $D$, 
are listed in Tables \ref{codim2} and \ref{codim4ormore},
in analogy with the classification of Kraft and Procesi for minimal degenerations \cite[Table 1]{Kraft-Procesi:classical}.
The minimal special degenerations of codimension two are already  minimal degenerations and so are contained in \cite{Kraft-Procesi:classical}, except for the action of $A(e)$.  
The notation of $[2B_n]^+,$ means that the image of $C(\sg)$ acts by a Klein $4$-group $V_4$ on $\SOe$,
where one generator switches the two components of the exceptional fiber and a second generator preserves both components of type $A_{2n-1}$, but acts by outer automorphism on each one.  The table assumes that $G$ is 
the orthogonal group $\rm O(2n)$ for type $D_n$, hence making use of the outer $\SG_2$-action of $D_n$.  

In type $D_{n}$ without this outer action, we would 
get these same singularities but without some or all of the action on $\SOe$.
Specifically,  $D_k$ and $d_k$ arise, without the $\SG_2$-action.  
The singularity $[2B_k]^+$ will become $B_k$ for the minimal degenerations
where $\0$ is a very even orbit. 
We discuss this further in \S \ref{Section:outer_autos}.

\begin{table}[htp]  \label{min_special_degens}
	\caption{Minimal special degenerations of codimension two}
	\begin{center}
		\begin{tabular}{|c|c|c|c|c|c|}
			\hline
		    Name of singularity &  $a$ & $b$ & $c$ & $d$ & $e$ \\
			Lie algebra & $\mathfrak{sp}_{2}$ & $\mathfrak{sp}_{2n}$  &  $\mathfrak{so}_{2n+1}$  & $\mathfrak{sp}_{4n+2}$     & $\mathfrak{so}_{4n}$  \\
			&   &$n\geq 2$   & $n \geq 1$ &  $n \geq 1$ &   $n \geq 1$ \\
			$l$ rows removed   & $l  \equiv \epsilon'$ & any & $l  \not \equiv \epsilon'$  &$l  \equiv \epsilon'$ & $l  \equiv \epsilon'$   \\
			$s$ columns removed& $s \not\equiv \epsilon$ & $s \not \equiv \epsilon$ & $s \equiv \epsilon$  & $s \not\equiv \epsilon$ & $s \equiv \epsilon$   \\
			\hline
			$\lambda$ & $[2]$  & $[2n]$  & $[2n{+}1]$   &  $[2n{+}1, 2n{+}1]$   &  $[2n,2n]$    \\
			$\mu$ & $[1,1]$ & $[2n{-}2,2]$ & $[2n{-}1,1,1]$   &  $[2n,2n,2]$  & $[2n{-}1,2n{-}1,1,1]$    \\
			
			\hline
			Singularity &  $C_1$& $C_n$  &   $B_n$  &  $B_n$ & $[2B_n]^+$ \\
			\hline
		\end{tabular} \label{codim2}
	\end{center}
\end{table}

\begin{table}[htp]  \label{min_special_degens_codim_2}
	\caption{Minimal Special Degenerations of codimension 4 or more}
	\begin{center}
		\begin{tabular}{|c|c|c|c|c|c|}
			\hline
			Name of singularity  & $g_{sp}$  & $h$ & $f^1_{sp}$ & $f^2_{sp}$  & $h_{sp}$   \\
			Lie algebra & $\mathfrak{sp}_{2n}$ & $\mathfrak{so}_{2n}$   & $\mathfrak{so}_{2n+1}$     & $\mathfrak{sp}_{4n+2}$ &  $\mathfrak{sp}_{4n}$  \\
			& $n\geq 2$   &$n\geq 3$   &$n\geq2$ & $n\geq2$ &$n\geq2$  \\
			$l$ rows removed  & $l  \equiv \epsilon'$  & $l   \equiv \epsilon'$  &$l  \not \equiv \epsilon'$ & $l  \equiv \epsilon'$  & $l  \not \equiv \epsilon'$  \\
			$s$ columns removed & $s \not\equiv \epsilon$ & $s \equiv \epsilon$   & $s \equiv \epsilon$ & $s \not \equiv \epsilon$ & $s \not \equiv \epsilon$  \\
			\hline
			$\lambda$  & $[2^2,1^{2n-4}]$  & $[2^2,1^{2n-4}]$  &  $[3, 1^{2n-2}]$   &  $[3^2,2^{2n-2}]$    & $[4,2^{2n-2}]$   \\
			$\mu$ & $[1^{2n}]$ & $[1^{2n}]$   &  $[1^{2n+1}]$  & $[2^{2n+1}]$  & $[2^{2n}]$   \\
			codimension & $4n\!-\!2$ & $4n\!-\!6$  & $4n\!-\!2$ & $4n\!-\!2$ & $4n\!-\!2$  \\
			Singularity &  $c^{sp}_n$ & $d_n^{+} $  &  $b^{sp}_n$ & $b^{sp}_n$  &   $d_{n+1}/V_4 $ \\
			\hline		
		\end{tabular} \label{codim4ormore}
	\end{center}
	\label{default}
\end{table}%

\begin{remark}
	The $h$ singularity for $n=2$ is $d_2^+$, which coincides with the $e$ singularity for $n=1$.  We use $d_2^+$ in the graphs for the classical groups since the action of $A(e)$ for the $e$-singularity with $n=1$ is actually only by $\SG_2$.
\end{remark}

The proof that these tables give the classification of minimal special degenerations is given in \S \ref{combinatorial_classification}.  
In \S \ref{codim4_degs}, we establish that the singularities in classical types are as given in Table \ref{codim4ormore},
and in \S \ref{Section:min_special_deg} we complete the story in the exceptional groups.   
In \S \ref{subsect:defn_A_action} and \S\ref{Section:action_codim4_more}, 
we establish the $A(e)$-action both for minimal special degenerations and minimal degenerations.
The graphs at the end of the paper give the results for the exceptional groups and several examples in the classical groups \S \ref{Section:graphs}.

\subsection{Duality}
Using the Springer correspondence, Lusztig defined two maps, which are order-reversing involutions:
$d: \nilcone_{o}^{sp} \to \nilcone_{o}^{sp}$ and $d_{LS}: \nilcone_{o}^{sp} \to {^L} \nilcone_{o}^{sp}$
(see \cite{Carter}).

For $G = GL_n$ all nilpotent orbits are special 
and Kraft and Procesi \cite{Kraft-Procesi:GLn} computed the singularity type of 
$\SOe$ for minimal degenerations (hence, minimal special degenerations).  The singularity 
is either of type $A_k$ or $a_k$ for some $k$.  
Kraft and Procesi observed that if 
the singularity of $(\0, \0')$ is of type $A_k$ then the singularity
of $(d(\0'), d(\0))$ is of type $a_{k}$.   In the case of $GL_n$,
each orbit is given by a partition and the dualities $d=d_{LS}$ are given by taking the transpose partition.

Our duality is a generalization of the Kraft-Procesi observation, but with some wrinkles.
It says that typically an irreducible component of a simple surface singularity (with $A(e)$-action) is interchanged with 
the minimal special orbit of Langlands dual type (after taking the quotient of the $A(e)$-action).
More explicitly, $d_{LS}$ exchanges the following singularities.
\begin{eqnarray*}\label{eqn:main_interchange}
A_n    & \leftrightarrow& a_n \\
B_n      & \leftrightarrow&  a_{2n-1}^+ \text{ or } c_n^{sp} \\  
C_n      & \leftrightarrow&  d_{n+1}^+ \text{ or } b_n^{sp} \\
D_n    & \leftrightarrow & d_n \\
G_2   & \leftrightarrow & d_4^{++} \text{ or } g_2^{sp} \\
F_4   & \leftrightarrow & e_6^{+}  \text{ or } f_4^{sp} \\
E_n     & \leftrightarrow & e_n  
\end{eqnarray*}
The only interchange of dimension two with dimension two is when 
both slices have irreducible components of type $A_1$.  The fact that for each dual pair of orbits, one of the pairs yields a slice of dimension two was observed by Lusztig \cite{Lusztig:adjacency}. 
For the cases with two options on the right, notice that the first option arises as cover of the second (see e.g. \cite{Fu-Juteau-Levy-Sommers:GeomSP}).  Indeed we expect this cover to occur intrinsically since in all these cases $\0$ itself admits such a cover. 
We could also alternatively say that the second option is a quotient of the first by the $A(e)$-action.

There are three families of situations that do not obey this relationship. 
\begin{enumerate}
	\item Sometimes $$C_{n+1}  \leftrightarrow  c_n^{sp} \text{ or } a_{2n-1}^+$$ 
	\item  When  $d_{n+1}/V_4$ or $d_4/\SG_4$ occurs in a dual pair of orbits, we always have 
	\begin{eqnarray*}
		C_n &\leftrightarrow & d_{n+1}/V_4 \\
		G_2 &\leftrightarrow & d_4/S_4 
	\end{eqnarray*}
	\item For the three exceptional special orbits in $E_7$ and $E_8$, 
	\begin{eqnarray*}
		A_2^+    & \leftrightarrow& a_2^+  \text{ or } a_2/S_2\\
		A_4^+      & \leftrightarrow&  a_{4}^+  
	\end{eqnarray*}
\end{enumerate}


In the first case, Lusztig's canonical quotient of $A(e)$ is playing a role.  Namely, the kernel of the map from $A(e)$ to 
the canonical quotient $\bar{A}(e)$ is acting by outer action on $\SOe$.  We denote this property by adding a $*$ to the singularity, $C_{n+1}^*$.  This phenomenon is described in \S \ref{canonical_quotient}.
In the second case, 
there is an impact of the canonical quotient, see again  \S \ref{canonical_quotient}.
In the third case, these cases arise because the only representative of an order two element in $A(e)$ is 
an order $4$ element in $C(\sg)$ 
(see \cite{FJLS}).
We gather the duality results into one theorem in \S \ref{Section:duality_theorem}.


\subsection{Full automorphism group of $\g$}

We also consider, building on work of Slodowy, the case where $G = \rm Aut(\g)$.  For $A_n, E_6$, and $D_4$, we leave this for \S \ref{Section:outer_autos}.   We find that in type $A_n$, all singularities acquire the expected outer action and thus, for example, $A_k^+ \leftrightarrow a_k^+$ for the full automorphism group of $\g$.

To get more uniform statements for type $D_n$, we use $G=\rm O(2n)$ at the beginning and then explain what changes when $G = \SO(2n)$ in \S \ref{subsect:A_action_spec_orth} and \S \ref{subsection:A_action_h_sing}.

\subsection{Three quartets of singularities in classical types $B$, $C$, $D$}

The duality of the last section has a richer structure in types $B$, $C$, $D$.
The internal duality $d$ for $B_n$ and $C_n$, 
together with $d_{LS}$ and the composition $f:=d \circ d_{LS}$, yield $4$ related special orbits (see Figure \ref{dual_squares}).
Applying these $3$ maps to a minimal special degeneration, we find there are only three possible outputs for the four singularities that arise (see Figure \ref{quartet}).

There is also a story that involves $D_n$.  As mentioned above, we work with $G=\rm O(2n)$.  Then there is a subset of the nilpotent orbits in type $C_n$ that is a slight modification of the special orbits, by changing the parity condition in the definition.  
We call these the alternative special nilpotent orbits in type $C$ and denote them by 
$\mathcal N_o^{C,asp}$ in \S \ref{Subsect:special_partitions}.
Its minimal element is the minimal orbit in type $C_n$ of dimension $2n$.  There is a
bijection between $\mathcal N^{D,sp}_o$ and $\mathcal N_o^{C,asp}$, also denoted $f$, 
that preserves the partial order and codimensions (more precisely, it sends an orbit of dimension $N$ to one of dimension $N+2n$).  
This bijection, together with $d_{LS}$ and $d=f \circ d_{LS}$, also gives rise to the same three quartets of singularities
as in Figure \ref{quartet}.
An example is given in Figure \ref{Nilcone_Alt_special_for_C}.
This is also the first case where all three quartets arise.   


\begin{figure}[H]\caption{Dualities}\label{dual_squares}
	\hspace*{-0.2\linewidth}
	\begin{tikzpicture}
		\begin{scope}[shift={(-4,0)}]
			\matrix (BC_interchange)
			[matrix of math nodes,
			nodes in empty cells,
			nodes={outer sep=0pt,minimum size=0pt},
			column sep={1.7cm,between origins},
			row sep={1.5cm,between origins}]
			{
				&& | (a) |  \mathcal N^{B,sp}_o  && | (b) | \mathcal N^{C,sp}_o      \\
				&& | (c) |   \mathcal N^{B,sp}_o  && | (d) | \mathcal N^{C,sp}_o   \\
			};
			
			\draw[to-to] (a) -- (b) node[midway, above] {$f$};
			\draw[to-to] (c) -- (d) node[midway, below] {$f$};
			\draw[to-to] (a) -- (c) node[midway, left]{$d$};  
			\draw[to-to] (b) -- (d) node[midway, left]{$d$};  
			\draw[to-to,dash=on 5pt off 5pt phase  0pt]  (a) -- (d) node[above=.35cm, left =.9cm,scale=.8]{$d_{LS}$};  
			\draw[to-to,dash=on 5pt off 5pt phase  0pt]  (b) -- (c) node[above=.35cm, right =.9cm,scale=.8]{$d_{LS}$};  
		\end{scope}

	\begin{scope}[shift={(4,0)}]
	\matrix (BC_interchange)
	[matrix of math nodes,
	nodes in empty cells,
	nodes={outer sep=0pt,minimum size=0pt},
	column sep={1.7cm,between origins},
	row sep={1.5cm,between origins}]
	{
		&& | (a) |  \mathcal N^{D,sp}_o  && | (b) | \mathcal N^{C,asp}_o      \\
		&& | (c) |  \mathcal N^{C,asp}_o &&   | (d) | \mathcal N^{D,sp}_o   \\
	};
	
	\draw[to-to] (a) -- (b) node[midway, above] {$f$};
	\draw[to-to]  (c) -- (d) node[midway, below] {$f$};
	\draw[to-to] (a) -- (c) node[midway, left]{$d$};  
	\draw[to-to] (b) -- (d) node[midway, left]{$d$};  
	\draw[to-to,dash=on 5pt off 5pt phase  0pt]  (a) -- (d) node[above=.35cm, left =.9cm,scale=.8]{$d_{LS}$};  
	\draw[to-to,dash=on 5pt off 5pt phase  0pt]  (b) -- (c) node[above=.35cm, right =.9cm,scale=.8]{$d_{LS}$};  
\end{scope}
	\end{tikzpicture}
\end{figure}
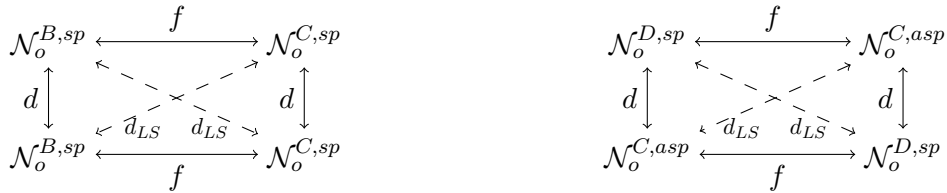

\begin{figure}[H]\caption{The three quartets of possible singularities in classical groups}\label{quartet}
	\hspace*{-0.2\linewidth}
	\begin{tikzpicture}
		\begin{scope}[shift={(-6,0)}]
			\matrix (BC_interchange)
			[matrix of math nodes,
			nodes in empty cells,
			nodes={outer sep=0pt,minimum size=0pt},
			column sep={1.7cm,between origins},
			row sep={1.5cm,between origins}]
			{
				&& | (a) | C_n  && | (b) |B_n    \  \\
				&& | (c) |  c^{\special}_n && | (d) | b^{\special}_n  \\
			};
			
				\draw[to-to] (a) -- (b) node[midway, above] {$f$};
		\draw[to-to] (c) -- (d) node[midway, below] {$f$};
		\draw[to-to] (a) -- (c) node[midway, left]{$d$};  
		\draw[to-to] (b) -- (d) node[midway, left]{$d$};  
		\draw[to-to,dash=on 5pt off 5pt phase  0pt]  (a) -- (d) node[above=.35cm, left =.9cm,scale=.8]{$d_{LS}$};  
		\draw[to-to,dash=on 5pt off 5pt phase  0pt]  (b) -- (c) node[above=.35cm, right =.9cm,scale=.8]{$d_{LS}$};  
		\end{scope}
		
		\begin{scope}[shift={(0,0)}]
			\matrix (BC2_interchange)
			[matrix of math nodes,
			nodes in empty cells,
			nodes={outer sep=0pt,minimum size=0pt},
			column sep={1.7cm,between origins},
			row sep={1.5cm,between origins}]
			{
				&& | (a) | C_{n}  && | (b) |   C_{n+1}^*  \\
				&& | (c) |   c^{\special}_{n}  && | (d) |d^+_{n+1} \\
			};
			
		\draw[to-to] (a) -- (b) node[midway, above] {$f$};
\draw[to-to] (c) -- (d) node[midway, below] {$f$};
\draw[to-to] (a) -- (c) node[midway, left]{$d$};  
\draw[to-to] (b) -- (d) node[midway, left]{$d$};  
\draw[to-to,dash=on 5pt off 5pt phase  0pt]  (a) -- (d) node[above=.35cm, left =.9cm,scale=.8]{$d_{LS}$};  
\draw[to-to,dash=on 5pt off 5pt phase  0pt]  (b) -- (c) node[above=.35cm, right =.9cm,scale=.8]{$d_{LS}$};  
		\end{scope}
		
		\begin{scope}[shift={(6,0)}]
			\matrix (BC3_interchange)
			[matrix of math nodes,
			nodes in empty cells,
			nodes={outer sep=0pt,minimum size=0pt},
			column sep={1.7cm,between origins},
			row sep={1.5cm,between origins}]
			{
				&& | (a) |   C_{n}   && | (b) |   [2B_{n}]^+ \\
				&& | (c) |  c^{\special}_{n}     && | (d) |  d_{n+1}/V_4  \\
			};
			
			\draw[to-to] (a) -- (b) node[midway, above] {$f$};
	\draw[to-to] (c) -- (d) node[midway, below] {$f$};
	\draw[to-to] (a) -- (c) node[midway, left]{$d$};  
	\draw[to-to] (b) -- (d) node[midway, left]{$d$};  
	\draw[to-to,dash=on 5pt off 5pt phase  0pt]  (a) -- (d) node[above=.35cm, left =.9cm,scale=.8]{$d_{LS}$};  
	\draw[to-to,dash=on 5pt off 5pt phase  0pt]  (b) -- (c) node[above=.35cm, right =.9cm,scale=.8]{$d_{LS}$};  
		\end{scope}
		
	\end{tikzpicture}
\end{figure}
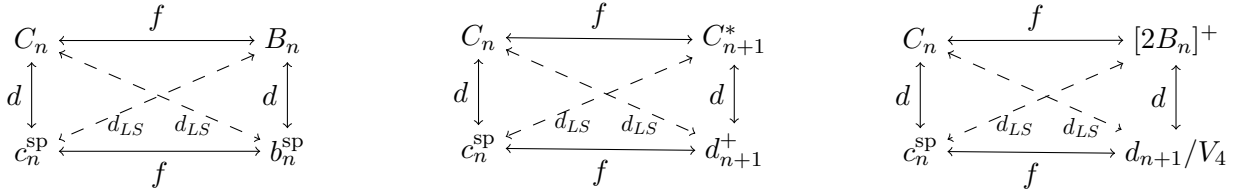

\begin{figure}[H]\caption{Duality between $\parti_D^{sp}(10)$ and $\parti_C^{asp}(10)$ } \label{Nilcone_Alt_special_for_C}
	\begin{tikzpicture}
		\hspace*{-0.15\linewidth}
		\begin{scope}[shift={(-4,0)}]
			\matrix (D4special)
			[matrix of math nodes,
			nodes in empty cells,
			nodes={outer sep=0pt,minimum size=0pt},
			column sep={1.8cm,between origins},
			row sep={1cm,between origins}]
			{
				&&& | (1) | [7,1]  \\
				&&& | (2) |[5,3] \\
				&& | (3) |[4^2]&& | (4) |[5,1^3] \\
				&&& | (5) |[3^2,1^2]  \\
				&& | (6) |[2^4] && | (7) |[3,1^5] \\
				&&& | (8) |[2^2,1^4] \\
				&&& | (9) | [1^{8}] \\	
			};
			

			\draw (1) -- (2) node[midway, right] {${\color{blue}C_3}$};
		\draw (2) -- (3) node[midway, above] {${\color{blue}C_1}$};
		\draw (2) -- (4) node[midway, right]{${\color{blue}B_1}$};  
		\draw (3) -- (5) node[midway, right,scale=.85]{${\color{blue}[2B_2]^+}$};
		\draw (4) -- (5) node[midway, below,scale=.9]{${\color{blue}C_2}$};	 
		\draw (5) -- (6) node[midway, above,scale=.96]{${\color{blue}c^{\special}_2}$};	  
		\draw (5) -- (7) node[midway, above,scale=.9]{${\color{blue}b^{\special}_2}$};	  
		\draw (6) -- (8) node[midway, below,scale=.9]{${\color{blue}d_2^+}$};	  
		\draw (7) -- (8) node[midway, below]{${\color{blue}C_1}$};	  
		\draw (8) -- (9) node[midway, right]{${\color{blue}d_4^+}$};

			\node [below=1cm, align=flush center,text width=8cm] at (9)
			{
				$D_4$ Minimal Special Degenerations  
			};
			\end{scope}
		
		\begin{scope}[shift={(4,0)}]
			\matrix (C4antispecial)
			[matrix of math nodes,
			nodes in empty cells,
			nodes={outer sep=0pt,minimum size=0pt},
			column sep={1.8cm,between origins},
			row sep={1cm,between origins}]
			{
				&&& | (1) | [8]  \\
				&&& | (2) |[6,2] \\
				&& | (3) |  [4^2]&& | (4) |  [6,1^2]\\
				&&& | (5) |[4,2^2]  \\
				&& | (6) |[2^4] && | (7) |[4,1^4] \\
				&&& | (8) |[2^3,1^2] \\	
				&&& | (9) | [2,1^{6}] \\	
			};
			
			\draw (1) -- (2) node[midway, right] {${\color{blue}C^*_4}$};
			\draw (2) -- (3) node[midway, above] {${\color{blue}C^*_2}$};
			\draw (2) -- (4) node[midway, right]{${\color{blue}C_1}$};  
			\draw (3) -- (5) node[midway, above]{${\color{blue}C_2}$};
			\draw (4) -- (5) node[midway, above,scale=.9]{${\color{blue}B_2}$};	 
			\draw (5) -- (6) node[inner sep=7pt, midway, above,scale=.75]{${\color{blue}d_3/V_4}$};	  
			\draw (5) -- (7) node[midway, above]{${\color{blue}c^{\special}_2}$};	  
			\draw (6) -- (8) node[midway, below]{${\color{blue}C_1}$};	  
			\draw (7) -- (8) node[midway, below]{${\color{blue}B_1}$};	  
			\draw (8) -- (9) node[midway, right]{${\color{blue}c^{\special}_3}$};
		
			\node [below=1cm, align=flush center,text width=8cm] at (9)
			{
				$C_4$ {\bf Alternative} Minimal Special Degenerations  
			};
		\end{scope}
		
	\end{tikzpicture}
\end{figure}

\subsection{Lusztig's Weyl group conjecture}  \label{intro_lusztig_conj} 

In \cite[\S 0.4]{Lusztig:adjacency}, Lusztig attached a Weyl group $W'$ to each minimal special degeneration.  
He then made a conjecture relating the exponents of $W'$ to what amounts to the $C(\sg)$-invariant 
part of the intersection homology $\rm IH^*(\SOe)$ when $\dim (\SOe) \geq 4$.   In \S \ref{lusztig} we prove his 
conjecture, which is open in types $B$, $C$, $D$, although we have to modify slightly the $W'$ that he attaches to those minimal degenerations $(\lambda, \mu)$ in type $D_n$ where there is a single odd part in $\mu$.

\subsection{Acknowledgments}
This work, major parts of which were sketched in 2012, is a continuation of the papers [FJLSa], [FJLSb] that were jointly authored with Baohua Fu.  We thank him for his vital contribution to the project from its inception.

\section{Background material in the classical group case}\label{partitions}

\subsection{Notation on partitions}\label{partition_notation}

In the classical groups it will be helpful  
to have a description of 
the elements of $\nilcone_o$ 
and the map $d$
in terms of partitions.  We introduce that notation
following the references \cite{C-M},  \cite{Carter}, \cite{Spaltenstein}.

Let $\parti (N)$ denote the set of partitions of $N$.
For $\lambda \in \parti(N)$,  we write $\lambda = [\lambda_1, \dots, \lambda_k]$, 
where $\lambda_1 \geq \dots \geq \lambda_k > 0$ and $|\lambda|:=  \sum \lambda_j$ is equal to $N$. 
Define 
$$m_\lambda(s)= \# \{j \ | \ \lambda_j = s \},$$
the multiplicity of the part $s$ in $\lambda$.   We use $m(s)$ if the partition is clear.
Sometimes we write $[\dots, s^{m(s)}, \dots]$ instead of 
$$[\dots, \overbrace{s, s, \dots, s}^{m(s)}, \dots]$$ 
for a part $s$ in $\lambda$.  
The set of nilpotent orbits $\nilcone_{o}$ in $\g=\mathfrak{sl}_{n}$ under the 
adjoint action of $G=SL_n$  is in bijection with 
$\parti (n)$. 

For $\epsilon \in \{ 0,1 \}$,
let $V=V_\epsilon$ be a vector space, of dimension $N$,
with a nondegenerate bilinear form satisfying 
$\langle v, v' \rangle = (-1)^\epsilon \langle v', v \rangle$ for $v,v' \in V$. 
Let $\g(V)$ be the Lie algebra associated to the form on $V$, 
so that $\g(V) = \mathfrak{so}_{N}$ when $\epsilon =0$
and $\g(V) = \mathfrak{sp}_N$ when $\epsilon=1$ and $N$ is even.

Let $$\parti_{\epsilon} (N):= \{\lambda \in \parti(N)  \ | \   m(s) \equiv 0 \text{ whenever }   s \equiv \epsilon  \},$$
where all congruences are modulo $2$.
Then the set of nilpotent orbits $\nilcone_{o}$ in $\g(V)$ under the  group $G=G(V)$
preserving the form is given by 
$\parti_{1} (2n)$ when $\g$ is of type $C_n$;
by $\parti_{0} (2n+1)$ when $\g$ is of type $B_n$; and by  
$\parti_{0} (2n)$ when $\g$ is of type $D_n$, except that those partitions
with all even parts correspond to two orbits in $\nilcone_{o}$ (called
the very even orbits, where there are two orbits interchanged by the orthogonal group).
We will also refer to $\parti_{1} (2n)$ as $\parti_{C} (2n)$;
to $\parti_{0} (2n+1)$ as  $\parti_{B} (2n+1)$;
and to $\parti_{0} (2n)$ as  $\parti_{D} (2n)$.
We sometimes call a partition $\lambda \in \parti_\epsilon(N)$ an $\epsilon$-partition.
For $\lambda \in \parti(N)$ or $\lambda \in \parti_\epsilon(N)$, we denote by $\0_\lambda$ the corresponding
nilpotent orbit in $\g$.

Define the height of a part $s$ in $\lambda$ to be the number 
$$h_\lambda(s):=\#\{ \lambda_j \, | \, \lambda_j \geq s \}.$$
We write $h(s)$ if the partition is clear.
In terms of Young diagrams, the  position $(s,h(s))$ is a corner of the diagram, writing each part $\lambda_i$ as the boxes with upper right corner $(1,i), \dots (\lambda_i,i)$.
In other words, we have $\lambda_{h(s)} = s$ and $\lambda_{h(s)+1} < \lambda_{h(s)}$.

The dual or transpose partition of $\lambda$, denoted 
$\lambda^*$, is defined by 
$$(\lambda^*)_i = \# \{ j \ | \ \lambda_j \geq i \}.$$
If we set $j = h(s)$, then $\lambda^*$ is the partition with part $h(s)$ occurring $\lambda_j - \lambda_{j+1}=s- \lambda_{j+1}$ times.

The set $\parti(N)$ 
is partially ordered by the dominance order on partitions,
where $\mu \preceq \lambda$ whenever $\sum_{i=1}^k \mu_i \leq  \sum_{i=1}^k \lambda_i$ for all $k$.
This induces a partial ordering on the sets  
$\parti_{C} (2n)$, $\parti_{B} (2n+1)$, and $\parti_{D} (2n)$
and these partial orderings coincide
with the partial ordering on nilpotent orbits given
by the closure ordering.  We will refer to nilpotent orbits and partitions
interchangeably in the classical groups (with the
caveat mentioned earlier for the very even orbits in type $D$).

Let $X=B$, $C$, or $D$.
Let $N$ be even (resp. odd) if $X$ is of type $C$ or $D$ (resp. $B$).
The $X$-collapse of $\lambda \in \parti(N)$ is the
partition $\lambda_X \in \parti_{X}(N)$ satisfying
$\lambda_X \preceq \lambda$ and such that 
if $\mu \in \parti_{X}(N)$ and $\mu \preceq \lambda$, then $\mu \preceq \lambda_X$.
The $X$-collapse always exists and is unique. 

\subsection{Special partitions and the duality maps} \label{Subsect:special_partitions}

The special nilpotent orbits were defined by Lusztig \cite{Lusztig:special}.
Denote by $\nilcone^{sp}_{o}$ the special nilpotent orbits in $\g=\g(V)$.
All nilpotent orbits are special in type $A$.  Here we describe the special nilpotent orbits in 
types $B$, $C$, and $D$, as well as introduce a second subset of $\nilcone_{o}$, which behaves like special orbits.
We define four sets of partitions, with $\epsilon' \in \{0,1\}$, as follows
\begin{equation}\label{special_condition}
	\parti_{\epsilon,\epsilon'} (N) := \{\lambda \in \parti_{\epsilon}(N)  \ | \   h(s) \equiv \epsilon' \text{ whenever }  s \equiv \epsilon  \}.
\end{equation}
Because of the $s=0$ case, for $N$ odd, the set is nonempty only when $(\epsilon,\epsilon') = (0,1)$.
For $N$ even, the set is nonempty for $(\epsilon,\epsilon') \in \{(0,0),(1,0),(1,1)\}$.
Then the partitions for the special orbits in type $B_n,C_n,D_n$ are given by
$\parti^{sp}_B (2n+1):= \parti_{0,1} (2n+1)$, $\parti^{sp}_C (2n):=  \parti_{1,0} (2n)$,
and $\parti^{sp}_D (2n):= \parti_{0,0} (2n)$.
The fourth case leads to a second subset of $\parti_C (2n)$, which is
$\parti^{asp}_C (2n):= \parti_{1,1} (2n)$.   We refer to these nilpotent orbits in type $C$
as the alternative special nilpotent orbits.

Each of $\parti^{sp}_B (2n+1)$, $\parti^{sp}_C (2n)$, $\parti^{sp}_D (2n)$,
and $\parti^{asp}_C (2n)$ inherits the partial order from the set of all partitions and 
this agrees with the one coming from inclusion of closures of the corresponding nilpotent orbits. 

The sets $\parti^{sp}_B (2n\!+\!1)$ and $\parti^{sp}_C (2n)$ are in  bijection
(see \cite{Spaltenstein}, \cite{Kraft-Procesi:special}), and also
the sets $\parti^{sp}_D (2n)$ and $\parti^{asp}_C (2n)$ are in bijection, as we now describe.

Given $\lambda =[\lambda_1 \geq \dots 
\geq \lambda_{k-1} \geq \lambda_k > 0]$,
let 
$$\lambda^{-} = [\lambda_1 \geq \dots \geq \lambda_{k-1}\geq \lambda_k -1 ]$$
and
$$\lambda^{+} = [\lambda_1 +1 \geq \dots \geq \lambda_{k-1} \geq \lambda_k].$$
Then the bijections are given as follows, using $f$ for each map: 
\begin{equation}
	\begin{aligned}
	& f_{BC}: \parti^{sp}_B (2n+1) \to \parti^{sp}_C (2n) \text{ given by } f(\lambda) = (\lambda^{-}) _C \\
	& f_{CB}: \parti^{sp}_C (2n) \to \parti^{sp}_B(2n+1) \text{ given by }  f(\lambda) = (\lambda^{+}) _B \\
	& f_{DC}: \parti^{sp}_D(2n) \to \parti^{asp}_C (2n) \text{ given by }  f(\lambda) = ((\lambda^{+})^-) _C \\
	& f_{CD}: \parti^{asp}_C (2n) \to \parti^{sp}_D(2n) \text{ given by }  f(\lambda) = \lambda _D  \\
	\end{aligned} 
\end{equation}
Note that in general $f$ maps $\parti_{\epsilon,\epsilon'}$ to $\parti_{1-\epsilon,1-\epsilon'}$.

Each of these maps respects the partial order.  The first two maps are dimension preserving (and codimension preserving).
The second two maps are codimension preserving (as we shall see).  More precisely, the $f_{DC}$ map sends
an orbit of dimension $N$ to one of dimension $N+2n$.  The shift is because the minimal orbit in $C_n$ is the minimal element 
in $\parti^{asp}_C (2n)$.

Write $d(\lambda)$ for $\lambda^*$.
It is known (or easy to show) that $d$ determines bijections between the following sets:
\begin{equation} 
	\begin{aligned}
	\parti^{sp}_B (2n\!+\!1) &\xrightarrow{\text{  \,  d  \,  }}  \parti^{sp}_B (2n\!+\!1) \\
	\parti^{sp}_C (2n) &\xrightarrow{\text{  \,  d  \,  }}  \parti^{sp}_C (2n) \\
	\parti^{sp}_D (2n) &\xleftrightarrow{\text{  \,  d  \,  }}  \parti^{asp}_C (2n) \\
	\end{aligned}
\end{equation}
It is order-reversing since this holds for all partitions.   We refer to $d$ as the {\it internal duality}
or just {\it transpose}.

It is known that $d \circ f = f \circ d$ and the duality $d_{LS}$ of Lusztig-Spaltenstein is 
given by $d_{LS}= d \circ f = f \circ d$. 
We have squares relating the three kinds of maps between the orbits (or their 
corresponding partitions) as shown in Figure \ref{dual_squares}.   

\subsection{Explicit description of $f$ maps}

We now describe more specifically how the $f$ maps work.  


Let $X$ be one of the types $B$, $C$, $D$.  
Let $\lambda \in \parti(N)$ with $N$ even for type $C$ and $D$ and odd for type $B$.
We want to find $\lambda_X$.   Set $\epsilon = \epsilon_X$.
List the parts $s$ in $\lambda$ with $s \equiv\epsilon$ and $m(s)$ odd as 
$a_1 > a_2 > \dots >  a_{2n} \geq 0$.    
For $X=C$, these are the odd parts, so there is an even number of them.  But for  $X=B$ or $X=D$, since $\epsilon=0$,
we will add a single part equal to $0$, if necessary, so that there is an even number of even parts with odd multiplicity.
Next, between $a:= a_{2i-1}$ and $c:=a_{2i}$, list the parts $s \equiv \epsilon$ as $b_1 > b_2 > \dots > b_j$,
which necessarily have $m(b_i)$ even.  Ignoring 
the parts not congruent to $\epsilon$ , then $\lambda$ will look locally like
$$a^{m(a)}, b_1^{m(b_1)}, \dots , b_j^{m(b_j)},  c^{m(c)}.$$
Then under the collapse of $\lambda$ to $\lambda_X$, these values will change to
$$a^{m(a)-1},a{-}1,b_1{+}1, b_1^{m(b_1)-2},b_1{-}1, \dots , b_j{+}1, b_j^{m(b_j)-2}, b_j{-}1, c{+}1, c^{m(c)-1}$$
so that the multiplicities of the parts congruent to $\epsilon$ are now even, as required to be in $\parti_X(N)$.
The other parts of $\lambda$ are unaffected under the collapse.
As a result of this rule, there is a formula for the collapse for a part $s$ based on its height $h(s)$ and its multiplicity $m(s)$.
\begin{lemma}
	Let $X$ be $B, C,$ or $D$ and $\epsilon = \epsilon_X$.
	Let $\lambda \in \parti(N)$ with $N$ as above.  Assume $m(s)$ is even if $s \not \equiv \epsilon$.
	Let $s \equiv \epsilon$ be a part in $\lambda$.   Then $[s^{m(s)}]$ in $\lambda$ changes to
	the following in $\lambda_X$:
	\begin{eqnarray*}
		& [s^{m(s)-1},s{-}1] &  \text{ if } h(s)\equiv 1, m(s) \equiv 1 \\
		& [s{+}1, s^{m(s)-1}] &  \text{ if } h(s)\equiv 0, m(s) \equiv 1 \\ 
		& [s{+}1, s^{m(s)-2},s{-}1] &  \text{ if } h(s)\equiv 1, m(s) \equiv 0 \\ 		 
		&[s^{m(s)} ] &  \text{ if } h(s)\equiv 0, m(s) \equiv 0
	\end{eqnarray*} 
\end{lemma}

\begin{proof}
	Since $h(s) = \sum_{j \geq s} m(j)$, it is clear that $h(s) \equiv  \#\{m(j) \ | \ m(j) \text{ is odd and } j \geq s\}$.
	Since any part $s$ with  $s \not \equiv \epsilon$ has $m(s)$ even, 
	it follows that 
	$$h(s) \equiv  \#\{m(j) \ | \ m(j) \equiv 1, j \geq s, \text{ and } j \equiv \epsilon\}.$$
	So the four conditions in the lemma specify whether the part $s$ plays the role of some $a_{2i-1}$, $a_{2i}$, $b_k$, or
	a part between some $a_{2i}$ and $a_{2i+1}$ and hence unaffected by the collapse.
\end{proof}

Let $X$ be one of the types $B,C,D$ or $C'$.  Here $C'$, refers to the alternative special setting.
Now we can say what happens under the $f$ maps passing from $X$ to type $f(X)$.
\begin{lemma}\label{f_map}
	Let $\lambda \in \parti^{sp}_{X}(N)$. 
	Let $s \not \equiv \epsilon_X$ be a part in $\lambda$. 
	Then $f(\lambda)$ is computed by replacing each occurrence of $[s^{m(s)}]$ by 
		\begin{eqnarray}
		& [s^{m(s)-1},s{-}1] &  \text{ if } h(s)\equiv  \epsilon'_X, m(s) \equiv 1 \\
		& [s{+}1, s^{m(s)-1}] &  \text{ if } h(s)\not \equiv  \epsilon'_X, m(s) \equiv 1 \\ 
		& [s{+}1, s^{m(s)-2},s{-}1] &  \text{ if } h(s)\equiv \epsilon'_X, m(s) \equiv 0 \\ 		 
		&[s^{m(s)} ] &  \text{ if } h(s)\not \equiv  \epsilon'_X, m(s) \equiv 0 
	\end{eqnarray} 
\end{lemma}

\begin{proof}
	First, if $s \not \equiv \epsilon_X$, then $s \equiv \epsilon_{f(X)}$.
	For $f_{BC}$ and $f_{CD}$ the map is just the ordinary collapse (except for the smallest two parts in type $B$).  
	In these cases, $\epsilon'_X=1$ and we are in the situation of the previous lemma when performing the collapse in type $f(X)$.
	In type $B$,  there are a couple of cases to check that the effect of $\lambda_k$ being replaced by $\lambda_k{-}1$ is consistent with the above cases.
		
	On the other hand, for $f_{DC}$ and $f_{CB}$ we have $\epsilon'_X=0$.  
	For the $f$ map, we first increase $\lambda_1$ by 1 and then perform the collapse for type $f(X)$.
	This $1$, under the collapse, moves down to the first part $x$ with $x \not \equiv \epsilon_X$.  By the assumption that
	the parts congruent to $\epsilon_X$ have even multiplicity,
	we have that $h(s)$ odd.  So the rule is correct for the part $x$.  Call this new partition, where $x$ changes to $x+1$, $\lambda'$.  Then 
	$$h_{\lambda'}(s) \not \equiv  \#\{m_{\lambda'}(j) \ | \ m_{\lambda'}(j) \equiv 1, j \geq s, \text{ and } j \equiv \epsilon_{f(X)} \}.$$
	Since $\epsilon'_X=0$, the previous lemma gives the result again for the collapse of $\lambda'$, which is $f(\lambda)$.
	For $f_{DC}$,  there are a couple of cases to check that the effect of $\lambda_k$ being replaced by $\lambda_k-1$ is consistent with the above cases.
\end{proof}

\subsection{Special pieces result}\label{special_pieces}

Spaltenstein \cite{Spaltenstein} showed that each non-special orbit $\0'$ belongs to the closure of a unique special orbit $\0$,
which is minimal among all special orbits whose closure contains $\0'$.  That is,
if a special orbit contains $\0'$ in its closure, then it contains $\0$ in its closure. 

We now describe the process for finding the partition $\lambda$ for $\0$ given the partition $\nu$ for $\0'$.
Let $X$ be one of the four types $B,C,D,C'$.  Let $\nu \in \parti_X(N)$ be non-special.
Let $S$ be the collection of parts $s$ in $\nu$ such that $s \equiv \epsilon_X$ and 
$h(s) \not\equiv \epsilon'_X$.  These are the parts that fail the condition 
for $\nu$ to be in $\parti_{\epsilon,\epsilon'}$ as required by \eqref{special_condition}.  
Note that $s \equiv \epsilon_X$ means that $m(s)$ is even, so $m(s) \geq 2$.
Let $\lambda$ be obtained from $\nu$
by replacing the subpartition
$[s^{m(s)}]$ in $\nu$ by $[s{+}1,   s^{m(s)-2}, s{-}1]$, for each
$s \in S$.  It is clear that $\lambda\in \parti_X(N)$
and that it satisfies the special condition in (\ref{special_condition}), so lies in $\parti^{sp}_X(N)$.
In fact, $\lambda$ is the partition for $\0$ by \cite{Kraft-Procesi:special} for the cases of $B,C,D$ (the case of $C'$ is similar).

\subsection{Removing rows and columns from a partition}\label{subsect:canceling_rows_cols}

In \cite{Kraft-Procesi:GLn} and \cite{Kraft-Procesi:classical}, Kraft and Procesi defined two operations that take a pair of partitions 
$(\lambda,\mu) \in \parti(N)$ to another pair of partitions.
Viewing the partitions as Young diagrams, the first operation is removing common initial rows of $\lambda$ and $\mu$ 
and the second operation is removing common initial columns.  

\subsubsection{Type $A$}
More precisely, we say $(\lambda, \mu)$ is {\it leading-row-equivalent (after removing $r$ rows)} to $(\lambda', \mu')$  if
$\lambda_i = \mu_i$ for $i \leq r$, while $\lambda_{i}=\lambda'_{i-r}$ and $\mu_{i}=\mu'_{i-r}$ for $i > r$. 
We say $(\lambda, \mu)$ is {\it column-equivalent (after removing $s$ columns)} to $(\lambda', \mu')$ if
$\lambda_i = \mu_i$ for $i > \ell$ and $\lambda_{i}=\lambda'_{i}+s$ and $\mu_{i}=\mu'_{i}+s$ for $i\leq\ell$,
where $\ell = \max\{ i \ | \ \lambda_i > s \}$.  
In both cases, $|\lambda'| = |\mu'|$,  so $\lambda'$ and $\mu'$ are partitions of the same integer.
We say  $(\lambda, \mu)$ is {\it equivalent} to $(\lambda', \mu')$ if they are related by a sequence of these two equivalences, and it follows in that case when $\lambda \preceq \mu$ that 
\begin{enumerate}
	\item $\lambda' \preceq \mu'$
	\item $\codim_{\bar{\0}_\lambda} \0_\mu = \codim_{\bar{\0}_{\lambda'}} \0_{\mu'}$ 
	\item  The singularity of $\bar{\0}_\lambda$ at $\0_\mu$ is smoothly equivalent to the singularity of $\bar{\0}_{\lambda'}$ at $\0_{\mu'}$. 
\end{enumerate}
for the corresponding nilpotent orbits  in $\mathfrak{sl}_n$ \cite{Kraft-Procesi:GLn}.

\subsubsection{Other classical types}

For $\epsilon \in \{ 0, 1\}$ and $\g=\g(V_\epsilon)$  as in \S \ref{partition_notation}, 
similar results hold as above hold when we cancel $r$ leading rows and $s$ columns, with an additional condition.
Let  $\lambda, \mu \in \parti_\epsilon(N)$
and assume when we cancel $r$ leading rows that 
\begin{equation}\label{row_cancellation_condition}
[\lambda_1, \dots, \lambda_r] \text{ is an } \epsilon\text{-partition.}
\end{equation}
This condition always holds if we choose the maximal possible number of rows to cancel between $\lambda$ and $\mu$.
If \eqref{row_cancellation_condition} holds,  
then $\lambda'$ and $\mu' $ are ${\tilde \epsilon}$-partitions,  with $\tilde \epsilon \equiv \epsilon + s$, where $s$ is the number of columns canceled.
Then the above three results hold when the nilpotent orbits are considered in $\g$ \cite[\S 13]{Kraft-Procesi:classical}.
A pair of partitions $(\lambda, \mu)$ is {\it irreducible} if no common rows or columns can be canceled.

Next, we say $(\lambda, \mu) \in \parti_\epsilon(N)$ is  {\it $\epsilon$-row-equivalent} 
to $(\lambda', \mu') \in  \parti_\epsilon(m)$, if the latter is obtained from the former by canceling some leading and some trailing rows
of the Young diagram.  Namely,
there exist $r, r' \in \mathbb N$ so that $\lambda_i = \mu_i$ for $i \leq r$ and $i \geq r'$, 
while $\lambda_{i}=\lambda'_{i-r}$ and $\mu_{i}=\mu'_{i-r}$ for $r< i < r'$.    We pad the partitions by adding zeros so that both partitions have the same number of parts. 
If we set $\nu = [\lambda_1, \dots, \lambda_r, \lambda_{r'}, \lambda_{r'+1}, \dots]$,
then $\nu$ is also an $\epsilon$-partition.
We also say $(\lambda, \mu)$ is {\it locally of the form} $(\lambda', \mu')$.

Now suppose that $(\lambda, \mu)$ is {\it $\epsilon$-row-equivalent}  to $(\lambda', \mu')$.
Let $V=V_\epsilon$.
Then, as in \cite[\S 13.4]{Kraft-Procesi:classical}, there is an orthogonal decomposition $V = V_1 \oplus V_2$, 
with $\dim V_1 = |\lambda'| = |\mu'|$ and $\dim V_2 = |\nu|$ and the $V_i$ carry a nondegenerate $\epsilon$-form by restriction from $V$.
Moreover, $\lambda', \mu' \in \parti_\epsilon(\dim V_1)$ and  $\nu \in \parti_\epsilon(\dim V_2)$, so we can pick 
nilpotent elements $x_1, e_1 \in \g(V_1)$ with partitions $\lambda', \mu'$, respectively,
and $e_2 \in \g(V_2)$ with partition $\nu$.  Then $x = x_1 +e_2$ has partition $\lambda$ and $e = e_1+e_2$ has partition $\mu$.  
The arguments in \S13 in \cite{Kraft-Procesi:classical} give
\begin{proposition} \label{slice_row_equiv}
	Choose an $\mathfrak{sl}_2$-triple for $x_1$ in $\g(V_1)$. 
	Then the natural map of $\g(V_1)$ to $\g(V_1)+e_2 \subset \g$ gives an isomorphism of 
the slice ${\mathcal S}_{\0_{\lambda'},e_1}$ in $\g(V_1)$ to the slice ${\mathcal S}_{\0_\lambda,e}$ in $\g$.
\end{proposition}
The key ideas in the proof are that both slices have the same dimension (by the codimension result) and the fact that 
the closure of any nilpotent orbit in $\mathfrak{gl}_N$ is normal.

We note that if $(\lambda', \mu')$ are obtained from $(\lambda, \mu)$ by removing $r$ leading rows and $s$ columns
and if condition \eqref{row_cancellation_condition} holds, then 
$(\lambda, \mu)$ is {\it $\epsilon$-row-equivalent} to :  
\begin{equation}\label{row_equivalent}
	(\lambda'', \mu''):=([\lambda'_1+s, \lambda'_2+s, \dots ], [\mu'_1+s, \mu'_2+s, \dots])
\end{equation}
Finally, we call $(\lambda, \mu)$ and $(\lambda'', \mu'')$ locally equivalent, or say locally 
$(\lambda, \mu)$ is equal to  $(\lambda'', \mu'')$.

In the next section \S \ref{combinatorial_classification},
 we show that each pair of partitions corresponding to a minimal special degeneration 
 in orthogonal and symplectic type is equivalent to 
 a unique pair of partitions  $(\lambda, \mu)$ of $N$ for a unique smallest $N$.  
 These pairs are irreducible in the sense of Kraft and Procesi:  
 the maximal possible number of common rows and columns has been removed from the original pair of partitions to obtain $(\lambda, \mu)$.

\section{Combinatorial classification of minimal special degenerations in $B$, $C$, $D$}\label{combinatorial_classification}

%

\begin{theorem}
	Let $(\lambda, \mu) \in \parti_{\epsilon,\epsilon'}$ be partitions corresponding to a minimal special degeneration in the corresponding classical
	Lie algebra.  Then $(\lambda, \mu)$ is equivalent to a unique entry in Table \ref{codim2} or Table \ref{codim4ormore}.  
\end{theorem}

\begin{proof}
	If a minimal special degeneration $(\lambda, \mu)$ is not already minimal, then there exists a non-special orbit $\0_{\nu}$ 
	such that $\0_{\mu}  \leq \0_{\nu}  \leq \0_{\lambda}$, and such that $(\nu, \mu)$ is a minimal degeneration.  Hence, the latter would 
	be one of the entries in the Kraft-Procesi list \cite[\S 3.4]{Kraft-Procesi:classical}.   
	We need to show that $(\lambda, \mu)$ must 
	be equivalent to one of the five cases in Table \ref{codim4ormore}.    
	
	First, since  $\0_{\nu}$ is not special, there is a unique special orbit whose closure contains $\0_{\nu}$ and which is contained in the closure of all other special orbits whose closure contains $\0_{\nu}$  (see \S\ref{special_pieces}). 
	Consequently, $\0_{\lambda}$ must be this orbit, as we are assuming the degeneration $(\lambda, \mu)$ is minimal among special degenerations.   

	Next, we will show that $(\nu, \mu)$  cannot be one of the cases in Table \ref{codim2}.
	Let $X$ be the type of the Lie algebra and $\epsilon = \epsilon_X$, $\epsilon' = \epsilon'_X$.
	
	If $(\nu, \mu)$ is type $a$, then locally it is $([s{+}2,s], [(s{+}1)^2] )$ where $s \not \equiv \epsilon$.  
	Since $s{+}1 \equiv \epsilon$ and $s{+}1$ appears exactly twice, 
	the heights satisfy 
	$h_\nu(x) \equiv h_\mu(x)$ for all $x \equiv \epsilon$.  
	This means that $\nu$ must be special since $\mu$ is special.
	Therefore the type $a$ minimal degeneration cannot occur 
	between a larger orbit that is not special and a smaller orbit that is special.
	
	If $(\nu, \mu)$ is type $b$, then locally it is $([s{+}2n,s], [s{+}2n{-}2,s{+}2] )$ where $s \not \equiv \epsilon$.  
	Hence, all four of $s{+}2n,s, s{+}2n{-}2,$ and $s{+}2$ are not congruent to $\epsilon$.
	As in the previous case, $h_\nu(x) \equiv h_\mu(x)$ for all $x \equiv \epsilon$, and again this forces $\nu$ to be special too, a contradiction.
	
	If $(\nu, \mu)$ is of type $c,d$ or $e$, 
	it will be possible for $\nu$ to be non-special, but we will show that then the degeneration $(\lambda, \mu)$
	is not minimal among degenerations between special orbits.
	
	For type $c$, the pair $(\nu, \mu)$ is locally $([s\!+\!2n\!+\!1,s^2], [s\!+\!2n\!-\!1,(s\!+\!1)^2] )$ where $s \equiv \epsilon$.
	In this case $\nu$ will be non-special, as noted in the Table \ref{codim2}, 
	exactly when the number of rows $l$ removed is congruent to $\epsilon'$.
	This  means $h_\nu(s)= l+3 \not \equiv \epsilon'$ (and  necessarily $s \geq 1$).
	If that is the case, 
	then by \S \ref{special_pieces}, $\lambda$ must locally be $[s\!+\!2n\!+\!1,s\!+\!1, s\!-\!1]$ .
	But then $\lambda$ degenerates to the partition $\nu'$ that is locally $[s\!+\!2n\!-\!1,s\!+\!3, s\!-\!1]$, which is also special
	and degenerates to $\mu$.  
	Hence the degeneration $(\lambda,\mu)$ is not a minimal special degeneration, which is what we wanted to show.
	
	For type $d$, the pair $(\nu, \mu)$ is locally 
	$$([s\!+\!2n\!+\!1,s\!+\!2n\!+\!1,s], [s\!+\!2n,s\!+\!2n,s\!+\!2] )$$ 
	where $s \not\equiv \epsilon$ . 
	In this case $\nu$ will be non-special exactly when 
	$h_{\nu}( s\!+\!2n\!+\!1) \not \equiv \epsilon'$. 
	If that is the case, 
	then $\lambda$ must locally be $[s\!+\!2n\!+\!2,s\!+\!2n, s]$ from \S \ref{special_pieces}.
	But $\lambda$ also degenerates to the partition $\nu'$ that is locally $[s\!+\!2n\!+\!2,s\!+\!2n{-}2, s\!+\!2]$, which is also special
	and degenerates to $\mu$.  
	Hence the degeneration $(\lambda,\mu)$ is not a minimal special degeneration.
	
	For type $e$, the pair $(\nu, \mu)$ is locally 
	$$([s\!+\!2n,s\!+\!2n,s,s], [s\!+\!2n\!-\!1,s\!+\!2n\!-\!1,(s\!+\!1)^2] )$$ 
	where $s \equiv \epsilon$.
	In this case $\nu$ will be non-special exactly 
	when $h_{\nu}( s\!+\!2n) \not \equiv \epsilon'$ (and $ s \geq 1$ is forced).
	Then $\lambda$ must locally be $[s{+}2n\!+\!1,s\!+\!2n\!-\!1, s\!+\!1, s\!-\!1]$ by \S \ref{special_pieces}.  
	But $\lambda$ degenerates to the partition $\nu'$ that is locally 
	$[s\!+\!2n\!-\!1,s\!+\!2n{-}1, s\!+\!3,s\!-\!1]$,  whenever $n \geq 2$.
	This orbit is special since $\mu$ is.  Moreover, $\nu'$ degenerates to $\mu$,
	so $(\lambda,\mu)$ is not  a minimal special degeneration.  
	
	 This shows, for any minimal special degeneration $(\lambda, \mu)$, which is not already a minimal degeneration, that there exists an intermediate orbit $\nu$ such that $(\nu,\mu)$ is a minimal degeneration of codimension at least $4$,
	 {\bf unless} $\nu$ is of type $e$ with $n=1$. 
	 In Kraft-Procesi's classification \cite[Table I]{Kraft-Procesi:classical}, 
	 the minimal degenerations of dimension at least $4$ are labeled $f,g$ and $h,$ 
	 and are given by the minimal nilpotent orbit closures in types $B$, $C$ and $D$, respectively.
	
	Starting with type $g$, where $n \geq 2$, 
	the pair $(\nu, \mu)$ is locally 
	$$([s\!+\!2,(s\!+\!1)^{2n-2},s], [(s\!+\!1)^{2n}] )$$ 
	with $s \not\equiv \epsilon$. 
	Then $\nu$  is never special since $\mu$, being special, forces 
	$h_\nu(s{+}1) = h_\mu(s{+}1){+}1$  to fail the special condition.
	Then $\lambda$ is forced, locally, to  equal
	$[(s\!+\!2)^{2},(s\!+\!1)^{2n-4},s^2]$ by \S \ref{special_pieces}.
	Because $\mu$ is special, the number of rows $l$ removed is congruent to $\epsilon'$.
	After removing $s$ columns, we see that $(\lambda,\mu)$ has the type of $g_{sp}$, and the latter is indeed a minimal special degeneration, containing only the (non-special) orbit between $\lambda$ and $\mu$.

	For type $f$, 
	the pair $(\nu, \mu)$ is locally 
	$$([(s\!+\!2)^2,(s\!+\!1)^{2n-3},s^2], [(s\!+\!1)^{2n+1}] )$$ 
	with 
	$s \equiv \epsilon$ 
	and $n \geq 2$.
	This is never special since $h_\nu(s+2)$ and $h_\nu(s)$ have different parities, so 
	exactly one of them fails the special condition.
	In the former case, $\lambda$ is locally equal to 
	$[s\!+\!3,(s\!+\!1)^{2n-2},s^2]$ and the degeneration is 
	given by $f^1_{sp}$ in the Table \ref{codim4ormore}.  That this is a minimal
	such degeneration follows since $\nu$ is the only (non-special) orbit between $\lambda$ and $\mu$.	
	In the latter case, which forces $s \geq 1$, then $\lambda$ is locally equal to 
	$[(s\!+\!2)^2,(s\!+\!1)^{2n-2},s\!-\!1]$ and this 
	is the minimal special degeneration $f^2_{sp}$.
	Again, $\nu$ is the only (non-special) orbit between $\lambda$ and $\mu$.
	
	Finally, assume the pair 
	$(\nu, \mu)$ is locally $$([(s\!+\!2)^2,(s\!+\!1)^{2n-4},s^2], [(s\!+\!1)^{2n}] )$$ 
	with $s \equiv \epsilon$ for $n \geq 2$.  
	This is type $e$ if $n=2$ (for $n=1$ in the table for $e$) and type $h$ for $n \geq 3$.
	Observe that the special condition $h_\nu(s+2) \equiv \epsilon'$ 
	is satisfied if and only if $h_\nu(s) \equiv \epsilon'$, 
	since these heights differ by the even number $2n-4$.  
	If both conditions are met, then $\nu$ will be special since $\mu$ is special and 
	this is handled by the minimal degeneration cases.  
	
	Otherwise, both the pairs $(s{+}2,s{+}2)$ and $(s,s)$ in $\nu$ cause $\nu$ to fail to be special,
	which implies $s \geq 1$.  Then $\lambda$ takes the form locally
	$[s\!+\!3,(s\!+\!1)^{2n-2},s\!-\!1]$ by \S \ref{special_pieces}.
	This is the form of $h_{sp}$ in Table \ref{codim4ormore} after removing $s{-}1$ columns.
	This is  a minimal special degeneration containing $3$ (non-special) orbits
	between $\lambda$ and $\mu$:
		$$
	\small{\xymatrix@=.2cm{
			{} & {[s\!+\!3,(s\!+\!1)^{2n-2},s\!-\!1]} \ar@{-}[dl]  \ar@{-}[dr]  &  \\
			{[(s\!+\!2)^2,(s\!+\!1)^{2n-3},s\!-\!1] }\ar@{-}[dr] & {} & {[s\!+\!3,(s\!+\!1)^{2n-3},s^2]}\ar@{-}[dl]  \\
			{} & { [(s\!+\!2)^2,(s\!+\!1)^{2n-4},s^2]}\ar@{-}[d]^h & \\ 
			{} & {[(s\!+\!1)^{2n}]} & }}
	.$$
	The four unlabeled edges all have an $A_1 = C_1$ singularity (type $a$).

	We have therefore shown that every minimal special degeneration is either minimal or takes the form in Table \ref{codim4ormore}.
\end{proof}

Next we will show that each degeneration in Table \ref{codim4ormore} has the given singularity type.

\section{Determining the singularities in Table \ref{codim4ormore}} \label{codim4_degs}

For each type in Table \ref{codim4ormore}, we need to show that the degeneration is as promised.  
The case of type $h$ was done in \cite{Kraft-Procesi:classical}.  We begin with the $g_{sp}$ case.

\subsection{Type $g_{sp}$ case}\label{gsp_case}

As discussed in \S \ref{partitions}, 
for the classical Lie algebras  
$\mathfrak{so}_{2n+1}$, $\mathfrak{sp}_{2n}$, and $\mathfrak{so}_{2n}$, 
the nilpotent orbits under the groups $\rm O(2n+1)$, $\Sp(2n)$, and $\rm O(2n)$
are parametrized by partitions in $\parti_B(2n+1)$, $\parti_C(2n)$, and $\parti_D(2n)$.
This occurs via the Jordan-canonical form of the matrix in the ambient general linear Lie algebra.

Let $e \in \g$ be nilpotent.  Fix an $\mathfrak{sl}_2$-subalgebra $\sg$ through
$e$ and let $\cg(\sg)$ be the centralizer of $\sg$ in $\g$, which is a 
maximal reductive subalgebra of the centralizer of $e$ in $\gg$.   Let $C(\sg)$ be the centralizer in $G$.
Then  $C(\sg)$ is a product of orthogonal and symplectic groups, with each 
part $s$ of $\lambda$ contributing a factor $G^s$, which is isomorphic to  $\mathrm{O}(m(s))$
when $s \not \equiv \epsilon$ and isomorphic to 
$\Sp(m(s))$ when $s  \equiv \epsilon$.  
Denote by $\g^s$ the Lie algebra of $G^s$.
See \cite{C-M} for this background material.


Let $V$ denote the defining representation of $\g$ via the ambient general linear Lie algebra.
If $\lambda$ is the partition corresponding to $e$, then 
under $\sg$, the representation $V$ decomposes as a direct sum 
$$\bigoplus_{s} V(s{-}1)^{\oplus m(s)}$$ over the distinct parts $s$ of $\lambda$.
Here $V(m)$ is the irreducible  $\mathfrak{sl}_2$-representation of highest weight $m$.

Now let $e_0 \in \cg(\sg)$ be  nilpotent.  Then
$e_0 = \sum_s e^{(s)}_0$ for some nilpotent $e^{(s)}_0 \in \gg^s$.
Choose  an $\mathfrak{sl}_2$-subalgebra through $e^{(s)}_0$ in
$\gg^s$ and let $\sg_0$ be the diagonal
$\mathfrak{sl}_2$-subalgebra for
$$e_0 = \sum_s e^{(s)}_0.$$
Each $e^{(s)}_0$ corresponds to a partition $\mu^{(s)}$ of
$m(s)$, using the defining representation of $\gg^s$.

Under the sum $\sg \oplus \sg_0$,
$V$ decomposes as
$$\bigoplus_{s,j} V(s{-}1) \otimes V({\mu^{(s)}_j{-}1})$$
where $s$ runs over the distinct parts of $\lambda$ and $j$ indexes the parts of $\mu_{s}$.

Now consider the diagonal $\mathfrak{sl}_2$-subalgebra for $e+e_0$
in $\sg + \sg_0$. An application of the Clebsch-Gordan
formula immediately gives
\begin{lemma}  \label{lemma:partition_adding}
	The nilpotent element $e+e_0$ in $\gg$ has partition equal to the
	union of the partitions
		\begin{eqnarray*}
		& [s{+}\mu^{(s)}_j{-}1, & s{+}\mu^{(s)}_j {-}3 , \dots, | s -\mu^{(s)}_j|{+}1]
	\end{eqnarray*}
	for each distinct part $s$ in $\lambda$ and each part $\mu^{(s)}_j$ of $\mu^{(s)}$.
\end{lemma}

Suppose that $e_0 \in  \cg(\sg)$ is a  nilpotent element such that each 
$e^{(s)}_0 \in \gg^s$  has partition of the form 
\begin{equation} \label{height_two}
\mu^{(s)} = [2^{a_s}, 1^{b_s}]
\end{equation}
for some
positive integers $a_s$ and $b_s$ with $2a_s + b_s = m(s)$. 
Then the partition $\nu$ of
$e+e_0$ equals the union of the partitions
\begin{eqnarray*}
	[(s+1)^{a_{\tiny i}}, s^{b_i}, (s-1)^{a_i}]
\end{eqnarray*}
for each part $s$ in $\lambda$.  
This follows immediately from the previous lemma since the part $s$
contributes $[s{+}1,s{-}1]$ to $\nu$  when $\mu^{(s)} _j=2$ and it contributes $[s]$ when $\mu^{(s)} _j=1$.

\begin{proposition} \label{codimension_classical}
	Let $e_0 \in \cg(\sg)$ be a nilpotent element satisfying \eqref{height_two}.
	Let $\0$ be the orbit through $e+e_0$.  
	Then the slice ${\mathcal S}_{\0,e}$ 
	is isomorphic to
\begin{equation}
	\prod_s \overline\0_{\mu^{(s)}}
\end{equation}
where the product is over the distinct parts $s$ of $\lambda$.
Here, $\0_{\mu^{(s)}}$ is the orbit with partition $\mu^{(s)}$ in $\mathfrak{sp}(m(s))$
if $s  \equiv \epsilon$ and in $\mathfrak{so}(m(s))$ if $s  \not \equiv \epsilon$.
\end{proposition}

\begin{proof}
	The partition of $e^{(s)}_0$ in
	$\gg$ (rather than $\gg^i$) is equal to
	$$[2^{s \cdot a_s}, 1^{N - 2s \cdot a_s}]$$ where $N$ is the dimension of $V$.
	Setting $a = \sum_ s a_s$, the partition of $e_0$ is
	equal to $$[2^{a}, 1^{N - 2a}], $$ which is of height $2$ in $\g$.  
	Then Corollary 4.9 in \cite{FJLS}
	implies that ${\mathcal S}_{\0,e}$ and $\prod_s \overline\0_{\mu^{s}}$
	have the same dimension.  The latter is isomorphic to 
	$$f + \overline{C(\sg)\cdot (e+e_0)} = f +e+\overline{C(\sg)\cdot e_0},$$
	which is a subvariety of ${\mathcal S}_{\0,e}$.   The result follows from \cite[Cor 13.3]{Kraft-Procesi:classical} 
	if we can show that  $\overline \0$ is normal at $e$.
	
	Since the only minimal degeneration of $[2^{a_s}, 1^{b_s}]$ in $\gg^s$ is $[2^{a_s-1}, 1^{b_s+2}]$ when $a_s>1$
	is of minimal type (that is, types $a$, $f$, $g$, or $h$ in \cite[Table 1]{Kraft-Procesi:classical}),
	the only minimal degenerations of $\0$ that contain $e$ are also of minimal type.  The argument 
	in \cite[Thm 16.2]{Kraft-Procesi:classical} then shows that $\overline \0$ is normal at $e$.
	\end{proof}

We can now prove the $g_{sp}$ case. 
\begin{corollary}
	Let $\0 = \0_\lambda$ and $e \in \0_\mu$, where $(\lambda,\mu)$ are of type $g_{sp}$ in Table \ref{codim4ormore}.
	Then ${\mathcal S}_{\0,e}$ 
	is isomorphic to 
	$\overline{\0}_{[2^2,1^{2n-4} ]}$, the closure of the minimal special orbit in $\mathfrak{sp}_{2n}$.
\end{corollary}

\begin{proof}
If $k$ is the number of columns removed, then $s=k+1$ is the relevant part of $\mu$ and $m(s) = 2n$. 
We take $e_0$ to be exactly equal to $e^{(s)}_0 \in \gg^s$  with partition $\mu^s = [2^2,1^{2n-4}]$.  
Then the result follows from the proposition.  Since $k \not \equiv \epsilon$ for the $g_{sp}$ case,
we have $s \equiv \epsilon$ and  $\gg^s = \mathfrak{sp}_{2n}$.
\end{proof}

\begin{remark}
The $h$ case in the table, already done in \cite{Kraft-Procesi:classical}, corresponds to the situation where $s \not \equiv \epsilon$  and $m(s)$ is even, as well as having $\0_\lambda$ is special (so the condition on $l$ holds).   Of course, the slice is still of type $h$ even if the degeneration is not a minimal special one.
\end{remark}

\subsection{Type $f_{sp}^{1}$ and $f_{sp}^{2}$}

These cases occur when $(\lambda, \mu)$ is $\epsilon$-row equivalent to 
$(\lambda^{(i)}, [a^{N}])$ where $N$ is odd and $a \not\equiv \epsilon$ and
\begin{eqnarray*}\label{eqn:lambda1_2}
\lambda^{(1)}= [a{+}2, a^{N-3},(a{-}1)^2] & \text{         (type } f_{sp}^{1}) \\
\lambda^{(2)}=[(a{+}1)^2, a^{N-3},a{-}2]  & \text{    (type } f_{sp}^{2}) 
\end{eqnarray*}
with $a\geq 1$ when $i=1$ and $a \geq 2$ when $i=2$.
By Proposition \ref{slice_row_equiv}, it is enough to assume that $\mu = [a^N ]$ and $\lambda$ equals $\lambda^{(1)}$  or $\lambda^{(2)}$.
Since we will need it for the next section,we consider the more general case $N$ is any integer with $N \geq 3$.
\begin{proposition}\label{f_cases}
	Let $\g = \mathfrak{so}_{aN}$ if $a$ is odd and $\mathfrak{sp}_{aN}$ if $a$ is even.
	 Let $e \in \0_\mu$.  
	For $i \in \{1,2\}$, let $\0 = \0_{\lambda^{(i)}}$.
	Then there is an isomorphism
	$${\mathcal S}_{\0, e} \simeq 	\overline{\0}_{[3,1^{N-3}]}$$
	where the orbit closure is in $\mathfrak{so}_{N}$.  
\end{proposition}

\begin{proof}
If $a=1$, the $\lambda^{(1)}$ case is clear since there is equality between the slice and the orbit closure, 
so assume $a \geq 2$.
The situation is very similar to \S 11.2 in \cite{FJLS}.
Let $I_N$ be the $N \times N$ identity matrix.
Let us define the form on $V$ explicitly to be the one defined by the 
$a \times a$ block anti-diagonal matrix $J$ with
$$J = \begin{pmatrix} 
	0 & 0 &  \dots &0 & 0 & I_N  \\
   0 & 0 &  \dots& 0  &-I_N & 0  \\
 0 &  0 & \dots & I_N &0 & 0 \\
\dots \\
(-1)^{a-1} I_N & 0 &  \dots & 0  & 0 & 0 
\end{pmatrix}.$$
The bilinear form defined by $J$ is nondegenerate and 
 is symmetric if $a$ is odd and symplectic if $a$ is even.
Since $a \not\equiv \epsilon$, this is the correct form for defining $\g= \g(V_\epsilon)$.  

The $a \times a$-block-matrices $e$ and $f$ given by 
\begin{eqnarray}
	e = \begin{pmatrix} 
		0 & 0 & \dots& 0 & 0  \\
		c_1 I_N &  0 &\dots & 0  & 0  \\
		0 &c_2  I_N&  \dots & 0 & 0 \\
		&&\dots &&\\
		0 & 0&   \dots & c_{N-1} I_N & 0
\end{pmatrix},  &
	f = \begin{pmatrix} 
		0 & 0 &   \dots& 0 & 0  \\
		I_N & 0 &\dots & 0  & 0  \\
		0 & I_N&  \dots & 0 & 0 \\
		&&\dots &&\\
		0 & 0&   \dots & I_N & 0
	\end{pmatrix}^T
\end{eqnarray}
with $c_j = j(N-j)$
lie in $\g$ and $e$ and $f$ are both nilpotent with partition $\mu$.  They complete to an $\sl_2$-triple as in \cite[\S 11.2]{FJLS}.
The centralizer $\gg^f$ is the set of block upper triangular matrices of the form 
$$X = \begin{pmatrix} 
	Y_1 & Y_2 & Y_3 &...\text{   } Y_{a-2}& Y_{a-1} & Y_{a}  \\
	0 &Y_1 & Y_2 &\dots & Y_{a-2} & Y_{a-1} \\
	0 & 0 &  Y_1 &\dots & Y_{a-3} & Y_{a-2}  \\
	&&&\dots &&\\
		0 & 0&   0 &\dots & Y_2  & Y_3 \\
	0 & 0&   0 &\dots & Y_1 & Y_2 \\
	0 & 0&   0 &\dots& 0& Y_1
\end{pmatrix}.$$
Then $X$ lies in $\g$ if and only if $Y_i = (-1)^i Y^T_i$.  Let $\Sigma_N$ denote the set of $N \times N$ symmetric matrices and set
$$\phi: \mathfrak{so}_N \times \Sigma_N \times \mathfrak{so}_N \times \dots  \to \g^f$$ denote the map
where $\phi(Y_1, Y_2, \dots)$ is given by the matrix $X$ above.
The reductive centralizer $\cg(\sg) \simeq \mathfrak{so}_N$ is given by $Y_i=0$ for $i  \geq 2$;
similarly $C(\sg)\simeq \rm O(N)$.
An element $g \in C(\sg)$ acts on $e+X \in \mathcal S_e$ by sending $Y_i$ to $gY_ig^T$. The $\C^*$-action on $\mathcal S_e$
is given by $t.Y_i = t^{2i}Y_i$ for $t \in \C^*$.

Let $e_0 \in \cg(\sg)$ be a nilpotent element with partition $[3,1^{N-3}]$.  
Pick an $\mathfrak{sl}_2$-triple $\sg_0$ through $e_0$ in $\cg(\sg)$ and assume that the semisimple element $h_0$ is a diagonal matrix.
By Lemma \ref{lemma:partition_adding} or computing $h+h_0$ directly, we see that
$e+e_0$ has partition $\nu:=[a{+}2, a^{N-2},a{-}2]$ since $a \geq 2$.  

Let $N=3$.  By \cite{Kraft-Procesi:classical} we have
${\mathcal S}_{\0_{\lambda^{(i)}},e} \simeq 	\overline{\0}_{[3]}$ since for $i{=}1$ the degeneration 
is type $c$ and for $i{=}2$ it is type $d$ in Table \ref{codim2}.   
Set $A = e_0$, then $A^2 \in \Sigma_N$.
Then $\phi(0,A^2, 0, \dots)$ is an eigenvector for both $\ad(h)$ and $\ad(h_0)$, with eigenvalue $-2$ and $4$, respectively.
Since the absolute values of the eigenvalues of $\ad(h_0)$ on $\zg(e)$ are at most $4$ and the eigenvalue $4$ only occurs once in $\Sigma_N$,
there are no other exceptional pairs in the sense of \cite[\S 4]{FJLS}.
It follows that $e+\phi(A, z_i A^2, 0, \dots,0) \in \0_{\lambda^{(i)}}$ for a unique $z_i \in \C^*$.
Since ${\mathcal S}_{\0_{\lambda^{(i)}},e}$ has dimension two and is irreducible, this
means it is exactly the set of elements $e+\phi(A, z_i A^2, 0, \dots,0)$ where $A \in  \mathfrak{so}_3$ is nilpotent, 
giving the isomorphism to $\overline{\0}_{[3]}$ explicitly.

Now consider the general $N$ case.  
We can embed $\mathfrak{so}_3$ into $\cg(\sg) \simeq \mathfrak{so}_N$  via the first $3$ coordinates 
and similarly for the rest of the centralizer of $e$ in  the $N=3$ case.  
Clearly, for $A \in \0_{[3]}$, then $\phi(A, 0 \dots, 0) \in \cg(\sg)$ lies in $\0_{[3,1^{N-3}]}$,
but also $e + \phi(A, z_iA^2, 0, \dots, 0) \in {\mathcal S}_e$  lies in $\0_{\lambda^{(i)}}$, by observing
the action of this element on the standard basis of $V$.  It follows, using the action of $C(\sg)$, that  
${\mathcal S}_{\0_{\lambda^{(i)}},e}$ contains $e + A+z_iA^2$ for $A \in \overline{\0}_{[3,1^{N-3}]}$.

Next, $\dim \0_{\lambda_1} = \dim \0_{\lambda_2}$ since both orbits are minimal degenerations from  
$\0_\nu$ of type $a$, hence they are codimension two in $\overline{{\mathcal \0_\nu}}$.
The pair $(\lambda_1, \mu)$ is equivalent to $([3,1^{N-3}], [1^N])$ after canceling $a{-}1$ columns, 
thus the codimension of $\0_\mu$ in $\overline{\0_{\lambda^{(i)}}}$  equals the dimension of $\overline{\0}_{[3,1^{N-3}]}$
for both $i=1$ and $i=2$.
The only minimal degeneration from $\0_{\lambda^{(i)}}$ that contains $\0_\mu$ is to
the partition $[(a{+}1)^2, a^{N-4},(a{-}1)^2]$, which an $A_1$ singularity for both $i=1$ and $i=2$.  
Hence, as in \S \ref{gsp_case},
we have $\overline{\0_{\lambda^{(i)}}}$  is unibranch at $e$.  Thus ${\mathcal S}_{\0_{\lambda^{(i)}},e} \simeq  \overline{\0}_{[3,1^{N-3}]}$.
\end{proof}


\subsection{Type $h_{sp}$}\label{Subsect:h_sp}  

As in the previous subsection, we are reduced by Proposition \ref{row_equivalent} 
to the case where $\lambda = [a{+}2, a^{N-2},a{-}2]$ and $\mu =[a^N]$.  We have 
the same description for $e \in \0_\mu$, $\sg$, etc. as above.
In the previous subsection, ${\mathcal S}_{\0_{\lambda^{(i)}},e}$ was the closure of a $C(\sg)$-orbit.
We will first show that 
${\mathcal S}_{\0_{\lambda},e}$ is the closure of a $C(\sg) \times \C^*$-orbit for $N \geq 2$, where
the $\C^*$-action is as above.

The arguments from \cite[\S 11.2.2]{FJLS} apply.  First, we use them to show 
that for $M \in {\mathcal S}_{\0_{\lambda},e}$, the matrices $Y_3, Y_4, \dots$ are equal to sums of products of $Y_1$ and $Y_2$.  
This follows since $\mathrm{rank}(M^i) \leq N(a-i)$ for $i = 1, \dots, a-2$ as in loc. cit.
However, $\mathrm{rank}(M^{a-1}) \leq N+1$ unlike in loc. cit. and this implies that 
\begin{equation}\label{eqn:rank}
\mathrm{rank}(Y_2- dY_1^2)\leq 1
\end{equation} 
for some $d \in \C^*$.
The condition $M^{a+2}=0$ yields the equation,  in the block lower left corner, 
\begin{equation} \label{firstY2}
	d_1Y_3+d_2(Y_1Y_2+Y_2Y_1)+d_3Y_1^3=0
\end{equation}
where $d_1 = \frac{(a+2)!(a-1)!}{6}$ and $d_3 = \frac{(a-1)(a-2)}{5} \cdot d_1$.
It follows that the $C(\sg) \times \C^*$-equivariant map from ${\mathcal S}_{\0_{\lambda},e}$ to 
the space $e+X$ where $Y_i=0$ for $i \geq 3$ is an isomorphism.    

Now for $N \geq 3$,  there are actually two equations involving a linear term for $Y_3$.  
The one from the lower left corner of  $M^{a+2}=0$ and the one from the $\mathrm{rank}(M^{a-2}) = N$ condition:
\begin{equation} \label{secondY2}
c_1Y_3+c_2(Y_1Y_2+Y_2Y_1)+c_3Y_1^3=0.
\end{equation}
where $c_1 = \frac{a!(a-1)!^2(a-2)!^2(a-3)!}{24}$ and  $c_3= \frac{(a+1)(a+2)}{5} \cdot c_1$.
The equations \eqref{firstY2} and \eqref{secondY2} are not multiples of each other since 
$\frac{d_3}{d_1}  \neq \frac{c_3}{c_1}$ for $a>0$.
It follows, by canceling the $Y_3$ term, that
\begin{equation}\label{eqn:y03}
tY_1^3 = Y_1Y_2+Y_2Y_1
\end{equation}
for some nonzero $t$.

Consider the $N=2$ case.   Conjugating in $GL_2$ so that $C(\sg)$ becomes the diagonal torus in $SL_2$, we can represent 
$Y_1 = \left( \begin{smallmatrix} 
	x & 0  \\
	0 & -x
\end{smallmatrix}\right)$ and 
$Y_2 =  \left(\begin{smallmatrix} 
	y& z  \\
	w & y\end{smallmatrix}\right)$
for $x,y,z,w \in \C$.
Then \eqref{eqn:y03} implies that either $x$ is identically $0$, i.e., $Y_1\equiv0$ or $y=\frac{t}{2} x^2$, i.e., $\tr(Y_2-tY_1^2)=0$.
By \eqref{eqn:rank}, $\det(Y_2-dY_1^2)=0$. 
If $Y_1 \equiv 0$, then $\det(Y_2)=y^2-zw=0$ and $x=0$ are the conditions defining the slice.  
If $t=d$, the condition is $zw=0$, with $x$ arbitrary.  
Since $([a{+}2, a{-}2], [a^2])$ is a minimal degeneration of type $b$, the slice is isomorphic to the $A_3$-simple surface singularity,
so neither of these cases hold.  Instead, $d \neq t$ and the defining equations are $y=  \frac{t}{2} x^2$
and $t^2x^4  = 4wz$, which is indeed an $A_3$-singularity.
Moreover, the points with $x \neq 0$ form a single $C(\sg)\times \C^*$-orbit, each of which has finite stabilizer, so this orbit
is dense in the slice.   Let $v=e+\phi(v_1, v_2, \dots)$ be such a point in the slice.

Next, as in \S\ref{f_cases},  we bootstrap up to the general case by embedding the slice for the $N=2$ case into the general slice by using the first two coordinates in each block.  Since the coefficients in the equations given above continue to hold for
${\mathcal S}_{\0_{\lambda},e}$  in ${\mathcal S}_e$, independent of $N$, we do indeed have a $\SO_2\times \C^*$-equivariant 
embedding of the $N=2$ case.  
Note that $v_1 \in \so_N$ is a multiple of a $C(\sg)$-conjugate of $h_0 $ from \S \ref{f_cases}. Is stabilizer in $C(\sg)$ is 
$\SO_2 \times \SO_{N-2} \subset C(\sg)$.  From the $N=2$ case it follows that the connected stabilizer in $C(\sg)\times\C^*$of $v$ is
$\SO_{N-2}$.
Hence, the $C(\sg)\times \C^*$-orbit through $v$ has dimension $(N(N-1)/2 + 1 )- (N-2)(N-3)/2=2N-2$.  
This is also the dimension ${\mathcal S}_{\0_{\lambda},e}$. 
Since $\overline{\0}_\lambda$ is unibranch at $e$, we conclude

\begin{proposition}\label{prop:dense_shared_orbit}
	The slice ${\mathcal S}_{\0_{\lambda},e}$ is isomorphic to the closure of a $C(\sg)\times \C^*$-orbit through 
	$(A,B) \in \so_N \times \Sigma_N$ 
	where $A = h_0 \in \so_N$, and $B\in \Sigma_N$ satisfies $\mathrm{rank}(B-dA^2) = 1$
	and $\tr(B-tA^2)=0$ for some nonzero $d,t$ with $d \neq t$.
\end{proposition}

Our goal now is to identify the subvariety of $\so_N \times \Sigma_N$ in the proposition with 
the quotient of closure of the orbit $\0_{[3,1^{N-2}]}$ in $\so_{N+1}$.
First, by employing the 
$C(\sg)\times \C^*$-equivariant isomorphism of $\so_N \times \Sigma_N$  to itself sending $(A,B)  \to (A,B-dA^2)$,
it follows that the slice is isomorphic
to the closure of the $C(\sg)\times \C^*$-orbit through $(A,B) \in \so_N \times \Sigma_N$  
where $B$ is now a matrix of rank $1$ and $\tr(B-tA^2)=0$ for some nonzero $t$.

Next, we recall material from \cite[\S 2]{Fu-Juteau-Levy-Sommers:GeomSP}.
Let $X \in \so_{N+1}$ be written as $$\begin{pmatrix} M & u \\ -u^T & 0 \end{pmatrix}$$
where $M \in \mathfrak{so}_N$ and $u\in \C^N$ is a column vector.
Let $\theta$ be the involution of $\so_{N+1}$ given by conjugation of ${\rm diag}(1,1,\ldots ,1,-1) \in {\rm O}_{N+1}$.
Identifying $\so_{N+1}$ with the set of pairs $(M,u)$, we see that $\theta$ maps $(M,u)\mapsto (M,-u)$.
Then the map $\varphi:  \so_{N+1} \to \so_N \times \Sigma_N$ sending $X$ to $(M,uu^T)$ 
induces an ${\rm O}_n \times \C^*$-equivariant isomorphism of $\mathfrak{so}_{n+1}/\langle \theta \rangle$ 
with $\mathfrak{so}_n\times \Xi$ where $\Xi$ is the cone of elements of $\Sigma_N$ of rank at most 1.

We can now state 
\begin{proposition}
	The slice ${\mathcal S}_{\0_{\lambda},e}$ is isomorphic to
	$$\overline{\0}_{[3,1^{N-2}]}/\langle \theta \rangle.$$
\end{proposition}

\begin{proof}
By  \cite[Corollary  2,.2]{Fu-Juteau-Levy-Sommers:GeomSP}, if $Y = \overline{\0}_{[3,1^{N-2}]}$,
then $Y/\langle \theta \rangle \simeq \varphi(Y)$.
As before we can use the $N=2$ case.  
In that case, $X = \begin{pmatrix} 0 & a & b \\ -a & 0 & c \\ -b & 0 & -c \end{pmatrix}$
and 
$$\varphi(X) = \left( \left( \begin{smallmatrix} 
	0 & a  \\
	{-}a &  0 
\end{smallmatrix}\right), 
 \left(\begin{smallmatrix} 
	b^2 & bc  \\
	bc & c^2 \end{smallmatrix}\right)  \right).$$	
The condition for $X$ to be nilpotent is $a^2+b^2+c^2 =0$ and so the image 
is exactly the matrices $(A,B)$ where $\det(B)=0$ and $\tr(B + A^2)=0$.   In the general case, 
we embed $\so_3$ into the lower right corner.  It follows 
from the discussion above and the proof of Proposition \ref{prop:dense_shared_orbit}
that $\varphi(Y)$ is isomorphic to the closure of the $C(\sg)\times \C^*$-orbit through $(A,B)$ with $A \neq 0$,
and hence isomorphic to the slice by Proposition \ref{prop:dense_shared_orbit}.
\end{proof}


Let the Klein four-group $V_4$ act on $\mathfrak{so}_{N+2}$ via the pair of commuting involutions $\theta_1, \theta_2$ 
given by conjugation by $\diag(1,\ldots ,1,-1)$ and $\diag(1,\ldots ,1,-1,1)$, respectively.
Let $\overline{{\mathcal O}}_{[2^2,1^{N-2}]}$ be the minimal orbit in $\mathfrak{so}_{N+2}$ .
Then by \cite[Corollary 2.5]{Fu-Juteau-Levy-Sommers:GeomSP}, for example, it follows that
$$\overline{{\mathcal O}}_{[3,1^{N-2}]} \simeq \overline{{\mathcal O}}_{[2^2,1^{N-2}]}/\langle \theta_1 \rangle.$$

\begin{corollary}\label{Cor:h_sp_sing}
	We have the isomorphism
$${\mathcal S}_{\0_{\lambda},e} \simeq \overline{{\mathcal O}}_{[2^2,1^{N-2}]}/\langle \theta_1, \theta_2 \rangle.$$	
and hence the minimal special degeneration $h_{sp}$ in Table \ref{codim4ormore} is $d_{n+1}/V_4$.
\end{corollary}

\begin{proof}
The $h_{sp}$ degeneration is covered by the case when $N$ is even with $n=N/2$.
\end{proof}

\begin{remark}
	a) The special case $n=3$ was already observed in \cite{Fu:wreath}.
	In that case we have $\overline{\0_{\rm{ min}}(\mathfrak{so}_5)}\cong {\mathbb C}^4/\{\pm 1\}$ and so we obtain isomorphisms of $\SOe\cap\overline{{\mathcal O}'}$ with (i) ${\mathbb C}^4/W(B_2)$; (ii) ${\mathcal N}(\mathfrak{so}_4)/\mathfrak{S}_2=({\mathcal N}(\mathfrak{sl}_2)\times{\mathcal N}(\mathfrak{sl}_2))/\langle\theta\rangle$ where $\theta$ swaps the two copies of $\mathfrak{sl}_2$.
	
	b) The orbits which intersect non-trivially with 
	$\SOe$ are the nilpotent orbits lying between $\0$ and $\0'$ in the partial order.
   If $N\geq 4$ then there are five of these, as indicated by the following diagram:
$$
	\small
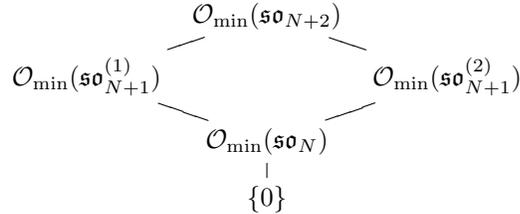
{\xymatrix@=.2cm{
			 &&&  \0_{\rm{ min}}(\mathfrak{so}_{N+2})\ar@{-}[dl] \ar@{-}[dr] & \\
			&& \0_{\rm{ min}}(\mathfrak{so}_{N+1}^{(1)})\ar@{-}[dr] & & \0_{\rm{ min}}(\mathfrak{so}_{N+1}^{(2)})\ar@{-}[dl] \\
			&&& \0_{\rm{ min}}(\mathfrak{so}_N)\ar@{-}[d] & \\
			 &&& \{ 0\} & }}
	$$
	where $\mathfrak{so}_{N+1}^{(1)}$ and $\mathfrak{so}_{N+1}^{(2)}$ are the two fixed point subalgebras for the two proper parabolic subgroups of $\mathfrak{S}_2\mathfrak{S}_2$.
	For $N=3$ there is no orbit with partition $[3^2,2^{N-4},1^2]$, equivalently, $\mathfrak{so}_3$ contains no elements of $\0_{\rm{ min}}(\mathfrak{so}_5)$.

	a)  For $N$ odd, the singularity $\overline{{\mathcal O}}_{[2^2,1^{N-2}]}/\langle \theta_1, \theta_2 \rangle$ arises
as a slice, but never for a minimal special degeneration.  This is because the $f_{sp}$ singularities arise in this case as the minimal special degenerations.
\end{remark}

\subsection{Minimal special degenerations in the exceptional Lie algebras} \label{Section:min_special_deg}

There are three unexpected singularities that arise in the exceptional Lie algebra: (i) $\mu$ (with normalization $A_3$); (ii) $a_2/\mathfrak{S}_2$; (iii) $d_4/\mathfrak{S}_4$, which are dealt with in \cite{FJLS}, \cite{Fu-Juteau-Levy-Sommers:GeomSP}.  They appear once, once, and twice, respectively.
We will show in this subsection that all remaining singularities associated to minimal special degenerations in exceptional types are unions of simple surface singularities or minimal special singularities.

The case of $G_2$ is clear. 
Most of the minimal special degenerations are minimal degenerations and hence were dealt with in \cite{FJLS} 
or \cite{Fu-Juteau-Levy-Sommers:GeomSP}.
There are three (resp. three, eight, ten) minimal special degenerations which are not minimal degenerations in type $F_4$ (resp. $E_6$, $E_7$, $E_8$).
These cases, with two exceptions, are covered by the following proposition.

\begin{proposition}\label{nonminlem}
Let $\0'$ be a special nilpotent orbit in an exceptional Lie algebra such
that the reductive centralizer $\cg(\sg)$ contains a non-simply-laced simple component $\cg_0={\rm Lie}(C_0)$.

(a) There is a unique special orbit $\0>\0'$ such that $\codim_{\overline\0}\0'$ is equal to the dimension of the minimal special nilpotent $C_0$-orbit $\0_0$ in $\cg_0$.

(b) If $\0'=\0_{2A_2}$ in type $E_8$ then there are two such simple components $\cg_0$, both of type $G_2$, and $\SOe$ is a union of two copies of $\overline{\0_0}$.  The two copies are interchanged by $C(\sg)$.
Other than this case, there is exactly one such $\cg_0$ and $\SOe \simeq \overline{\0_0}$.
\end{proposition}

\begin{proof}
Statement (a) is a straightforward check using the tables of nilpotent orbits and Hasse diagrams in \cite{Carter}.

The singularities in (b) can be classified using the arguments in \cite[\S 4.3]{FJLS}.
Indeed, several of these are discussed there, see \cite[\S 11, Table 13]{FJLS}.
Let $e_0 \in  \0_0$.
We claim that, with the sole exception of $\0'=\0_{A_2+3A_1}$ in type $E_7$, $e+e_0 \in \0$.
By unibranchness and dimensions, it follows that 
$\SOe=\overline{f+C_0\cdot e_0}\cong\overline{\0_0}$.
By \cite[Prop. 4.8]{FJLS}, it suffices to verify the following condition: let $\langle h_0,e_0,f_0\rangle=\sg_0\subset\cg_0$ be an $\mathfrak{sl}_2$-subalgebra, then all irreducible $\sg_0$-summands in $\g^e(i)$ have dimension $\leq (i+1)$.
This can be checked by inspecting the tables in \cite{Lawther-Testerman}.
If $\cg_0$ is of type $B$, then all non-trivial simple summands for the action on the centralizer of $e$ are natural modules or spin modules; a short root element acts with Jordan blocks of size 2 on the spin module and of size $\leq 3$ on the natural module, so we only need to check that no natural modules occur in $\g^f(1)$.
When $\cg_0$ is of type $G_2$ (excluding $\0'=\0_{A_2+3A_1}$ in type $E_7$), all non-trivial summands are isomorphic to the minimal faithful representation for $\cg_0$; $e_0$ acts on the minimal faithful representation with Jordan blocks of size $\leq 3$, so we only need to check that the minimal representation doesn't appear in $\g^e(1)$.
Finally, $\cg_0$ of type $C$ occurs once, when $\0'=\0_{D_4}$ in type $E_7$ and $\cg_0=\mathfrak{sp}_6$; here one has to check that $e_0$ has no Jordan blocks of size $>7$ on $V(\varpi_2)$, hence on the alternating square of the natural module, which is straightforward.

This only leaves the case $\0'=\0_{A_2+3A_1}$ in $E_7$.
Here $\cg_0=\g^e\cap\g^h$ is simple of type $G_2$ and the positive graded parts of $\g^f$ are: $$\g^f(2)=V(2\varpi_1)\oplus\C e,\quad \g^f(4)=V_{\mini},$$
where $V_{\rm min}=V(\varpi_1)$ is the minimal faithful representation for $\cg_0$.
Note that the action of $\cg_0$ on $V_{\mini}$ induces an embedding in $\mathfrak{so}_7$, and $\mathfrak{sl}_7$ decomposes over $\cg_0\subset\mathfrak{so}_7$ as $\mathfrak{so}_7\oplus V(2\varpi_1)$.
(In the notation of \S \ref{f_cases}, $V(2\varpi_1)=\Sigma_7$.)
Furthermore, the matrix square operation on $\mathfrak{gl}_7$ determines a quadratic map $\mathfrak{so}_7\rightarrow V(2\varpi_1)$ which restricts to $\cg_0$ to give a $C_0$-equivariant map $\psi:\cg_0\rightarrow V(2\varpi_1)\subset\g^e(2)$.
In particular, if $x=e_{\beta_2}+e_{3\beta_1+\beta_2}$ then $\psi(x)$ is non-zero of weight $3\beta_1+2\beta_2$.
We checked using GAP that (with this notation) there exists an element of $\0=\0_{D_4(a_1)}$ of the form:
$e+x+\psi(x)$, where $x$ is in the subregular nilpotent orbit in $\cg_0$.
It follows that $\SOe=e+\overline{C_0\cdot (x+\psi(x))}$, hence is isomorphic to the closure of the $C_0$-orbit through $x$, which completes our proof.
\end{proof}

\begin{proposition}
All minimal special degenerations in exceptional Lie algebras are either: (1) minimal degenerations, (2) covered by the above proposition, or (3) isomorphic to $d_4/\mathfrak{S}_4$, which occurs for the two cases of $\0_{F_4(a_3)}>\0_{A_2}$ in $F_4$ and $\0_{E_8(a_7)}>\0_{D_4+A_2}$ in $E_8$.
\end{proposition}
\begin{proof}
We checked that all the minimal special degenerations, which are not minimal degenerations, are covered by
the proposition, except for the two cases listed.
The exceptional special degeneration $\0_{F_4(a_3)}>\0_{A_2}$ in type $F_4$ was dealt with in 
\cite[Theorem 4.11]{Fu-Juteau-Levy-Sommers:GeomSP}.
Hence it remains to show that the Slodowy slice singularity from $D_4+A_2$ to $E_8(a_7)$ in $E_8$ is also isomorphic to $d_4/\mathfrak{S}_4$.
To do this, we first consider $\0'=\0_{D_4}$.
Let ${\mathcal S}_{D_4}$ be the Slodowy slice at $f\in\0'$.
Repeating the calculation in the proof of Proposition \ref{nonminlem}, we see that the condition of \cite[Prop. 4.8]{FJLS} holds for an element of $\cg$ of type $F_4(a_3)$.
It follows that ${\mathcal S}_{D_4}\cap \overline{\0_{E_8(a_7)}}=f+\overline{C_0\cdot e_2}$ where $e_2$ belongs to the $F_4(a_3)$ orbit in $\cg_0$.
By the same calculation (or by direct observation), ${\mathcal S}_{D_4}\cap\overline{\0_{D_4+A_2}}=f+\overline{C_0\cdot e_1}$ where $e_1$ is in the $A_2$ orbit in $\cg_0$.
Now we use the following fact (which follows from equality of dimensions): if $\{ e_1,h_1,f_1\}$ is an $\mathfrak{sl}_2$-triple in $\cg_0$ such that $\dim C_0\cdot e_1$ equals the codimension of $\0'$ in $\overline{G\cdot (f+f_1)}$, then the centralizer of $e+e_1$ equals $\g^e\cap\g^{e_1}$.
Hence the Slodowy slice at $f+f_1$ is contained in the Slodowy slice at $f$.
It follows that ${\mathcal S}_{D_4+A_2}\cap\overline{\0_{E_8(a_7)}}$ is isomorphic to the Slodowy slice singularity in $F_4$ from $\0_{A_2}$ to $\0_{F_4(a_3)}$, hence is isomorphic to $d_4/\mathfrak{S}_4$. 
\end{proof}

The following is true in both the classical and exceptional types.
\begin{corollary}\label{Corollary:simple_factors_control_singularity}
	Let $\0'=\0'_e$ be special. 	
	The action of $C(\sg)$ on $\cg(\sg)$ induces an action of $A(e)$ on the set of simple components of $\cg(\sg)$.  
	Each $A(e)$-orbit of simple components $\cg_0$ corresponds to a unique special nilpotent orbit $\0$ in $\g$ such that $(\0,\0')$ is a minimal special degeneration. Moreover, $\SOe$ contains a subvariety isomorphic to the minimal special nilpotent orbit closure in $\cg_0$. All minimal special degenerations of codimension at least $4$ arise in this way
\end{corollary}


\begin{proof}
We just showed this in the exceptional types when $\cg_0$ is not simply-laced, but it also holds when $\cg_0$ is simply-laced where it gives a minimal degeneration.  It also holds in the cases of $d_4/\SG_4$ and $a_2/\SG_2$ 
from \cite{FJLS}.   In the classical types, we showed that each simple factor of $\cg(\sg)$ leads to a unique minimal special degeneration.  The $A(e)$-orbits on the simple factors 
of $\cg(\sg)$ are singletons except for the case where $\cg(\sg)$ contains a copy of $\so_4$.  This corresponds to the case of $[2A_1]^+ = d_2^+$.
\end{proof}


\section{$A(e)$-action on slices}\label{section:A_action}

In this section we compute the action of $A(e)$ on the slice $\SOe$ for both minimal degenerations and minimal special degenerations in the classical types, and determine when the action is outer. 
This was done in the exceptional groups in \cite{FJLS} for minimal degenerations.
There is  only a single case of a minimal special degeneration not covered by those results:  the case of $e \in A_2$ from Proposition \ref{nonminlem}, which we now denote as $[2g_2^{sp}]^+$.

\subsection{Union of simple surface singularities}\label{subsect:defn_A_action}
Recall that $C(\sg)$ acts on $\SOe$.   In the case of a simple surface singularity, as discussed in the introduction, we use Slodowy's notion of action, which amounts to the action on the projective lines in the exceptional fiber.
Even when $\SOe$ is not irreducible, we want to describe how 
$C(\sg)$ permutes the projective lines in the fiber, something we did in the exceptional groups.   Since $C^\circ(\sg)$ acts trivially, we get a permutation action of $A(e) \simeq C(\sg)/C^\circ(\sg)$ on the $\mathbb P^1$'s.  We call this the outer action of $A(e)$ on the slice.

To compute the action for $\dim(\SOe)=2$, 
we use \cite[Lemma 5.8]{FJLS}.   
We do not assume that the orbits are special, so the set-up is a minimal degeneration 
$(\OC_{\lambda}, \OC_{\mu})$ in the classical groups where $\dim(\SOe)=2$ for $e \in \0_\mu$, 
and where $\lambda, \mu$ are the appropriate partitions indexing the nilpotent orbits.  
Let $\nilrad_P$ denote the nilradical of the Lie algebra of  a parabolic subgroup $P$ of $G$ such that $\0_\lambda$ is Richardson for $\nilrad_P$.   Then we have the proper, surjective map 
$\pi:  G \times^P \ \nilrad_P \to \overline{\OC}_{\lambda}$, which is generically finite.   Below, we will always choose $\nilrad_P$ so that $\pi$ is birational. 

Next, assume that the reductive centralizer for an element in $\0_\lambda$ is semisimple.  
Let $\0_1, \0_2, ...\0_t$ be the maximal orbits in the complement of  $\0_\lambda$ in its closure.  
Assume that all $\0_i$ are codimension two in $\bar{\OC}_{\lambda}$.
Let $e_i \in \0_i$.  Let $r_i$ equal the number of $A(e_i)$-orbits on $\pi^{-1}(e_i)$.  Then as in \cite[Lemma 5.8]{FJLS}, if $G$ is connected, we have $\sum_i r_i$ is equal to rank of $\g$ minus the rank of the Levi subgroup of $P$.   The quantities $r_i$ will be enough to determine the outer action. 

Remarkably, in types $B$ and $C$,  the actions are large as possible 
as they were in the exceptional types (at least given the size of $A(e)$).

\begin{proposition} \label{prop:A_action_codim2}
In the classical groups $B,C,D$ (working in the full orthogonal group for $D$), 
\begin{enumerate}
\item If $\SOe$ is a simple surface singularity of type $D_{k+1}$
or  $A_{2k-1}$, then the $A(e)$-action upgrades these singularities to $C_k$ and $B_k$, respectively.
\item If $\SOe$ is a union of two branches of type $A_{2k-1}$, the $A(e)$-action 
is $[2B_k]^+$ as described in \S \ref{subsect:min_sp_intro}.
\end{enumerate}
\end{proposition}

The proof will occupy the remainder of this section.
For the moment let $G = \rm O(V)$ or $\rm Sp(V)$,
so that, as noted in \S \ref{codim4_degs}, a reductive subgroup of the centralizer $G^e$ of $e$ in $G$ is
$C(\sg)$, which is a product of orthogonal and symplectic groups.  

Then the component group $A(e):=G^e/(G^e)^\circ$ of $e$ with partition $\mu$ 
is generated by the corners of Young diagram corresponding to parts $s$ with $s \not \equiv \epsilon$.
Each such part $s$ determines a copy of an orthogonal group in $C(\sg)$
and we denote by $x_s$ an element of determinant $-1$ in each orthogonal group. 
Then $A(e)$ is elementary abelian ${\mathbf Z}_2^r$ where $r$ is the number of parts 
$s$ with $s \not \equiv \epsilon$.

\subsection{Type $b$ degeneration}\label{b_sing}

This is the case of a simple surface singularity of type $D_{k+1}$ and it arises whenever
$(\lambda, \mu)$ is locally $(\lambda', \mu'):=([a\!+\!2k, a], [a\!+\!2k\!-\!2, a\!+\!2])$, by \cite{Kraft-Procesi:classical}.  Here $k \geq 2$.
This is a valid pair of partitions when $a$ is even if $\g$ is of type $C$ and odd if $\g$ is of types $B$ or $D$.  
By  Proposition \ref{slice_row_equiv}, we can replace 
$(\lambda, \mu)$ by  $(\lambda', \mu')$.
We note that
the centralizer of $e_1$ in $G(V_1)$ is a subgroup of the centralizer of $e$ in $G$.  This gives an embedding
of the component group fo $e_1$ of $G(V_1)$, which is the Klein $4$-group $V_4$, 
into $A(e)$, given by sending $A(e_1)$ to the subgroup 
of $A(e)$ generated by $x_{a\!+\!2k\!-\!2}$ and $x_{a\!+\!2}$.  
The other parts contributing to $A(e)$ act trivially on 
$\g(V_1)$ and hence trivially on the slice. 

\subsubsection{$G$ is of type $C$, $a$ even}
The weighted Dynkin diagram for $\OC_{\lambda}$ is
$$\overbrace{2 \dots 2}^{k} \overbrace{0 2 0 2 \dots 0 2} ^{a/2}$$
where the final node corresponds to the long simple root.
Taking the associated parabolic subgroup $P$, the map $\pi$ above is birational.

If $a=0$, we are in type $C_k$ and $\0_\lambda$ is regular.  There is a unique minimal degeneration to 
$\0_\mu$, the subregular orbit.  Hence, using  \cite[Lemma 5.8]{FJLS},  
there are exactly $k$ orbits for $A(e)$ on the $\mathbb P^1$'s in the fiber, which implies the action on $D_{k+1}$ must be $C_k$.  Indeed, the sole $A(e)$-orbit of size two is coming from the orbital variety corresponding to the long root.  (We could use knowledge of the Springer fiber in this case too).

Next if $a>0$, which means $a \geq 2$ since $a$ is even, 
there is the degeneration of $\lambda$ to $\mu$ but also to $\mu' = [a{+}2k, a{-}2, 2]$.
The latter minimal degeneration is equivalent to $([a], [a{-}2, 2])$, 
which is a simple surface singularity of type $D_{\frac{a}{2}+1}$ with action of $A(e_{\mu'})$ having  $\frac{a}{2}$ orbits, by induction.
Since the total number of  component group orbits on the fiber is $k+\frac{a}{2}$, 
that leaves $k$ orbits corresponding to the degeneration
to $e = e_\mu$.  This forces  the action on $D_{k+1}$ to be non-trivial and must be $C_k$, as desired.
Indeed, we could explicitly see by using instead the parabolic $P$ for the diagram 
$$\overbrace{0 2 0 2 \dots 0 2} ^{a/2} \overbrace{2 \dots 2}^{k},$$
which is also birational to $\0_\lambda$.  Then the orbital varieties for $\0_\mu$ correspond to 
the last $k$ two's.  The last node gives the $A(e)$-orbit with two elements.

Finally, the element $x_{a{+}2k} x_{a}$ acts trivially on the fibers, since it belongs to the center of $G$.  
So both $x_{a{+}2k}$ and $x_a$ will yield the outer action on the slice.

\subsubsection{$G$ is of type $D$, $a$ odd}

The weighted Dynkin diagram for $\OC_{\lambda}$ is
$$\overbrace{2 \dots 2}^{k-1} \overbrace{0 2 0 2 \dots 0 2} ^{(a-1)/2} 2$$
where the two final nodes correspond to orthogonal simple roots and the first $k-1$ nodes form a subsystem of type $A_{k-1}$.
Taking associated parabolic subgroup $P$, the map $\pi$ above is birational.
This is similar to the type $C$ case.   If we work in the full orthogonal group then $A(e)$ permutes 
the two $\PB^1$'s corresponding to the tails of the Dynkin diagram.
Finally, the element $x_{a{+}2k} x_{a}$ acts trivially on the fiber, since it belongs to the center of $G$.  
So both $x_{a{+}2k}$ and $x_a$ will yield the outer action on the slice.

\subsection{Type $c$ singularity}

This is a simple surface singularity of type $A_{2k-1}$ and it arises whenever 
$(\lambda, \mu)$ is equivalent to 
$$([a\!+\!2k\!+\!1, a, a], [ a\!+\!2k\!-\!1, a\!+\!1, a\!+\!1]).$$
Here, $a$ is even for types $B,D$ and odd for type $C$.
As in \S\ref{b_sing} using Proposition \ref{slice_row_equiv},  
we can first reduce to the case of $([a\!+\!2k\!+\!1, a, a], [ a\!+\!2k\!-\!1, a\!+\!1, a\!+\!1])$
where $G$ is type $B$ for $a$ even and type $C$ for $a$ odd.

The $A_{2k-1}$ simple surface singularity arises from the diagonal cyclic group $\Gamma$ of order $2k$ in $\SL_2(\C)$.  
The centralizer of $\Gamma$ in $\SL_2(\C)$ is the diagonal one-dimensional torus, leading to an invariant of degree two for the action of $\Gamma$ on $\C^2$.  Since the isomorphism to the slice is $\C^*$-equivariant, we see that the slice, upon projection to $\cg(\sg)$ must be isomorphic to the Lie algebra of the torus for $\Or(2)$ corresponding to the part $a{+}1$   in $\mu$.  The outer automorphism on $\C^2/\Gamma$ acts non-trivially on the diagonal torus, we see that $x_{a{+}1}$ gives rise to the action, while $x_{a\!+\!2k\!-\!1}$ acts trivially.

\subsection{Type $d$ degeneration}

This is again a simple surface singularity of type $A_{2k-1}$ and it arises whenever
$(\lambda, \mu)$ is equivalent to 
$$([a{+}2k{+}1, a{+}2k{+}1, a], [a{+}2k, a{+}2k, a{+}2]).$$
This is a valid pair of partitions when $a$ is even in type $C$ and odd in 
types $B$ or $D$.

As in the previous case,  it is enough to work it out for  the case 
$\lambda =  [a{+}2k{+}1, a{+}2k{+}1, a]$ and 
$\mu =[a{+}2k, a{+}2k, a{+}2]$.
when $G$ of type $C$ for $a$ even and type $B$ when $a$ is odd.
As before, we can detect the action by looking at the action of $\cg(\sg)$.
Thus $x_{a+2k}$ acts by outer action and $x_{a+2}$ acts trivially.

\subsection{Type $e$ degeneration}

This is a union of simple surface singularities $A_{2k-1} \cup A_{2k-1}$ 
and it arises whenever 
$(\lambda, \mu)$ is equivalent to  
$$([a{+}2k, a{+}2k, a,a],[a{+}2k{-}1, a{+}2k{-}1, a{+}1, a{+}1]).$$
Here, $a$ is odd in type $C$ and even in types $B$ or $D$.
As before, we are reduced to the case of 
$\lambda= [a{+}2k, a{+}2k, a,a]$ and 
$\mu = [a{+}2k{-}1, a{+}2k{-}1, a{+}1, a{+}1]$
in type $D$ for $a$ even and type $C$ for $a$ odd.
Here $C(\sg) \simeq {\rm O}(2) \times  {\rm O}(2)$.

The full automorphism group of the singularity is dihedral of order eight.  We want to show $A(e)$ 
embeds as the Klein $4$-group generated by the reflections through the midpoints of edges of the square.  This will
follow if we show that there is at least one orbit of size $4$ of $A(e)$ on the fiber over $e$.  This will force
there to be $k-1$ orbits of size $4$ on the $4k-2$  projective lines and one orbit of size $2$,

By the method of the previous two sections, the element $x_{a{+}2k{-}1}x_{a+1}$ must fix each irreducible component
and act by outer automorphism on each one individually.  This is because it is acting by $-1$ on the two-dimensional 
space $\cg(\sg)$.
The action $x_{a{+}2k{-}1}$ and $x_{a{+}1}$ can be determined in each case separately. 	
Both of them will interchange the two irreducible components.
%

\subsubsection{C case}

The Dynkin diagram of $\OC_{\lambda}$ is 
$$\overbrace{0 2 0 2 \dots  0 2}^{ k }\overbrace{0 0 0 2 0 0 0 2 \dots 0 0 0 2}^{(a-1)/2} 0 0.$$
Using the method of \S\ref{b_sing}, if $a=0$, 
we find there $k$ orbits for the unique minimal degeneration to $\0_\mu$.    At the same time,
there are $4k-2$ projective lines in the fiber over $e$.
Since $A(e)$ is isomorphic to $V_4$, the possible orbit sizes are $1$, $2$, and $4$.  
The only way for this to work is for there to be $k-1$ orbits of size $4$ and one orbit of size $2$.  
Therefore the action is  as desired.

When $a>0$, there is another minimal degeneration to $\0_\mu' = [(a{+}2k)^2, (a{-}1)^2, 2]$.
Then $(\lambda, \mu')$ is equivalent to $([a,a], [(a{-}1)^2, 2])$, which is of a type $d$ generation and has the form $B_{\frac{a-1}{2}}$.
This degeneration therefore accounts for 
$\frac{a-1}{2}$  of the $k+\frac{a-1}{2}$  orbits, leaving $k$ for the studied minimal degeneration and the result follows as in the $a=0$ case.  

\subsubsection{D case}

If $G$ is the full orthogonal group, there is a single orbit $\OC_{\lambda}$ with the given singularity.
Working in the special orthogonal group, there are two very even orbits with the given partition, interchanged by the action of any element of $O(N)$ not in $\SO(N)$.  This is where the two irreducible components are coming from, as they both degenerate to $\mu$, which has an element fixed by action.  Hence, the result follows.

\subsection{$G$ is special orthogonal}\label{subsect:A_action_spec_orth}
When $G$ is special orthogonal, there are two situations where the component group action changes.

For the type $b$ singularity when $\mu$ has exactly two odd parts  (e.g., $\mu =[8,8,5,3]$ or $\mu =[8,8,5,5]$ ).  In this case the component group is trivial.  If there were more than two odd parts for this degeneration, there would have to be  at least 3 distinct odd parts, which would guarantee the non-trivial action of $A(e)$.

For the type $e$ singularity when $\mu$ again has only the odd parts that appear in the local version of $\mu$ in Table \ref{codim4ormore}.  Otherwise, $\mu$ would have at least two additional odd parts (possibly equal), which would ensure the same action by $V_4$.  Now if $\mu$ has only the odd parts, say $[(a{+}2k{-}1)^2, (a{+}1)^2]$, then since its others parts are even, the partition $\lambda$ must be very even.  Then there are two orbits corresponding to $\lambda$ and $A(e) \simeq \SG_2$ acts by outer automorphism on each degeneration to $\mu$, so both are of type $B_k$.

\subsection{Dimension four or greater}\label{Section:action_codim4_more}

In \cite{FJLS}, we studied the image of 
$C(\sg)$ in $\rm Aut(\cg(\sg))$ via the adjoint action in the exceptional groups, 
and then restricted the action to orbits of simple factors of $\cg(\sg)$. 
We observed using \cite{Sommers:Bala-Carter} (also computable using \cite{Alexeevsky})
that $C(\sg)$ tends to act by outer automorphisms of simple factors of $\cg(\sg)$ that admit outer automorphisms.
As in Corollary \ref{Corollary:simple_factors_control_singularity}, the minimal (and minimal special degenerations) are controlled by $\cg(\sg)$ for most cases when $\dim(\SOe) \geq 4$.  We then recorded this outer action on minimal singularities $a_n$, $d_n$, $d_4$, and $e_6$, when they arose.

A more intrinsic framework is to use the intersection homology $IH^*(\SOe)$ 
of $\SOe$ under the induced action of $A(e)$.
Let $\rm p(X) = \sum_i \dim(IH^{2i}(X))q^i$.
When $\SOe \simeq \overline{\0_{\rm min.sp}}$ for the minimal special orbit in the simple Lie algebra $\cg_0$, then we have
$$p(\SOe) = q^{e_1-1}+q^{e_2-1}+\dots+q^{e_k}$$ 
where $e_i$ are the exponents of $\cg_0$ (see \cite{Lusztig:adjacency}).

Let $\cg_0$ be of type $A_k$, $D_k$, or $E_6$ and $\theta$ be an outer involution
and denote by $\cg_0':=\cg_0^{\langle \theta \rangle}$ the fixed subalgebra.
Then ${\langle \theta \rangle}$ acts trivially on the part of $IH^*(\overline{\0}_{\rm min.sp})$ 
corresponding to exponents of 
$\cg_0'$ and by the sign representation on the remaining part.  In other words,
$$IH^*(\overline{\0}_{\rm min.sp})^{\langle \theta \rangle} = IH^*(\overline{\0}_{\rm min.sp}/{\langle \theta \rangle}).$$
In the case of $\SG_3$ acting by outer automorphisms when $\cg_0$ is of type $D_4$,
the $\SG_3$-invariants on $IH^*(\overline{\0}_{\rm min.sp})$ correspond to the exponents of $G_2$ (namely, $1$ and $5$)
and $\SG_3$ acts by the reflection representation on the two-dimensional space $IH^4(\overline{\0}_{\rm min.sp})$ for the two exponents of $D_4$ equal to $3$ and again
$$IH^*(\overline{\0}_{\rm min.sp})^K = IH^*(\overline{\0}_{\rm min.sp}/K).$$

Since $C^\circ(\sg)$ acts trivially on $IH^*(\SOe)$,  there is an action of $A(e)$ on  $IH^*(\SOe)$
and this gives an intrinsic way to see the outer action when the slice is isomorphic to the closure of a minimal special orbit, rather than appealing to the action on $\cg_0$ itself, when $\cg_0$ is the relevant factor of $\cg(\sg)$ as in 
Corollary \ref{Corollary:simple_factors_control_singularity}.

%

\subsection{Type $h$ singularity}\label{subsection:A_action_h_sing}

This corresponds to the closure of the minimal nilpotent orbit in type $D_k$.
The local action of the reductive centralizer coincides with the orthogonal group $\rm O(2k)$, which
contains an outer involution of $\so_{2k}$ and so the $A(e)$ acts by outer action and coincides with the
$d_k^+$.

In the case of $G=\SO(2N)$, the component group $A(e)$ will still act by outer involution in this way, except 
for those cases where the partition $\mu$ contains exactly one odd part (of even multiplicity $2k$).   

\subsection{Exceptional degenerations}

\subsubsection{The case of $d_{n+1}/V_4$}

From \S\ref{Subsect:h_sp}, the $V_4$ is acting on $d_{n+1}$ with $\theta_1$ outer and $\theta_2$ inner.
Hence, 
$$IH^*(d_{n+1}/V_4) \simeq IH^*(d_{n+1}/\theta_1) \simeq  IH^*(b^{sp}_n).$$

Let $\SOe$ be a slice of type $h_{sp}$.
Recall that there is a natural $\rm O(2n)$ action on $d_{n+1}/V_4$ which is the fixed points of the $V_4$-action on $O(2n)$.  Under the isomorphism to $\SOe$, the $\rm O(2n)$-action becomes the action of $C(\sg)$ on $\SOe$.

Since the action of $O(2n)$ on $d_{n+1}$ is also inner, we find that $O(2n)$ acts trivially on
$IH^*(d_{n+1}/V_4)$, and hence $C(\sg)$ acts trivially on $d_{n+1}/V_4$.

On the other hand, it seems relevant that if we take the minimal degeneration to $\mu$ in $\SOe$, which is of type $h$, then indeed $A(e)$ acts by outer action on this $d_n$.

\subsubsection{The case of $d_{4}/S_4$}

From the proof of \cite[Theorem 4.11]{Fu-Juteau-Levy-Sommers:GeomSP},
$S_4$ acts on $d_4$ by the semi-direct product of an inner $V_4$ group and an outer $S_3$ group.
Hence, 
$$IH^*(d_{4}/S_4) \simeq IH^*(d_{4}/S_3) \simeq  IH^*(g^{sp}_2).$$

Let $\SOe$ be one of the two slices of type $d_{4}/S_4$.  There is a natural action
of $SL_3 \rtimes \SG_2$ on $d_{4}/S_4$, which is the fixed points of $S_4$ on the adjoint group of type $D_4$.
This action, as in the previous section, corresponds to $C(\sg)$ on $\SOe$ under the equivariant isomorphism.
The action of the $\SG_2$ is inner, so we again find that $A(e)$ acts trivially on $\SOe$.  

The minimal degeneration in $\SOe$ corresponds to an $a_2$ singularity coming from $\cg(\sg)$ and we note again that the action on this singularity is outer, $a_2^+$.

\section{Action of the canonical quotient} \label{canonical_quotient}

Let $\SOe$ be the slice for a minimal special degeneration.
In this section we explain how the kernel $H$ of the homomorphism from 
$A(e)$ to Lusztig's canonical quotient $\bar{A}(e)$ acts on the slice  $\SOe$.
When $H$ acts by outer action, the exchange of singularities under the duality is not as expected.

\subsection{Exceptional groups}

\begin{proposition}
Assume $G$ is connected of exceptional type and  $H$ is nontrivial for $A(e)$.
Then there exists a unique minimal special degeneration to $e$
and the degeneration is $C_k$ for $k \geq2$, $\mu$, or $d_4/\SG_4$.

In the $C_k$ cases, $H = A(e) = \SG_2$ acts by outer automorphism on $\SOe$.

In the one case $(D_7(a_1), E_8(b_6))$, where the singularity is of type $\mu$ (which is $C_2$ upon normalization),
$H$ acts trivially on $\SOe$ and the induced action of $\bar{A}(e)$ is by outer automorphism.

In the two cases where $\SOe$ is $d_4/\SG_4$, the action of $H$ is trivial on $IH^*(\SOe)$, however
the action of $H$ on the minimal degeneration to $e$ is outer.
\end{proposition}

\begin{proof}
The cases in the exceptional groups where $H$ is nontrivial for $A(e)$ can be read off from \cite{Sommers:Duality}. 
They are $A_2$ and $F_4(a_2)$ in type $F_4$; $A_3+A_2$ and $E_7(a_4)$ in type $E_7$; and $A_3+A_2$, $D_4+A_2$, $E_7(a_4)$, $D_5+A_2$, $E_8(b_6)$, $D_7(a_1)$, and $E_8(b_4)$ in type $E_8$. 
In all these cases, there is a unique $\0$ such that $(\0, \0'_e)$ is a minimal (special) degeneration.

If $e$ does not belong to the $E_8(b_6)$ orbit and is  not of type $A_2$ in $F_4$ or $D_4+A_2$ in $E_8$, 
then $A(e) \simeq \SG_2$ and $H = A(e)$, so that $\bar{A}(e)$ is trivial.
Since we already know from \cite{FJLS} that $A(e)$ is acting by outer
action, we see that $H$ does too.

If $e$ is not of type $A_2$ in $F_4$ or $D_4+A_2$ in $E_8$, then 
there exists a unique $\0$ where $\0$ is a minimal (special) degeneration to $e$ and 
$A(e)$ acts non-trivially $\mathcal{S}_{\0,e}$.  The singularity of 
$\mathcal{S}_{\0,e}$ is a simple surface singularity $D_{k+1}$, yielding a $C_k$ singularity.  
It follows that $H$ itself 
acts non-trivially on $IH^*(\SOe)$.

For the $E_8(b_6)$ orbit, we have $A(e) \simeq  \SG_3$ and $H$ is the cyclic group of order $3$.
The slice of $D_7(a_1)$ at $e$ is of type $\mu$, which is not normal, but has normalization of type $A_3=D_3$
and we previously computed that $A(e)$ acts by outer action upon the normalization, so that the normalization is $C_2$.  
Since the elements of $H$ cannot give an outer action, the outer action descends to $\bar{A}(e)$.

If $e$ is  of type $A_2$ in $F_4$ or $D_4+A_2$ in $E_8$, then 
$\mathcal{S}_{\0,e}$ is isomorphic to $d_4/S_4$.  
This occurs for $(F_4(a_3),A_2)$ and $(E_8(a_7), D_4+A_2)$.  
By \S \ref{Section:action_codim4_more}, the action of $A(e)$ on $IH^*(\SOe)$ is trivial.
On the other hand, the action of $A(e)$ is non-trivial on  $IH^*(\mathcal{S}_{\0'',e})$, 
where $\0''$ is the minimal non-special orbit between
$\0$ and $\0_e$. 
\end{proof}

\subsection{Classical types}
Let $X$ be of type $B,C,D,$ or $C'$.
Let $\epsilon$ and $\epsilon'$ be defined for the given type.
For a partition $\mu$, define   
$$R := \{ s \ | \  s \not\equiv \epsilon, m_\mu(s) \neq 0, h_\mu(s) \not \equiv \epsilon'\}.$$
For $s \in R$, define $s'$ to satisfy $s' \not\equiv \epsilon$, $m(s')\neq 0$ and maximal for this property with $s'<s$.
Set $s'=0$ if no such $s'$ exists and set $x_0 = 1$ in $A(e)$.
Define $H$ to the be subgroup of $A(e)$ generated by the following elements of $A(e)$:
\begin{equation}
	H := \langle  x_s x_{s'} \ | \ s \in R \rangle.
\end{equation}	

By \cite{Sommers:Duality}, the quotient of $A(e)$ by $H$ gives Lusztig's canonical quotient in $B$,$C$.  In $D$ we get an extra factor of  ${\mathbf Z}_2$, as opposed to working in the special orthogonal group.  In type $C'$, we get something new and we take this as the definition of the canonical quotient (we can give a definition that is simliar to the characterization in {\it op.\!\! cit.})
Let $r = \# R$.
Then the canonical quotient ${\bar A}(e)$ is elementary abelian with $r$ generators in types $C,D,C'$ and $r-1$ in type $B$.

Let $G$ be classical group.
\begin{proposition}
If the type of the minimal special degeneration is 
type $C_n$ for $n \geq 2$ and $l \equiv \epsilon'$ in Table \ref{codim2}, then $H$ acts non-trivially on slice.  
Otherwise, $H$ acts trivially on $\SOe$.

When the slice $\SOe \simeq d_n/V_4$, $H$ acts by outer automorphism on the $\mathcal{S}_{\0'',e}$,
where $\0''$ is the minimal non-special orbit between $\0$ and $\0_e$.
\end{proposition}

\begin{proof}
In Table \ref{codim2}, the element $e \in \0_\mu$.  
For the type $B_n$ singularities:
\begin{itemize}
\item Type $c$.  The elements acting non-trivially on the slice involve $x_{a+1}$.  But the part $a{+}2n{-}1$ has
height $l+1$ and the part $a+1$ has height $l+3$, both of which  are congruent to $\epsilon'$.  Hence, all the elements
in $H$ do not involve $x_{a+1}$ and $H$ acts trivially.
\item Type $d$.  The elements acting non-trivially on the slice involve $x_{a+2n}$.  But the part $a{+}2n$ has
height $l+2$, which is congruent to $\epsilon'$.  
For $s$ minimal for $s>a{+}2n$ and $s \not \equiv \epsilon$, we must have $h(s)$ even since
the parts between $s$ and $a{+}2n$ are congruent to $\epsilon$ and so come with even multiplicity.
Hence none of the elements generating $H$ involve $x_{a+2n}$ and $H$ acts trivially.
\item Type $e$.  The elements acting non-trivially on the slice involve $x_{a+2n-1}$ or $x_{a+1}$.  
Both of these parts have height congruent to $\epsilon'$, so as in the type $d$, 
none of the elements generating $H$ involve either part and $H$ acts trivially.
\end{itemize}

Next, we treat the case of type $b$.  
Here, $H$ acts non-trivially if either $x_{a+2n-2}$ or $x_{a+2}$ are involved in a generator of $H$, but not both.
The height of $a{+}2n{-}2$ is $l+1$ and $a+2$ is $l+2$.  
If $l \equiv \epsilon'$, then $x_{a+2n-2}x_{a+2}$ is in $H$, but no other generator involves $a{+}2n{-}2$ 
since $s$ minimal for $s>a{+}2n{-}2$ and $s \not \equiv \epsilon$ must have $s \equiv \epsilon'$.
So $H$ acts trivially.
But if $l \not\equiv \epsilon'$, then some element of the form $x_{a+2}x_{s'}$ is in $H$ and $H$ does not act trivially
on the slice.  
This happens exactly when the second diagram in Table \ref{quartet} occurs, for the upper right $C_{n+1}$ singularity.  Hence it is denoted $C_{n+1}^*$.

For exceptional type $g_{sp}$ there is no $A(e)$ outer action. However, $H$ acts non-trivially on the 
$\mathcal{S}_{\0'',e}$.  The proof is similar to the above cases, as is the proof that $H$ acts trivially for type $h$.
\end{proof}

\section{Combinatorial statement of duality, classical groups}\label{Section:comb_proof_duality}

Let $(\lambda,\mu)$ be a minimal special degeneration in types $B,C,D$ or $C'$.
Then all three of  $(f(\lambda),f(\mu))$, $(d(\mu),d(\lambda))$, and 
$(d_{LS}(\mu),d_{LS}(\lambda))$ are minimal special degenerations.  
We now prove that the four types of singularities are given by the Figure \ref{quartet}.

There is a bit more going on in the first quartet where the partition pattern of type $c$ will get interchanged under
internal duality with that of type $f^1_{sp}$ and 
that of type $d$ is interchanged under internal duality with type $f^2_{sp}$. 
Let internal duality map from type $X$ to $d(X)$.  
\begin{lemma} \label{vertical_duals}
	The vertical arrows in Figure 	\ref{quartet} are correct.  
	
	In particular, if $l \equiv \epsilon'_X$, 
	then the singularity of type $b$ is interchanged with the minimal special $g_{sp}$.  
	
	If $l \not \equiv \epsilon'_X$, 
	then the singularity of type $b$ is interchanged with the minimal special $h$.  
	
	Each of the two types of $B_n$ singularities switches with a corresponding type of $b_n^{sp}$ singularity:  type $c$ with 
	$f^1_{sp}$ and type $d$ with $f^2_{sp}$.  
	The type $e$ singularity switches with type $h_{sp}$ and type $a$ goes to type $a$.
\end{lemma}

\begin{proof}
	If $(\lambda,\mu)$ becomes $(\lambda',\mu')$ after removing $l$ rows and $s$ columns, then clearly 
	$(d(\mu),d(\lambda))$ becomes $(d(\mu'),d(\lambda'))$ after removing $s$ rows and $l$ columns.
	So it is sufficient to work with the irreducible forms in Tables \ref{codim2} and \ref{codim4ormore} to understand how a pair of partitions behaves under $d$.
	
	A quick check shows that, under the transpose operation on partitions,  
	the partition in the table of type $c$ is interchanged with the one of type $f^1_{sp}$;
	the type $d$ partition is interchanged with the one of type $f^2_{sp}$;
	the one of type $e$ is interchanged with the one of type $h_{sp}$; 
	the one of type $b$ is interchanged with the partition found in ether $g_{sp}$ and $h$;
	and type $a$ is self-dual.
	
	The behavior of $\epsilon$ and $\epsilon'$ under $d$ is also follows:
	$$(\epsilon, \epsilon')_X  + (\epsilon', \epsilon)_{d(X)} \equiv (1,1).$$
	As a result, for the interchange of the first three partition types described above, the switching of 
	$l$ and $s$ upon going from $(\lambda,\mu)$  to $(d(\mu),d(\lambda))$ 
	agrees with the restriction on $l$ and $s$ in the dual type $d(X)$ given in the Tables.
	However, for type $b$,  when $l \equiv \epsilon'_X$, the interchange is with type $g_{sp}$
	and when $l \not \equiv \epsilon'_X$, the interchange is with type $h$.  The self-dual type $a$ is clear.
\end{proof}

Next we want to write down the rules for the horizontal arrows in Figure \ref{quartet}.   
First, we start with the case of type $b$, the $C_n$ singularity from \cite{Kraft-Procesi:classical}.
In that case, we can write $(\lambda,\mu)$ locally 
as $([(2n\!+\!s)^t,s^u],(2n\!+\!s)^{t-1}, 2n\!-\!2\!+\!s, s\!+\!2, s^{u-1}]$ for positive integers $t$ and $u$,
where $s \not\equiv \epsilon_X$ where $f$ maps $X$ to $f(X)$.  

\begin{lemma} \label{horizontal}
	Assume the degeneration is of type $b$ for a given $n$.  
	When $l \equiv \epsilon'_X$, under the $f$ map,
	the degeneration $(\lambda, \mu)$ is carried to a singularity of  type
	\begin{eqnarray*}
		e & \text{ if } t \geq 2, u \geq 2  \\
		d & \text{ if } t \geq 2, u=1  \\
		c & \text{ if } t=1, u\geq 2  \\		 
		C_{n+1} & \text{ if } t=u=1  
	\end{eqnarray*}
	When $l \not \equiv \epsilon'_X$, then type $b$ is exchanged with type $b$ with $n$ replaced by $n{-}1$,
	that is $C_{n-1}$.
\end{lemma}

\begin{proof}
	Since $l$ rows are removed, $h_\lambda(2n{+}s) = l{+}1$ and $h_\lambda(s)=l+1 + m_\lambda(s)$. 
	Also, $h_\mu(2n{-}2+s)=l{+}1$  and $h_\mu(s{+}2)=l{+}2$ and $m_\mu(2n{-}2+s)=m_\mu(s{+}2)=1$
	if $n >2$; and 	$h_\mu(2n{-}2+s)=h_\mu(s{+}2)=l{+}2$ and $m_\mu(s{+}2)=2$ if $n=2$.
	All these parts are not congruent to $\epsilon_X$ since $s \not \equiv \epsilon$.
	
	If $l \not \equiv \epsilon'_X$,then 
	$h_\lambda(2n{+}s) \equiv \epsilon'_X$
	and $h_\lambda(s) +m_\lambda(s) \equiv\epsilon'_X$.
	In particular, in Lemma \ref{f_map}, $2n{+}s$ obeys line 1 or 3, and $s$ obeys lines 2 or 3.
	Hence the partition $[2n{+}s, s]$ in $\lambda$ gets replaced by $[2n{+}s{-}1, s{+}1]$ in $f(\lambda)$
	regardless of the parities of $m_\lambda(2n{+}s)$ and $m_\lambda(s)$.
	Moreover, $[2n{-}2{+}s,s{+}2]$ in $\mu$ goes to $[2n{-}3{+}s,s{+}3]$
	since $s{+}2$ obeys line 2 if $n >2$ and line $3$ if $n=2$.
	Hence we end of with $f(\lambda), f(\mu)$ locally equal to 
	$$([2n{-}1{+}s, s{+}1],[2n{-}3{+}s,s{+}3]),$$ which is of type $C_{n-1}$ with $s+1$ rows removed.  Note
	that $s+1 \not \equiv \epsilon_{f(X)}$ since $\epsilon_X$ and $\epsilon_{f(X)}$ have different parities.
	%
	%
	
	Now if  $l \equiv \epsilon'$, there are four cases to consider depending on $t=m_\lambda(2n{+s})$
	and $u = m_\lambda(s)$:  if $t =u=1$, $\lambda$ is locally $[2n{+}s,s]$
	
	Now if  $l \equiv \epsilon'$, then
	$h_\lambda(2n{+}s) \not \equiv \epsilon'_X$ so if $m(2n{+}s) \geq 2$, 
	the last two values of  $2n{+}s$ in $\lambda$
	are unchanged in $f(\lambda)$ since lines two or four apply in Lemma \ref{f_map}.  But 
	if $m(2n{+}s) =1$, then line two applies and $2n{+}s$ becomes $2n{+}s+1$.
	Now we have $h_\lambda(s) +m_\lambda(s) \not \equiv \epsilon'_X$, so we are in the setting of lines one and four
	of the lemma.  Thus  the first two values of $s$ are unchanged if $m(s) \geq 2$ and the solve values of $s$ 
	changes to $s-1$ if $m(s)=1$.
	Thus $f(\lambda)$ is locally $[2n{+}s,2n{+}s, s,s]$, $[2n{+}s,2n{+}s, s{-}1]$, $[2n{+}s{+}1, s,s]$, or $[2n{+}s{+}1,s{-}1]$
	when $(t,u)$ is $(\geq 2, \geq 2)$, $(\geq 2, 1)$, $(1, \geq 2)$, $(1,1)$, respectively, after
	removing $l{-}1$,  $l{-}1$, $l$ or $l$ rows, respectively.
	
	In the four cases, $\mu$ looks locally like
	$$[2n{+}s, 2n{+}s{-}2,s{+}2,s], [2n{+}s, 2n{+}s{-}2,s{+}2], [2n{+}s{-}2,s{+}2,s], \text{ and }
	[2n{+}s{-}2,s{+}2],$$ respectively.
	Using that $h_\mu(2n{+}s{-}2) \not \equiv \epsilon'_X$ and $m_\mu(2n{+}s{-}2) = 1$  
	and $h_\mu(s)+m_\mu(s) \equiv \epsilon'$  we get
	using Lemma \ref{f_map},
	that $f(\mu)$ is locally 
	$[2n{+}s{-}1, 2n{+}s{-}1,s{+}1,s{+}1]$
	$[2n{+}s{-}1, 2n{+}s{-}1,s{+}1]$
	$[2n{+}s{-}1,s{+}1,s{+}1]$
	$[2n{+}s{-}1,s{+}1]$,
	after removing  $l{-}1$,  $l{-}1$, $l$ or $l$ rows, respectively.
	
	Hence, after removing $s$, $s{-}1$, $s$, or $s{-}1$ rows, respectively,
	we find that $(f(\lambda), f(\mu))$ is of type $e$, $d$, $c$, for the same value of $n$ 
	or  type $b$ with $C_{n+1}$.
\end{proof}

\begin{proposition}\label{prop:quartet_proof}
	The four singularities behave as in the three quartets.
\end{proposition}

\begin{proof}
	The same ideas in the previous lemma, or the fact that $f \circ f = \text{id}$, shows that 
	the type $c$,$d$,$e$ partitions for rank $n$ are mapped to the $C_n$ singularity.  
	
	Now given any of the singularities in Table \ref{codim2} or \ref{codim4ormore}, either the singularity is codimension two or the 
	degeneration obtained using $d$ is.  Either this singularity is type $C_n$ or applying $f$ to the degeneration gives a type 
	$C_n$ singularity with $l \equiv \epsilon'_X$.  Putting this singularity in the upper left corner of the square of singularities, the two lemmas show that the three corners are as shown in Figure \ref{quartet} after we note that if $C_n$ is carried to $C_{n+1}$ under $f$, the value of $l$ stays the same and satisfies $l  \not \equiv \epsilon'_{f(X)}$.  This means that applying $d$ leads to the $h$ singularity according to Lemma \ref{vertical_duals}.  In effect, the three quartets (or four if we include the two different ways to obtain $B_n$ and $b_n^{sp}$) are controlled by the four possibilities for $t$ and $u$ in 
	the $l \equiv \epsilon'_X$ case in  Lemma \ref{horizontal}.
\end{proof}

\begin{remark}
	We can also write the specific conditions that describe the action under $f$ where the singularity starts as $c_n^{sp}$.
	Writing $\mu$ locally as $[(s{+}2)^t,(s{+}1)^n,s^{u}]$ where $t=m(s{+}2)$ and $u =m(s)$,
	then $c_n^{sp}$ maps to $h,f^1_{sp}, f^2_{sp},$ or $h_{sp}$ 
	when $(t,u)$ is $(\geq 1,\geq 1)$,   $(0,\geq1)$,     $(\geq 1, 0)$, $(0,0)$.  
	respectively.  Here, $s=0$ is considered a part in type $C$. 
	These conditions are exactly the conditions from the four cases in Lemma \ref{horizontal} for $C_n$, under $d$
	when $l \equiv \epsilon'_X$.
\end{remark}

\section{Outer automorphisms of $\g$}\label{Section:outer_autos}

\subsection{Outer automorphisms for $A_n$ and $E_6$}
Besides considering the automorphism group for $D_n$, we can do the same for $A_n$ and $E_6$ (and the full automorphism group for $D_4$).
The ideas follow \cite[\S 7.5]{Slodowy:book} where the case of the regular and subregular orbits were handled.
Let $CA(\sg)$, called the {\it outer reductive centralizer} in {\it loc. cit.}, denote the centralizer of $\sg$ in $\rm Aut(\g)$.
There is a surjective map $\pi: \rm Aut(\g) \to \rm Aut(\Delta)$ where $\Delta$ is the Dynkin diagram of $\g$.
Let $e \in \nilcone_{o}$ and $\sg$ be an $\sl_2$-triple $\{e,h,f\}$ for $e$.  
Assume that the weighted Dynkin diagram for $e$ is invariant under 
$\rm Aut(\Delta)$.  Then $\sigma(h)$ is conjugate to $g.h$ for some $g \in G$.
This assumption holds as long as $e$ is not very even in $D_{2n}$ and is not $[5,1^3]$ or $[3,1^5]$ in $D_4$.
Then the proof of Lemma 2 in {\it loc. cit.} still applies.
\begin{lemma}
	For all $\sl_2$-triples $\sg$ as above, the map $\pi$ restricted to $CA(\sg)$ is surjective onto $\rm Aut(\Delta)$. In particular, 
	$CA(\sg)/C(\sg) \simeq \rm Aut(\Delta)$.  
\end{lemma}

As before, notice that $CA(\sg)$ acts on $\cg(\sg)$ and we are interested in this action.
Let $\cg_0 \subset \cg(\sg)$ be a simple factor or the central toral subalgebra of $\cg(\sg)$.
The rest of Lemma 2 in {\it loc. cit.}  generalizes as follows:


\begin{lemma}
Suppose that $\g$ has type $A_n, D_{2n+1},$ or $E_6$. 
Then there is an element $\phi \in CA(\sg)$ that stabilizes $\cg_0$ and acts by $-1$ on some maximal toral subalgebra of $\cg_0$.
In particular, if $\cg_0$ is simple of type  $A_k, D_{2k+1},$ or $E_6$, then the image of $CA(\sg)$ in $\rm Aut(\cg_0)$ is an outer automorphism of order two.
\end{lemma}

\begin{proof}
Fix a toral subalgebra $\hg$ of $\g$. Let $\sigma \in \rm Aut(\g)$ be an automorphism of $\g$ that acts by $-1$ on $\hg$.
Choose $\mg$ to be a standard Levi subalgebra of $\g$ relative to $\hg$ so that $e$ is distinguished in $\mg$ \cite{Carter}.  
Pick $\sg$ so that  $\sg \subset \mg$.
Then $\sigma(\mg) = \mg$ and since $e \in \mg$ is distinguished (and so satisfies the assumption above), there exists 
$g \in G$ (specifically in the subgroup $M$ with Lie algebra $\mg$) such that $\rm Int(g)\circ \sigma$ is the identity on $\sg$.
That is, $\phi:=\rm Int(g)\circ \sigma \in CA(\sg)$.  

Next, by \cite{Carter} the center $\tg$ of $\mg$, which lies in $\hg$, is a maximal toral subalgebra of $\cg(\sg)$.  
Since $M$ acts trivially on $\tg$, it follows that $\phi$ acts by $-1$ on $\tg$.  
In particular, $\phi(\cg_0) = \cg_0$ and $\phi$ acts by $-1$ on its maximal toral subalgebra $\cg_0 \cap \tg$.   Since 
$-1$ is not in the Weyl groups of $A_k, D_{2k+1},$ or $E_6$, the induced automorphism of $\cg_0$ must by outer (since $-w_0$ i
s then clearly outer).
\end{proof}

\begin{remark}
	We want to apply this when $-1$ is not in the Weyl group of $\g$, but it is also applies when $-1$ is in the Weyl group and $\cg(\sg)$ has simple factors of type $a_k$, where it forces $C(\sg)$ to contain an element of order two.  So in $F_4, E_7 $ and $E_8$, this explains why
	the action is upgraded to $a_k^+$ always.  There are no type $D_k$ simple factors of $\cg(\sg)$ for $k>4$ 
	in the exceptional groups, except for the sole case of $d_6$ the minimal orbit in $E_7$.  It remains (along with its dual), the only case where a natural outer action exists on the slice, but the induced action of $CA(\sg)$ does not realize it.  
\end{remark}

\begin{corollary}
Let $G = \rm Aut(\g)$.

In type $A_n$, any singularity for a minimal degeneration, which are not type $A_1$, 
acquires an outer action (that is, becomes $A_k^+$ or $a_k^+$, when $k \geq2$).

In type $E_6$, the singularities with no outer action using $C(\sg)$, all acquire the natural outer action.
\end{corollary}

\begin{proof}
	For the minimal orbit types, we can use the previous lemma since the simple factors 
	of $\cg(\sg)$ are all of type $A$ or $E_6$.  
	
	For the simple surface singularities, we can use the action on the center of $\cg(\sg)$ as in \S \ref{section:A_action}  (the subregular case already follows from Slodowy), since they are simple surface singularities of type $A_k$.  
	
	Finally, the case of $[2a_2]^+$ becomes $[2a^+_2]^+$ since the outer automorphism preserves each simple factor of $\cg(\sg)$.
\end{proof}

In \S \ref{Section:graphs} we include the diagram for the $D_4$ case with the $\SG_3$-action.

\section{A conjecture of Lusztig}\label{lusztig}

In \cite[\S 0.4]{Lusztig:adjacency}, Lusztig associated to each minimal special degeneration in $\g$ a certain 
Weyl group $W'$.   We describe it in the classical groups.

Let $h:=\dim(\SOe)/2+1$ for the slice $\SOe$ associated to the degeneration.
In types $B$,$C$,$D$, the slice $\SOe$ is either dimension two or the dimension of the minimal special orbit for a Lie algebra $\cg_0$ of type $B$, $C$, $D$.  The dimension of the latter is $2h'-2$ where $h'$ is the Coxeter number of $\cg_0$.  
Since $h'$ is even for types  $B_k$, $C_k$, $D_k$, respectively equal to $2k$, $2k$, $2k-2$, 
the number $h$ is even when $\g$ has type $B,C,D$.

Here is the assignment of $W'$ to the degeneration, which varies in type $D$ from Lusztig's assignment.
For $\g$ of type $A$, the Weyl group $W'$ associated to the degeneration is the Weyl group of type $A_{h-1}$; 
for $\g$ of types $B/C$, it is of type $B_{h/2}$;
and for $\g$ of type $D$, it is of type $D_{h/2}+1$ when $\SOe$ is of type $d_k$ when $G = \SO(2N)$, and it is 
of type $B_{h/2}$, otherwise.  

Lusztig worked with the associated special representations of $W$, the Weyl group of $\g$.  
Our definition of $W'$ is slightly different in type $D$.  It is necessary to include more special representations than
were given in \cite{Lusztig:adjacency} for the conjecture to hold for the case of associating $W'=W(D_{h/2}+1)$, as the
following lemma shows.  In op. cit., only the case of $s=0$ and $\nu$ with parts all equal to $1$ was considered for associating 
 $W'=W(D_{h/2}+1)$.

\begin{lemma}
Let $(\lambda, \mu)$ be a minimal special degeneration of type $d_k$ in $\g$ for $G = \SO(2N)$. 
By \S \ref{subsection:A_action_h_sing}, this occurs precisely when 
$\mu$ has a single odd part equal to $2s+1$ with even multiplicity $2j$.

Then the Springer representations of $W(D_N)$ 
attached to $\0_\lambda$ and $\0_\mu$ with the trivial local systems on $\0_\lambda$ and $\0_\mu$ 
are given respectively by the bipartitions
$$([\nu, s{+}1,s^{j}, \nu'], [\nu, (s{+}1)^{j-1}, \nu'])$$ 
and 
$$([\nu, s^{j}, \nu'], [\nu, (s{+}1)^j, \nu']),$$ 
where $\nu$ is a partition with smallest part at least $s+1$ and $\nu'$ is 
a partition with largest part at most $s$, and such that $2|\nu|+2|\nu'|+j(2s+1) = N$.

Moreover, any such bipartition corresponds to a minimal special degeneration with one odd part in $\mu$ (necessarily of even multiplicity).
\end{lemma}

\begin{proof}
	We carried out the algorithm in \cite{Carter}.
\end{proof}

Let $\rm p'(\SOe) = \sum_i \dim(IH^{2i}(\SOe)^{A(e)})q^i$.
In classical types, \cite[Conjecture 1.4]{Lusztig:adjacency} can be interpreted as saying:
\begin{theorem}
Let $G$ be classical and assume $h>1$.
Then
$$\rm{p'}(\SOe) = q^{e_1-1}+q^{e_2-1}+\dots+q^{e_k-1}$$ 
where $e_1, e_2, \dots, e_k$ are the exponents of $W'$.
\end{theorem}

\begin{proof}
Since $h>1$, then $\dim(\SOe)\geq4$ and so by Table \ref{codim4ormore}, the singularity 
of $\SOe$ in types $B$ and $C$ is either $c^{sp}_{h/2}$, $b^{sp}_{h/2}$, $d_{{h/2}+1}^+$, or $d_{{h/2}+1}/V_4$.
Each of these satisfies
$$\rm p'(\SOe) = 1+q^{3}+\dots+q^{h-1},$$ 
by \S \ref{Section:action_codim4_more}, 
which are the exponents of $B_{h/2}$.

This also holds for $\SO(2N)$, except for the case where the singularity is $d_{{h/2}+1}$, in which
case $\rm{p'}(\SOe) = 1+q^{3}+\dots+q^{h-1}+q^{h/2},$ 
coming from the exponents will be those of $D_{{h/2}+1}$.
\end{proof}

In the exceptional groups a similar interpretation exists (except that a variation is needed for the $3$ exceptional orbits in $E_7$ and $E_8$).  That is to say, the simple factors of $\cg(\sg)$ under the action of $A(e)$ explain 
why $\rm p'(\SOe)$ sees the exponents that Lusztig observes in \cite[\S 4]{Lusztig:adjacency}.
We also point out a typo in {\it loc. cit.}:  the label of $567_{46} -- 1400_{37}$ should be $B_5$.

\section{Duality}\label{Section:duality_theorem}

We can now gather up our results to state the duality result for a minimal special degeneration $(\0,\0')$
in $\g$ and its dual minimal special degeneration $(d_{LS}(\0'),d_{LS}(\0))$ in $^L \g$, the Langlands dual Lie algebra of $\g$.

Let $X$ be the normalization of an irreducible component of the slice $\SOe$ for $(\0,\0')$, where $e \in \0'$.
Let $Y$ be an irreducible component of the slice $\mathcal S$ for  $(d_{LS}(\0'),d_{LS}(\0))$.  Let $e' \in d_{LS}(\0)$.

By \cite{Lusztig:adjacency}, we can assume that $\dim(\SOe)=2$.  This also follows
from Proposition \ref{prop:quartet_proof} and inspection of the graphs in \S \ref{Section:graphs}.
Hence, $X$ is simple surface singularity.  Denote by $\rm Out(X)$ its group of outer automorphisms, which 
are the graph automorphisms of the ADE diagram corresponding to $X$ as in \S \ref{section:A_action}.
From \cite{FJLS} and \S \ref{section:A_action}, we know that $A(e)$ acts transitively 
on the irreducible components of $\SOe$ and of $\mathcal S$.  
Let $J(e) \subset A(e)$ be the stabilizer of $X$.  Let $K(e)$ be the image of $J(e)$ in $\rm Out(X)$.

On the dual side, let $\rm Out(Y)$ be the outer automorphisms of the minimal symplectic leaf in $Y$, as discussed in \S \ref{Section:action_codim4_more}.    Let $J(e') \subset A(e')$ be the stabilizer of $Y$ and let 
$K(e')$ be the image of $J(e')$ in the $\rm Out(Y)$.

The pair of minimal degenerations falls into one of three mutually exclusive cases:
\begin{enumerate}
	\item The map from $K(e)$ to its image in $\bar{A}(e)$ is bijective and the map from 
	$K(e')$ to its image in $\bar{A}(e')$ is bijective.
	\item The map from $K(e)$ to its image in $\bar{A}(e)$ is not bijective.
	\item The map from $K(e')$ to its image in $\bar{A}(e')$ is not bijective.
\end{enumerate}
That these are mutually exclusive follows from \S \ref{canonical_quotient}.

\begin{theorem}
	 Let $G$ be of adjoint type or $G = \rm Aut(\g)$ or $G = \rm O(8)$.
	 We have the following duality of singularities under the Lusztig-Spaltenstein involution:
	 \begin{itemize}
	\item In case (1):  
	\begin{enumerate}
		\item[(a)]   If $(X, K(e))$ corresponds to a simple Lie algebra $\mg$,  in the sense of Slodowy, 
		then $Y/K(e')$ is isomorphic to the closure of the
	minimal special nilpotent orbit in the fixed subalgebra $(^L \mg)^{K(e')}$, where 
	$^L \mg$ is a simple component of the reductive centralizer of $e'$ in $^L \g$.
	\item[(b)] If the pair $(X, K(e))$ is of type $A_k^+$, then $Y/K(e')$ is isomorphic to $a_k/\SG_2$. 
	\end{enumerate}
	
	\item In case (2):  The pair $(X,K(e))$ is of type $C_{n+1}$, and $Y/K(e')$ is isomorphic to $c_n^{sp}$.
	
	\item In case (3):  The pair $(X,K(e))$ is of type $C_n$ or $G_2$, and $Y$ is 
	isomorphic to $d_{n+1}/V_4$ or $d_4/\SG_4$, respectively.  
	\end{itemize}
\end{theorem}

\begin{proof}
	This amounts to gathering up our results.  For the classical groups, the duality statements follow from  \S \ref{Section:comb_proof_duality}.  For the exceptional groups, it is by inspection of the graphs in 
	\S \ref{Section:graphs}.  When $G = \rm Aut(\g)$, we make use of \S \ref{Section:outer_autos}.
\end{proof}

\begin{remark}
We noticed in case (2) that the simple surface singularity $B_n$ is a two-fold cover of the simple surface singularity $C_{n+1}$.  
In case (3)  we observe that $b_n^{sp}$ is a two-fold cover of $d_{n+1}/V_4$ (as in \S \ref{Subsect:h_sp}) and that 
$g_2^{sp}$ is a four-fold cover of $d_4/\SG_4$.   
Accessing these covers would allow cases (2) and (3) to behave like the more well-behaved duality  in case (1).
\end{remark}

As we were finishing this paper, the preprint \cite{MMY:Extended_Sommers_Duality} appeared.  We expect there is some overlap
with our results, but we have not had a chance yet to understand the connection.

\newpage
\section{Graphs}\label{Section:graphs}

We include here the Hasse diagrams of the minimal special degenerations for the exceptional Lie algebras, except for the straightforward $G_2$, as well as several examples in the classical types.
We write $(Y)$ when we only know that the normalization of the Slodowy slice singularity is isomorphic to $Y$.
See \cite[\S 6.2]{FJLS} for a discussion of the component group action and branching in the exceptional types for the more complicated cases
in the graphs.  
We write $C^*_{n+1}$ to indicate when the kernel of the map to Lusztig's canonical quotient acts by outer action on $\SOe$ as in \S\ref{canonical_quotient}.   We use the notation $A_1$ in the exceptional groups, but 
in the classical groups  we use $B_1$ or $C_1$ to be consistent with Table \ref{codim2}; these are all the same singularity.

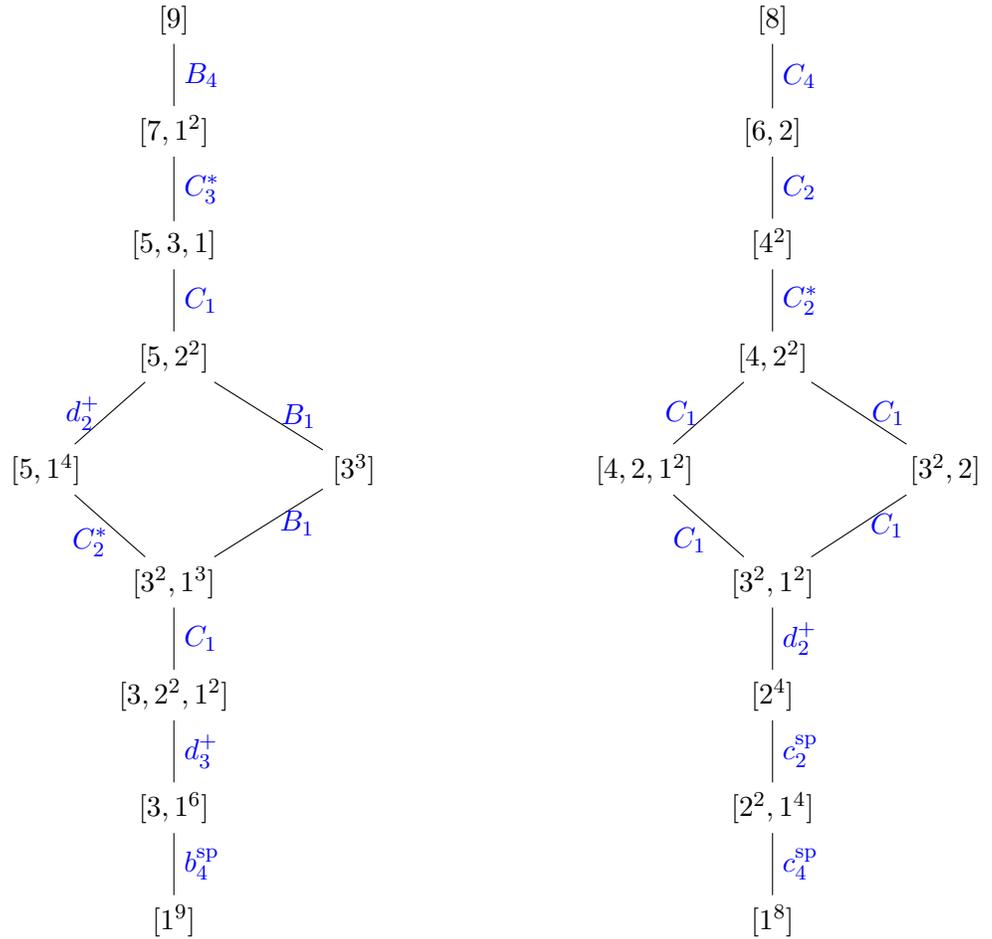
\begin{tikzpicture}
	\begin{scope}[shift={(-4,0)}]
		\matrix (B4special)
		[matrix of math nodes,
		nodes in empty cells,
		nodes={outer sep=0pt,minimum size=0pt},
		column sep={1.7cm,between origins},
		row sep={1.5cm,between origins}]
		{
			&&& | (a) | [9]  \\
			&&& | (b) | [7,1^2] \\
			&&& | (c) | [5,3,1]  \\
			&&& | (d) | [5,2^2]  \\
			&& | (e) | [5,1^4] &&  | (f) |  [3^3] \\
			&&&  | (g) | [3^2,1^3]  \\
			&&&  | (h) | [3,2^2,1^2] \\
			&&&  | (i) | [3,1^6] \\
			&&&  | (j) | [1^9]\\
		};

		\draw (a) -- (b) node[midway, right] {${\color{blue}B_4}$};
		\draw (b) -- (c) node[midway, right] {${\color{blue}C_3^*}$};
		\draw (c) -- (d) node[midway, right]{${\color{blue}C_1}$};  
		\draw (d) -- (e) node[midway, left] {${\color{blue}d_2^+}$}; 
		\draw (d) -- (f) node[midway, right=1pt]{${\color{blue}B_1}$}; 
		\draw (f) -- (g) node[midway,right=.5pt] {${\color{blue}B_1}$}; 
		\draw (e) -- (g) node[midway, below left=-3pt] {${\color{blue}C_2^*}$};
		\draw (g) -- (h) node[midway, right] {${\color{blue}C_1}$};
		\draw (h) -- (i) node[midway, right] {${\color{blue}d^+_3}$};  
		\draw (i) -- (j) node[midway, right] {${\color{blue}b^{\special}_4}$};
		
	\end{scope}
	
	\begin{scope}[shift={(4,0)}]
		\matrix (C4special)
		[matrix of math nodes,
		nodes in empty cells,
		nodes={outer sep=0pt,minimum size=0pt},
		column sep={1.7cm,between origins},
		row sep={1.5cm,between origins}]
		{
			&&& | (a) | [8]  \\
			&&& | (b) |[6,2] \\
			&&& | (c) | [4^2] \\
			&&& | (d) | [4,2^2] \\
			&& | (e) | [4,2,1^2]   && | (f) | [3^2,2]   \\
			&&& | (g) | [3^2,1^2] \\
			&&&  | (h) |  [2^4]  \\
			&&&  | (i) |  [2^2,1^4] \\
			&&&  | (j) | [1^8] \\
		};
		
		\draw (a) -- (b) node[midway, right] {${\color{blue}C_4}$};
		\draw (b) -- (c) node[midway, right] {${\color{blue}C_2}$};
		\draw (c) -- (d) node[midway, right]{${\color{blue}C_2^*}$};  
		\draw (d) -- (e) node[midway, left] {${\color{blue}C_1}$}; 
		\draw (d) -- (f) node[midway, right=1pt]{${\color{blue}C_1}$}; 
		\draw (f) -- (g) node[midway,right=.5pt] {${\color{blue}C_1}$}; 
		\draw (e) -- (g) node[midway, below left=-3pt] {${\color{blue}C_1}$};
		\draw (g) -- (h) node[midway, right] {${\color{blue}d_2^+}$};
		\draw (h) -- (i) node[midway, right] {${\color{blue}c^{\special}_2}$};  
		\draw (i) -- (j) node[midway, right] {${\color{blue}c^{\special}_4}$};

		\end{scope}
\end{tikzpicture}

\begin{tikzpicture}
	\begin{scope}[shift={(-8,0)}]
	\matrix (C6special)
	[matrix of math nodes,
	nodes in empty cells,
	nodes={outer sep=0pt,minimum size=0pt},
	column sep={1.7cm,between origins},
	row sep={1cm,between origins}]
	{
		&&& | (1) | [12]  \\
		&&& | (2) |[10,2] \\
		&&& | (3) |[8,4] \\
		&& | (4) |[8,2^2]  && | (6) |[6^2] \\
		&& | (5) |[8,2, 1^2] && | (7) |[6,4,2] \\		
		&& | (8) |[6,4,1^2] && | (11) |[5^2,2] \\
		&&  | (9) |[6,2^3] & | (12) |[5^2,1^2] &| (13) |[4^3] \\
		&& | (10) |[6,2,1^4] && | (14) |[4^2,2^2] \\		
		&& | (15) |[4^2,1^4]  & | (16) |[4,2^4] 	& | (19) |[3^4] \\		
		&& | (17) |[4,2^3,1^2]&& | (20) |[3^2,2^3] \\		
		&& | (18) |[4,2,1^6] && | (21) |[3^2,2^2,1^2] \\		
		&& | (22) |[3^2,1^6]  && | (23) |[2^6] \\		
		&&& | (24) |[2^4,1^4] \\		
		&&& | (25) |[2^2,1^8] \\	
		&&& | (26) | [1^{12}] \\	
};
	
	\draw (1) -- (2) node[midway, right] {${\color{blue}C_6}$};
	\draw (2) -- (3) node[midway, right] {${\color{blue}C_4}$};
	\draw (3) -- (4) node[inner sep=2.5pt,midway, above,scale=.95]{${\color{blue}C_2^*}$};  
	\draw (3) -- (6) node[midway, right]{${\color{blue}C_2}$};
	\draw (4) -- (5) node[midway, right]{${\color{blue}C_1}$};	  
	\draw (4) -- (7) node[midway, above]{${\color{blue}C_3}$};	  
	\draw (6) -- (7) node[midway, right]{${\color{blue}C_3^*}$};	  
	\draw (5) -- (8) node[midway, right]{${\color{blue}C_3}$};	  
	\draw (7) -- (8) node[inner sep=1.5pt,,midway, above]{${\color{blue}C_1}$};	  
	\draw (7) -- (11) node[midway, right]{${\color{blue}C_1}$};
	\draw (8) -- (12) node[midway, right]{${\color{blue}C_1}$};
	\draw (8) -- (9) node[midway, right]{${\color{blue}B_1}$};
	\draw (11) -- (12) node[midway, right]{${\color{blue}C_1}$};
	\draw (11) -- (13) node[midway, right]{${\color{blue}B_1}$};

	 \draw (9) -- (10) node[midway, right]{${\color{blue}c^{\special}_2}$};
\draw (9) -- (14) node[midway, right,scale=.8]{${\color{blue}C_2}$};
\draw (12) -- (14) node[midway, right,scale=.8]{${\color{blue}[2B_2]^+}$};
\draw (13) -- (14) node[midway, right,scale=.8]{${\color{blue}C_2}$};

	 \draw (10) -- (15) node[midway, right]{${\color{blue}C_2}$};
\draw (14) -- (15) node[inner sep=1pt, midway, above,scale=.9]{$\small{\color{blue}c^{\special}_2}$};
\draw (14) -- (16) node[inner sep=6pt,midway, right,scale=.7]{${\color{blue}d_3/V_4}$};
\draw (14) -- (19) node[midway, right,scale=.8]{${\color{blue}c^{\special}_2}$};

	 \draw (15) -- (17) node[midway, right]{${\color{blue}B_1}$};
\draw (16) -- (17) node[midway, right]{${\color{blue}C_1}$};
\draw (16) -- (20) node[midway, right]{${\color{blue}C_1}$};
\draw (19) -- (20) node[midway, right]{${\color{blue}B_1}$};

 \draw (17) -- (18) node[midway, right]{${\color{blue}c^{\special}_3}$};
\draw (20) -- (21) node[midway, right]{${\color{blue}C_1}$};
\draw (17) -- (21) node[midway, right]{${\color{blue}C_1}$};

 \draw (21) -- (23) node[midway, right]{${\color{blue}d_3^+}$};
  \draw (21) -- (22) node[inner sep=1pt,midway, above,scale=.95]{${\color{blue}c^{\special}_3}$};
  \draw (18) -- (22) node[midway, right]{${\color{blue}C_1}$};

 \draw (22) -- (24) node[midway, right]{${\color{blue}d_2^+}$};   
 \draw (23) -- (24) node[midway, right]{${\color{blue}c^{\special}_2}$};
 \draw (24) -- (25) node[midway, right]{${\color{blue}c^{\special}_4}$};
	\draw (25) -- (26) node[midway, right]{${\color{blue}c^{\special}_6}$};
	
\end{scope}
	\begin{scope}[shift={(0,0)}]
		\matrix (B6special)
		[matrix of math nodes,
		nodes in empty cells,
		nodes={outer sep=0pt,minimum size=0pt},
		column sep={1.7cm,between origins},
		row sep={1cm,between origins}]
		{
			&&& | (1) | [13]  \\
			&&& | (2) |[11,1^2] \\
			&&& | (3) |[9,3,1] \\
			&& | (4) |[9,2^2]  && | (6) |[7,5,1] \\
			&& | (5) |[9, 1^4] && | (7) |[7,3^2] \\		
			&& | (8) |[7,3,1^3] && | (11) |[5^2,3] \\
			&&  | (9) |[7,2^2,1^2] & | (12) |[5^2,1^3] &| (13) |[5,4^2] \\
			&& | (10) |[7,1^6] && | (14) |[5,3^2,1^2] \\		
			&& | (15) |[5,3,1^5]  & | (16) |[5,2^4] 	& | (19) |[3^4,1] \\		
			&& | (17) |[5,2^2,1^4]&& | (20) |[3^3,2^2] \\		
			&& | (18) |[5,1^8] && | (21) |[3^3,1^4] \\		
			&& | (22) |[3^2,1^7]  && | (23) |[3,2^4,1^2] \\		
			&&& | (24) |[3,2^2,1^6] \\		
			&&& | (25) |[3,1^{10}] \\	
			&&& | (26) | [1^{13}] \\	
		};
		
			\draw (1) -- (2) node[midway, right] {${\color{blue}B_6}$};
		\draw (2) -- (3) node[midway, right] {${\color{blue}C_5^*}$};
		\draw (3) -- (4) node[midway, right]{${\color{blue}C_1}$};  
		\draw (3) -- (6) node[midway, right]{${\color{blue}C_3}^*$};
		
			\draw (4) -- (5) node[midway, right]{${\color{blue}d_2^+}$};	  
		\draw (4) -- (7) node[midway, right]{${\color{blue}B_3}$};	  
		\draw (6) -- (7) node[midway, right]{${\color{blue}C_2}$};	  
		
			\draw (5) -- (8) node[midway, right]{${\color{blue}C_4^*}$};	  
		\draw (7) -- (8) node[midway, above]{${\color{blue}B_1}$};	  
		\draw (7) -- (11) node[midway, right]{${\color{blue}C_2^*}$};
		\draw (8) -- (12) node[midway, right]{${\color{blue}C_2^*}$};
		\draw (8) -- (9) node[midway, right]{${\color{blue}C_1}$};
		\draw (11) -- (12) node[midway, right]{${\color{blue}B_1}$};
		\draw (11) -- (13) node[midway, right]{${\color{blue}C_1}$};
		
		 \draw (9) -- (10) node[midway, right]{${\color{blue}d_3^+}$};
		 	 \draw (9) -- (14) node[midway, right]{${\color{blue}B_2}$};  
		\draw (12) -- (14) node[midway, right]{${\color{blue}C_2}$};
		\draw (13) -- (14) node[midway, right]{${\color{blue}[2B_2]^+}$};
		
		 \draw (10) -- (15) node[midway, right]{${\color{blue}C_3^*}$};
		 \draw (14) -- (15) node[midway, above, inner sep=1pt, scale=.95]{${\color{blue}b^{\special}_2}$};
		  \draw (14) -- (16) node[inner sep=2pt,midway, below,scale=.95]{${\color{blue}c^{\special}_2}$};
		  	  \draw (14) -- (19) node[midway, right,scale=.9]{${\color{blue}d_3/V_4}$};
		 
			 \draw (15) -- (17) node[midway, right]{${\color{blue}C_1}$};
		\draw (16) -- (17) node[midway, right]{${\color{blue}d_2^+}$};
			\draw (16) -- (20) node[midway, right]{${\color{blue}B_1}$};
		\draw (19) -- (20) node[midway, right]{${\color{blue}C_1}$};
		
		 \draw (17) -- (18) node[midway, right]{${\color{blue}d_4^+}$};
		\draw (20) -- (21) node[midway, right]{${\color{blue}d_2^+}$};
		\draw (17) -- (21) node[midway, right]{${\color{blue}B_1}$};

		 \draw (18) -- (22) node[midway, right]{${\color{blue}C_2^*}$};
		 \draw (21) -- (22) node[inner sep=1pt,midway, above,scale=.95]{${\color{blue}b^{\special}_3}$};
		 \draw (21) -- (23) node[midway, right]{${\color{blue}c^{\special}_2}$};
		 \draw (22) -- (24) node[midway, right]{${\color{blue}C_1}$};
		\draw (23) -- (24) node[midway, right]{${\color{blue}d^{+}_3}$};
		\draw (24) -- (25) node[midway, right]{${\color{blue}d^{+}_5}$};
		\draw (25) -- (26) node[midway, right]{${\color{blue}b^{\special}_6}$};
		
	\end{scope}
\end{tikzpicture}

\begin{figure}[H]\caption{Duality between $\parti_D^{sp}(10)$ and $\parti_C^{asp}(10)$ }
\begin{tikzpicture}
		\hspace*{-0.15\linewidth}
	\begin{scope}[shift={(-4,0)}]
		\matrix (D5special)
		[matrix of math nodes,
		nodes in empty cells,
		nodes={outer sep=0pt,minimum size=0pt},
		column sep={1.8cm,between origins},
		row sep={1cm,between origins}]
		{
			&&& | (1) | [9,1]  \\
			&&& | (2) |[7,3] \\
			&& | (3) |[7,1^3]&& | (4) |[5^2] \\
			&&& | (5) |[5,3,1^2]  \\
		&&& | (7) |[4^2,1^2] \\
			&& | (6) |[5,1^5] && | (8) |[3^3,1] \\
			 &&& | (9) |[3^2,2^2] \\
			&&& | (10) |[3^2,1^4] \\		
			&& | (11) |[3,1^7]&& | (12) |[2^4,1^2] \\
			&&& | (13) |[2^2,1^6] \\	
			&&& | (14) | [1^{10}] \\	
		};
		
		\draw (1) -- (2) node[midway, right] {${\color{blue}C_4}$};
		\draw (2) -- (3) node[midway, right] {${\color{blue}B_1}$};
		\draw (2) -- (4) node[midway, right]{${\color{blue}C_2}$};  
		\draw (3) -- (5) node[midway, right]{${\color{blue}C_3}$};
		\draw (4) -- (5) node[midway, right]{${\color{blue}B_2}$};	 
		\draw (5) -- (6) node[midway, left]{${\color{blue}b^{\special}_2}$};	  
		\draw (5) -- (7) node[midway, right]{${\color{blue}C_1}$};	  
		\draw (6) -- (10) node[midway, left]{${\color{blue}C_2}$};	  
		\draw (7) -- (8) node[midway, right]{${\color{blue}B_1}$};	  
		\draw (8) -- (9) node[midway, right]{${\color{blue}C_1}$};
		\draw (9) -- (10) node[midway, right]{${\color{blue}d_2^+}$};
		\draw (10) -- (11) node[midway, right,scale=.9]{${\color{blue}b^{\special}_3}$};
		\draw (10) -- (12) node[midway, above]{${\color{blue}c^{\special}_2}$};
		\draw (11) -- (13) node[midway, right]{${\color{blue}C_1}$};
		\draw (12) -- (13) node[midway, right]{${\color{blue}d^+_3}$};
		\draw (13) -- (14) node[midway, right]{${\color{blue}d^+_5}$};
		
		 \node [below=1cm, align=flush center,text width=8cm] at (14)
		{
			$D_5$ Minimal Special Degenerations  
		};

	\end{scope}
	
		\begin{scope}[shift={(4,0)}]
		\matrix (C5antispecial)
		[matrix of math nodes,
		nodes in empty cells,
		nodes={outer sep=0pt,minimum size=0pt},
		column sep={1.8cm,between origins},
		row sep={1cm,between origins}]
		{
			&&& | (1) | [10]  \\
			&&& | (2) |[8,2] \\
			&& | (3) |[8,1^2]&& | (4) |[6,4] \\
			&&& | (5) |[6,2^2]  \\
			&&& | (7) |[4^2,2] \\
			&& | (6) |[6,1^4] && | (8) |[4,3^2] \\
			&&& | (9) |[4,2^3] \\
			&&& | (10) |[4,2^2,1^2] \\		
			&& | (11) |[4,1^6]&& | (12) |[2^5] \\
			&&& | (13) |[2^3,1^4] \\	
			&&& | (14) | [2,1^{8}] \\	
		};
		
		\draw (1) -- (2) node[midway, right] {${\color{blue}C^*_5}$};
		\draw (2) -- (3) node[midway, right] {${\color{blue}C_1}$};
		\draw (2) -- (4) node[midway, right]{${\color{blue}C^*_3}$};  
		\draw (3) -- (5) node[midway, right]{${\color{blue}B_3}$};
		\draw (4) -- (5) node[midway, right]{${\color{blue}C_2}$};	 
		\draw (5) -- (6) node[midway, left]{${\color{blue}c^{\special}_2}$};	  
		\draw (5) -- (7) node[midway, right]{${\color{blue}C^*_2}$};	  
		\draw (6) -- (10) node[midway, left]{${\color{blue}B_2}$};	  
		\draw (7) -- (8) node[midway, right]{${\color{blue}C_1}$};	  
		\draw (8) -- (9) node[midway, right]{${\color{blue}B_1}$};
		\draw (9) -- (10) node[midway, right]{${\color{blue}C_1}$};
		\draw (10) -- (11) node[midway, right]{${\color{blue}c^{\special}_3}$};
		\draw (10) -- (12) node[midway, above]{${\color{blue}b^{\special}_2}$};
		\draw (11) -- (13) node[midway, right]{${\color{blue}B_1}$};
		\draw (12) -- (13) node[midway, right]{${\color{blue}c^{\special}_2}$};
		\draw (13) -- (14) node[midway, right]{${\color{blue}c^{\special}_4}$};
		
			 \node [below=1cm, align=flush center,text width=8cm] at (14)
		{
			$C_5$ {\bf Alternative} Minimal Special Degenerations  
		};
	\end{scope}
	
\end{tikzpicture}
\end{figure}

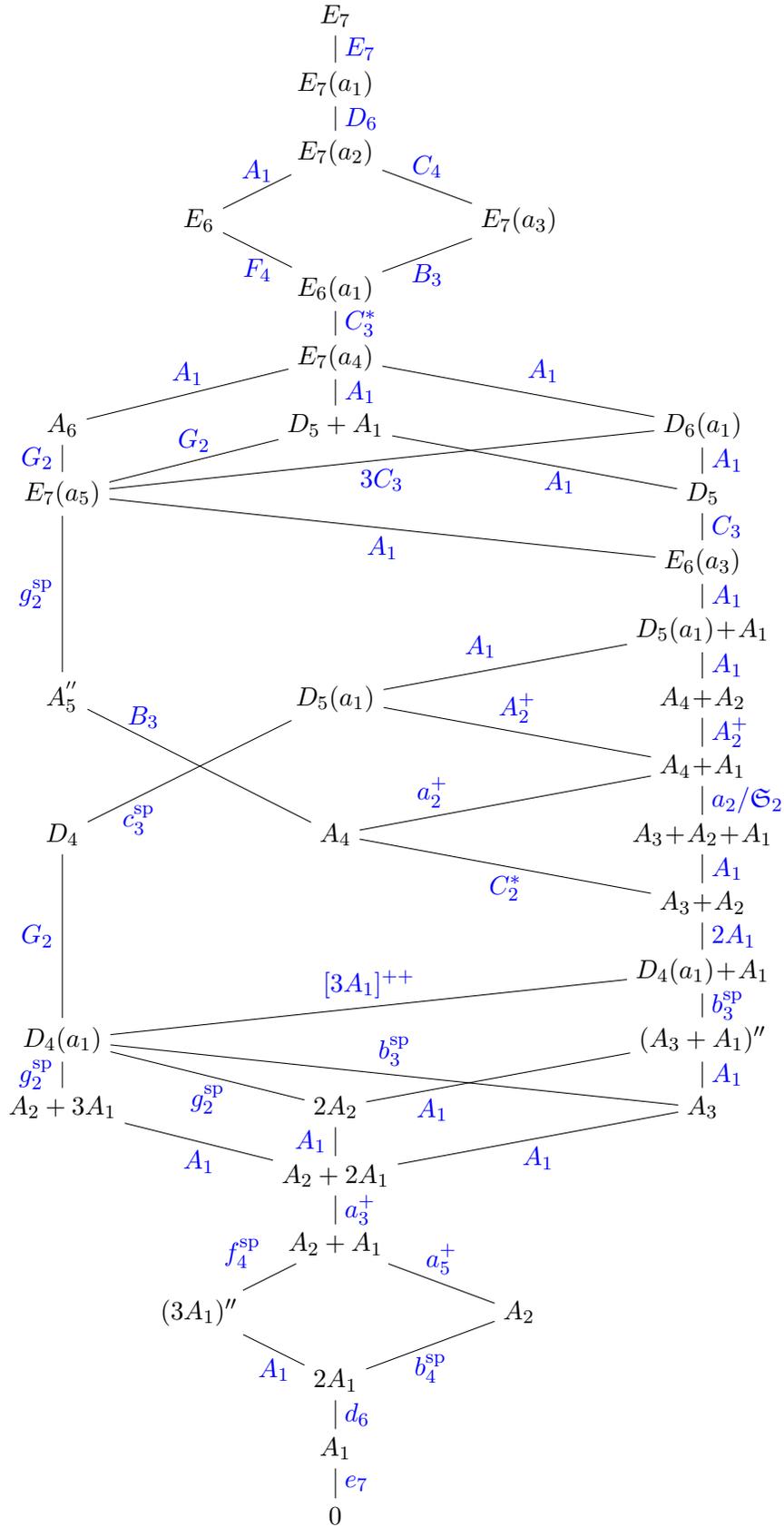
\begin{figure}[H]\caption{$F_4$ and inner $E_6$}
\begin{tikzpicture}
	\begin{scope}[shift={(-7,0)}]
		\matrix (F4special)
		[matrix of math nodes,
		nodes in empty cells,
		nodes={outer sep=0pt,minimum size=0pt},
		column sep={1.7cm,between origins},
		row sep={1.5cm,between origins}]
		{
			&&& | (F4) | F_4  \\
			&&& | (F4a1) | F_4(a_1) \\
			&&& | (F4a2) | F_4(a_2) \\
			&& | (B3) | B_3 && | (C3) | C_3  \\
			&&& | (F4a3) | F_4(a_3) \\
			&&  | (tA2) | \tilde{A_2}  && | (A2) | A_2  \\
			&&&  | (A1+tA1) | A_1 \! +\! \tilde{A_1} \\
			&&&  | (tA1) |  \tilde{A_1} \\
			&&&  | (0) | 0 \\
		};

		\draw (F4) -- (F4a1) node[midway, right] {${\color{blue}F_4}$};
		\draw (F4a1) -- (F4a2) node[midway, right] {${\color{blue}C_3^*}$};
		\draw (F4a2) -- (B3) node[midway, above left]{${\color{blue}A_1}$};  
		\draw (F4a2) -- (C3) node[midway, above right] {${\color{blue}A_1}$}; 
		\draw (C3) -- (F4a3) node[midway, below right]{${\color{blue}4G_2}$}; 
		\draw (B3) -- (F4a3) node[midway, below left] {${\color{blue}G_2}$}; 
		\draw (F4a3) -- (tA2) node[midway, above left=-3pt] {${\color{blue}g^{\special}_2}$};
		\draw (F4a3) -- (A2) node[midway, above right=-5pt] {${\color{blue}d_4/\mathfrak{S}_4}$};

		\draw (tA2) -- (A1+tA1) node[midway, below left] {${\color{blue}A_1}$};  
		\draw (A2) -- (A1+tA1) node[midway, below right] {${\color{blue}A_1}$};
		\draw (A1+tA1) -- (tA1)  node[midway, right] {${\color{blue}a^{+}_3}$};
		\draw (tA1) -- (0)  node[midway, right] {${\color{blue}f^{\special}_4}$};
		
	\end{scope}

	
	\begin{scope}[shift={(1,0)}]
		\matrix (E6special)
		[matrix of math nodes,
		nodes in empty cells,
		nodes={outer sep=0pt,minimum size=0pt},
		column sep={2cm,between origins},
		row sep={1cm,between origins}]
		{
			&&& | (E6) | E_6  \\
			&&& | (E6a1) | E_6(a_1) \\
			&&& | (D5) | D_5 \\
			&&& | (E6a3) | E_6(a_3)  \\
			&&& | (D5a1) | D_5(a_1) \\
			&&  | (A4+A1) | A_4 \! +\!A_1 \\
			&&&& | (D4) | D_4 \\
			&& | (A4) | A_4 \\
			&&& | (D4a1) | D_4(a_1)  \\
			&& | (A3) | A_3  \\
			&&&& | (2A2) | 2A_2 \\
			&& | (A2+2A1) | A_2+2A_1 \\
			&&&  | (A2+A1) | A_2+A_1 \\
			&&&   | (A2) | A_2 \\
			&&& | (2A1) | 2A_1 \\
			&&& | (A1) | A_1 \\
			&&&| (0) | 0 \\
		};
		\begin{scope}[]
			
			\draw (E6) -- (E6a1) node[midway, right] {${\color{blue}E_6}$};
			\draw (E6a1)--(D5) node[midway, right] {${\color{blue}A_5}$};    
			\draw (D5) --   (E6a3) node[midway,right] {${\color{blue}C_3}$}; 
			\draw (E6a3)  --  (D5a1) node[midway,right] {${\color{blue}A_2}$};  
			
			\draw (D5a1) -- (A4+A1) node[midway, above] {${\color{blue}A_2}$};
			\draw (D5a1) -- (D4) node[midway,above=2pt, right=1pt] {${\color{blue}a_2}$};  
			\draw (A4+A1) -- (A4) node[midway,left] {${\color{blue}A_1}$};
			\draw (D4) -- (D4a1) node[midway, above left=-3pt] {${\color{blue}G_2}$};
			\draw (A4) -- (D4a1) node[near start, above right=-3pt] {${\color{blue}3C_2}$};
			\draw (D4a1)  -- (A3) node[midway, above=-1pt] {${\color{blue}b^{\special}_2}$};
			\draw (D4a1)  -- (2A2) node[midway, above right=-3pt] {${\color{blue}g^{\special}_2}$};
			
			\draw (A3) -- (A2+2A1) node[midway,left] {${\color{blue}A_1}$};    
			
			\draw (2A2) -- (A2+A1) node[midway, below right=-3pt] {${\color{blue}A_2}$};   
			\draw (A2+2A1) -- (A2+A1) node[midway, below left] {${\color{blue}a_2}$};   
			\draw (A2+A1) -- (A2) node[midway,right] {${\color{blue}[2a_2]^+}$}; 
			\draw (A2) -- (2A1) node[midway, right] {${\color{blue}b^{\special}_3}$};
			\draw (2A1) -- (A1) node[midway,right] {${\color{blue}a_5}$};
			\draw (A1) -- (0) node[midway,right] {${\color{blue}e_6}$};
			
		\end{scope}
	\end{scope}
\end{tikzpicture}
\end{figure}
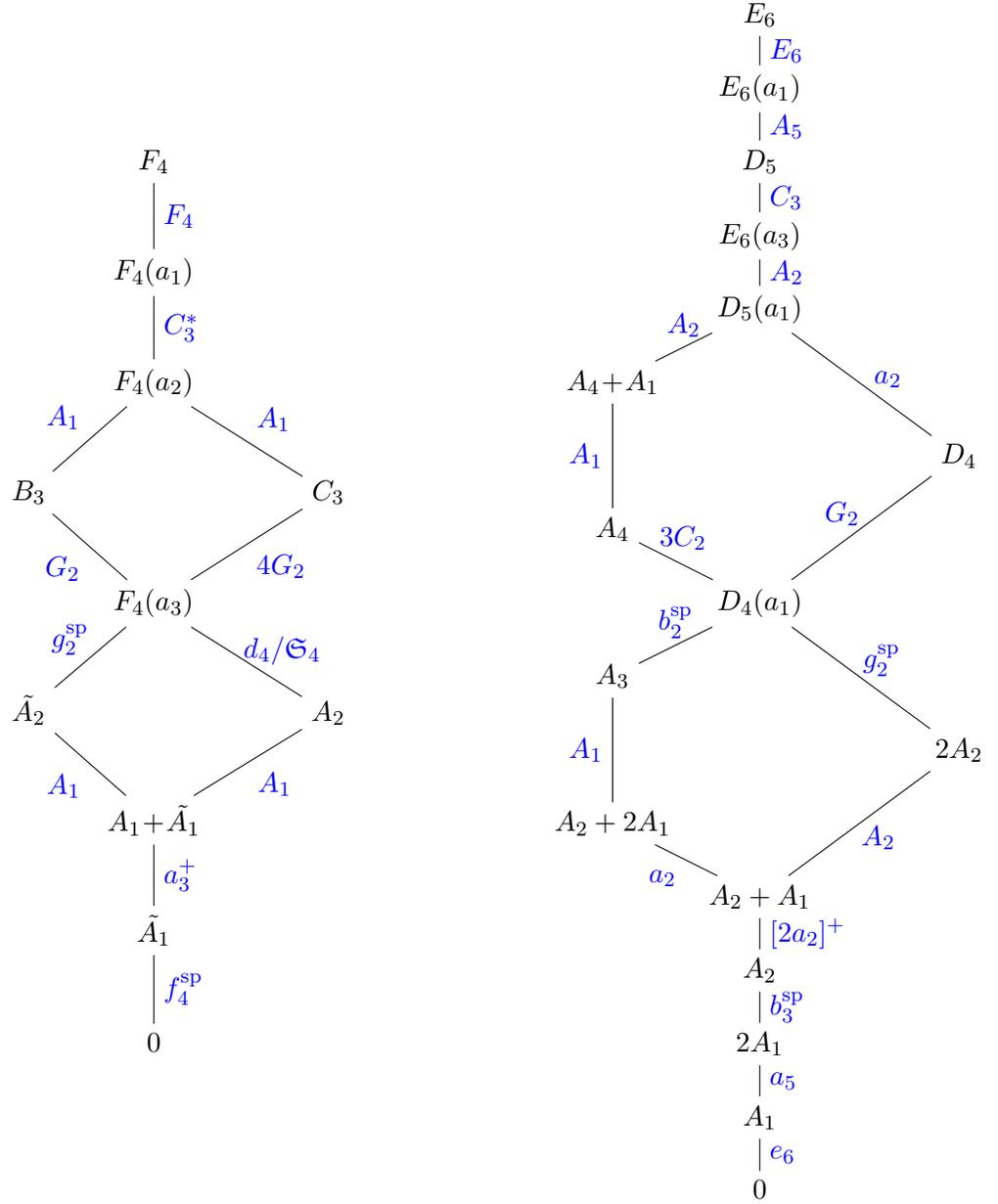

\begin{tikzpicture}
	\matrix (E7special)
	[matrix of math nodes,
	nodes in empty cells,
	nodes={outer sep=0pt,minimum size=0pt},
	column sep={2cm,between origins},
	row sep={1cm,between origins}]
	{
		&&&| (E7) |E_7 \\
		&&& | (E7a1) | E_7(a_1)  \\
		&&& | (E7a2) | E_7(a_2) \\
		&& | (E6) | E_6 &&  | (E7a3) | E_7(a_3) \\
		&&&| (E6a1) | E_6(a_1) \\ 
		&&& | (E7a4) | E_7(a_4)  \\
		& | (A6) | A_6  &&  | (D5+A1) | D_5+A_1  && | (D6a1) | D_6(a_1)\\
		& | (E7a5) | E_7(a_5) &&&& | (D5) | D_5 \\
		&&&&& | (E6a3) | E_6(a_3)  \\
		&&&&& | (D5a1+A1) | D_5(a_1)\! +\!A_1  \\
		& | (A5'') | A^{''}_5 && | (D5a1) | D_5(a_1) && | (A4+A2) | A_4 \! +\!A_2 \\
		&&&&&  | (A4+A1) | A_4 \! +\!A_1 \\
		& | (D4) | D_4 && | (A4) | A_4 && | (A3+A2+A1) | A_3 \! +\!A_2 \! +\!A_1 \\
		&&&&& | (A3+A2) | A_3 \! +\!A_2\\
		&&&&& | (D4a1+A1) | D_4(a_1) \! +\!A_1 \\
		& | (D4a1) | D_4(a_1) 
		&&&& | (A3+A1'') | (A_3+A_1)''  \\
		& | (A2+3A1) | A_2+3A_1 && | (2A2) | 2A_2  && | (A3) | A_3\\
		&&& | (A2+2A1) | A_2+2A_1 \\
		&&& | (A2+A1) | A_2+A_1 \\
		&& | (3A1'') | (3A_1)'' && | (A2) | A_2 \\
		&&& | (2A1) | 2A_1 \\
		&&& | (A1) | A_1 \\
		&&&| (0) | 0 \\
	};
	
	\begin{scope}[]

		\draw (E7) -- (E7a1) node[midway, right] {${\color{blue}E_7}$};
		\draw (E7a1) -- (E7a2) node[midway,right] {${\color{blue}D_6}$};
		\draw (E7a2) -- (E6) node[midway,above] {${\color{blue}A_1}$};  
		\draw (E7a2) -- (E7a3) node[midway, above] {${\color{blue}C_4}$};
		\draw (E6) -- (E6a1) node[midway,below] {${\color{blue}F_4}$}; 
		\draw (E7a3) -- (E6a1) node[midway,below] {${\color{blue}B_3}$};
		\draw (E6a1)--(E7a4) node[midway,right] {${\color{blue}C^*_3}$}; 
		\draw (E7a4) --(A6)  node[midway,above] {${\color{blue}A_1}$}; 
		\draw (E7a4) --   (D5+A1) node[midway,right] {${\color{blue}A_1}$}; 
		\draw (E7a4) --   (D6a1) node[midway,above right]{${\color{blue}A_1}$};
		\draw (D6a1) --   (D5) node[midway, right] {${\color{blue}A_1}$}; 
		\draw (D6a1) --   (E7a5) node[midway, below] {${\color{blue}3C_3}$};
		\draw (A6) --   (E7a5) node[midway,left]{${\color{blue}G_2}$}; 
		\draw (D5+A1) --  (D5) node[midway,below right]{${\color{blue}A_1}$}; 
		\draw (D5+A1) -- (E7a5)  node[midway,above] {${\color{blue}G_2}$};

		\draw (E7a5) --  (E6a3) node[midway,below] {${\color{blue}A_1}$};
		\draw (E7a5) --  ( A5'') node[midway,left] {${\color{blue}g^{\special}_2}$};
		
		\draw (D5) --   (E6a3) node[midway,right] {${\color{blue}C_3}$};  
		
		\draw (E6a3)  --  (D5a1+A1) node[midway,right] {${\color{blue}A_1}$};
		
		\draw (A5'') -- (A4) node[near start,above] {${\color{blue}B_3}$};
		
		\draw (D5a1+A1) -- (A4+A2) node[midway,right] {${\color{blue}A_1}$};
		\draw (D5a1+A1) -- (D5a1) node[midway,above left] {${\color{blue}A_1}$};
		\draw (D5a1) -- (D4) node[near end, below] {${\color{blue}c^{\special}_3}$}; \draw
		(D5a1) -- (A4+A1) node[midway, above] {${\color{blue}A^{+}_2}$}; \draw (A4+A2)
		-- (A4+A1) node[midway,right] {${\color{blue}A^{+}_2}$};
		
		\draw (A4+A1) -- (A3+A2+A1) node[midway,right] {${\color{blue}a_2/\mathfrak{S}_2}$};
		\draw (A4+A1) -- (A4) node[near end, above] {${\color{blue}a_2^+}$};
		
		\draw (A3+A2+A1)  -- (A3+A2) node[midway, right] {${\color{blue}A_1}$}; \draw
		(A4) -- (A3+A2) node[midway, below] {${\color{blue}C^*_2}$}; \draw (A3+A2)  --
		(D4a1+A1) node[midway,right] {${\color{blue}2A_1}$};
		
		\draw (D4) -- (D4a1) node[midway,left] {${\color{blue}G_2}$};  
		
		\draw (D4a1+A1) -- (D4a1) node[midway,above] {${\color{blue}[3A_1]^{++}}$};
		\draw (D4a1+A1) -- (A3+A1'') node[midway, right] {${\color{blue}b^{\special}_3}$};
		
		\draw (D4a1)  -- (A3) node[midway, above] {${\color{blue}b^{\special}_3}$};
		
		\draw (A3+A1'')  -- (A3) node[midway,right] {${\color{blue}A_1}$};
		\draw (A3+A1'')  -- (2A2) node[near end,below] {${\color{blue}A_1}$};
		
		\draw (D4a1) -- (2A2) node[midway, below] {${\color{blue}g^{\special}_2}$};
		\draw (D4a1) -- (A2+3A1) node[midway, left ] {${\color{blue}g^{\special}_2}$};
		
		\draw (2A2) -- (A2+2A1) node[midway, left] {${\color{blue}A_1}$};
		\draw (A3) -- (A2+2A1) node[midway,below] {${\color{blue}A_1}$};
		\draw (A2+3A1) -- (A2+2A1) node[midway,below] {${\color{blue}A_1}$};
		
		\draw (A2+2A1) -- (A2+A1) node[midway,right] {${\color{blue}a^{+}_3}$};
		\draw (A2+A1) -- (A2) node[midway,above] {${\color{blue}a^{+}_5}$};
		\draw (A2+A1) -- (3A1'') node[midway,above left] {${\color{blue}f^{\special}_4}$};
		\draw (3A1'') -- (2A1) node[midway,below] {${\color{blue}A_1}$};
		\draw (A2) -- (2A1) node[midway,below] {${\color{blue}b^{\special}_4}$};
		
		\draw (2A1) -- (A1) node[midway,right] {${\color{blue}d_6}$};
		\draw (A1) -- (0) node[midway,right] {${\color{blue}e_7}$};
		
	\end{scope}
\end{tikzpicture}

\begin{tikzpicture}
	\matrix (specialE8)
	[matrix of math nodes,
	nodes in empty cells,
	nodes={outer sep=0pt,minimum size=0pt},
	column sep={2cm,between origins},
	row sep={.75cm,between origins}]
	{
		&&&&   | (E8) | E_8 &&&\\
		&&&&  | (E8a1) | E_8(a_1) & & &\\
		&&&& | (E8a2) | E_8(a_2) \\
		&&&& |  (E8a3) | E_8(a_3) \\
		&&&& | (E8a4) | E_8(a_4) \\
		&&&& | (E8b4) | E_8(b_4)  \\
		&&& | (E7a1) | E_7(a_1) && | (E8a5) | E_8(a_5) \\
		&&&&| (E8b5) | E_8(b_5)  \\ 
		&&&& | (E8a6) | E_8(a_6) \\
		&&&&| (D7a1) | D_7(a_1) \\
		& | (E6) | E_6 &&  | (E7a3) | E_7(a_3) && | (E8b6) | E_8(b_6) \\
		&&&& | (E6a1+A1) | E_6(a_1)\!+\! A_1 \\
		&&&&& | (D7a2) | D_7(a_2) \\
		&&& | (E6a1) | E_6(a_1) && | (D5+A2) | D_5+A_2\\
		&&&& | (E7a4) | E_7(a_4) && | (A6+A1) | A_6+A_1 \\
		&&& | (D6a1) | D_6(a_1) &&| (A6) | A_6 \\
		&&| (D5) | D_5 && | (E8a7) | E_8(a_7) \\ 
		&&& | (E6a3) | E_6(a_3)   
		&& | (D4+A2) | D_4 \! + \! A_2  \\
		&&&&| (D5a1+A1) | D_5(a_1)\! +\!A_1  && | (A4+A2+A1) | A_4\! +\!A_2\! +\!A_1 \\
		&&& | (D5a1) | D_5(a_1) && | (A4+A2) | A_4 \! +\!A_2 \\
		&&&&& | (A4+2A1) | A_4 \! +\!2A_1 \\
		&&&& | (A4+A1) | A_4 \! +\!A_1 \\
		& | (D4) | D_4 &&  | (A4) | A_4 && | (D4a1+A2) | D_4(a_1) \! +\! A_2 \\
		&&&& | (A3+A2) | A_3 \! +\!A_2 \\
		&&&& | (D4a1+A1) | D_4(a_1) \! +\!A_1 \\
		&&&& | (D4a1) | D_4(a_1) \\ 
		&&& | (A3) | A_3  
		&& | (2A2) | 2A_2 \\
		&&&& | (A2+2A1) | A_2+2A_1 \\
		&&&& | (A2+A1) | A_2+A_1 \\
		&&&& | (A2) | A_2 \\
		&&&& | (2A1) | 2A_1 \\
		&&&& | (A1) | A_1 \\
		&&&& | (0) | 0 \\
	};
	\begin{scope}[]
		\draw (E8) -- (E8a1) node[midway,right] {${\color{blue}E_8}$};
		\draw (E8a1) -- (E8a2) node[midway,right] {${\color{blue}E_7}$};
		\draw (E8a2) -- (E8a3) node[midway,right] {${\color{blue}C_6}$};
		\draw (E8a3) -- (E8a4) node[midway,right] {${\color{blue}F_4}$};
		\draw (E8a4) -- (E8b4) node[midway,right,scale=.92] {${\color{blue}C_4^*}$};
		\draw (E8b4) -- (E7a1) node[midway, above left=-3pt] {${\color{blue}A_1}$};   
		\draw (E8b4) -- (E8a5) node[midway,above right=-3pt] {${\color{blue}C_3}$};  
		\draw (E7a1) -- (E8b5) node[midway, below left=-3pt] {${\color{blue}3(C_5)}$};  
		\draw (E8a5) -- (E8b5) node[midway,below right=-3pt] {${\color{blue}G_2}$};
		\draw (E8b5) -- (E6) node[midway, above] {${\color{blue}g^{\special}_2}$};   
		\draw (E8b5) -- (E8a6) node[midway,right] {${\color{blue}(G_2)}$};
		
		\draw (E8a6) -- (D7a1) node[midway,right,scale=.9] {${\color{blue}C_2^*}$};
		
		\draw (D7a1) -- (E7a3) node[midway,above left=-3pt] {${\color{blue}A_1}$};  
		\draw (D7a1) -- (E8b6) node[midway,above right=-3pt] {${\color{blue}\mu}$};
		
		\draw (E6) -- (E6a1) node[midway, above right=-3pt] {${\color{blue}F_4}$};
		\draw (E7a3) -- (E6a1+A1) node[midway,above right=-4pt] {${\color{blue}(A^{+}_4)}$}; 
		\draw (E8b6) -- (E6a1+A1) node[midway,above left=-4pt] {${\color{blue}A^{+}_2}$};  
		
		\draw (E6a1+A1)--(E6a1) node[midway,above left=-4pt] {${\color{blue}a^{+}_2}$}; 
		\draw (E6a1+A1)--(D7a2) node[midway,below left=-3pt] {${\color{blue}(A_2^+)}$};
		\draw (D7a2) --  (D5+A2) node[midway,right,scale=.9] {${\color{blue}(C_2)^*}$};
		
		\draw (E6a1)--(E7a4) node[midway,above right=-3pt,scale=.92] {${\color{blue}C_3^*}$};  
		\draw (D5+A2)--(E7a4) node[midway, above left=-4pt] {${\color{blue}A_1}$};   
		\draw (D5+A2)--(A6+A1) node[midway, above right=-4pt] {${\color{blue}A_1}$}; 
		
		\draw (E7a4) --   (D6a1)   node[midway, above left=-4pt] {${\color{blue}[2A_1]^+}$}; 
		\draw (E7a4) --   (A6)  node[midway, above right=-4pt] {${\color{blue}A_1}$}; 
		\draw (A6+A1) --  (A6) node[midway, below right=-3pt] {${\color{blue}A_1}$}; 
		\draw (D6a1) --   (D5) node[midway, above left=-4pt] {${\color{blue}b^{\special}_3}$};  
		\draw (D6a1) --   (E8a7) node[midway, above right=-4pt] {${\color{blue}10G_2}$};
		\draw (A6) --   (E8a7) node[midway, above left=-4pt] {${\color{blue}5G_2}$};
		\draw (E8a7) --  (E6a3) node[midway, above left=-3pt] {${\color{blue}g^{\special}_2}$}; 
		\draw (E8a7) --  (D4+A2) node[midway, above right=-5pt] {${\color{blue}d_4/\mathfrak{S}_4}$};
		
		\draw (D5) --   (E6a3) node[midway, above right=-3pt] {${\color{blue}C_3}$};

		\draw (E6a3)  --  (D5a1+A1) node[midway, above right=-3pt] {${\color{blue}A_1}$};

		\draw (D4+A2) -- (A4+A2+A1)node[midway, above right=-3pt] {${\color{blue}A_1}$};  
		\draw (D4+A2) -- (D5a1+A1) node[midway, above left=-3pt] {${\color{blue}A_1}$};
		
		\draw (A4+A2+A1) -- (A4+A2) node[midway, below right=-3pt] {${\color{blue}A_1}$}; 
		\draw (D5a1+A1) -- (A4+A2) node[midway, below left=-3pt] {${\color{blue}A_1}$};  
		\draw (D5a1+A1) -- (D5a1) node[midway, below right=-4pt] {${\color{blue}a^{+}_3}$};
		\draw (D5a1) -- (D4) node[midway,above] {${\color{blue}f^{\special}_4}$};
		\draw (D5a1) -- (A4+A1) node[midway, above right=-4pt] {${\color{blue}A^{+}_2}$}; 
		\draw (A4+A2) -- (A4+2A1) node[midway,right] {${\color{blue}A_1}$};
		
		\draw (A4+2A1) -- (A4+A1) node[midway, above left=-3pt] {${\color{blue}a^{+}_2}$};
		
		\draw (A4+A1) -- (D4a1+A2) node[midway, above right=-3pt] {${\color{blue}a^{+}_2}$};
		\draw (A4+A1) -- (A4) node[midway, above left=-3pt] {${\color{blue}a^{+}_4}$};
		\draw (D4a1+A2) -- (A3+A2) node[midway,below right=-3pt] {${\color{blue}b^{\special}_2}$};
		
		\draw (A4) -- (A3+A2) node[midway,below] {${\color{blue}C_2^*}$};  
		
		\draw (A3+A2)  -- (D4a1+A1) node[midway,right,scale=.9] {${\color{blue}[3A_1]^{++}}$};
		\draw (D4) -- (D4a1) node[midway,right] {${\color{blue}G_2}$};  
		
		\draw (D4a1+A1) -- (D4a1) node[midway,right=-1pt,scale=.89] {${\color{blue}d^{++}_4}$};

		\draw (D4a1)  -- (2A2) node[midway, above right=-3pt] {${\color{blue}[2g^{\special}_2}]^+$};
		\draw (D4a1)  -- (A3) node[midway, above left=-4pt] {${\color{blue}b^{\special}_5}$};
		
		\draw (2A2) -- (A2+2A1) node[midway, below right=-3pt] {${\color{blue}b^{\special}_3}$};
		\draw (A3) -- (A2+2A1) node[midway,below left] {${\color{blue}A_1}$}; 
		\draw (A2+2A1) -- (A2+A1) node[midway,right,scale=.89] {${\color{blue}a^{+}_5}$};
		\draw (A2+A1) -- (A2) node[midway,right] {${\color{blue}e^{+}_6}$};
		\draw (A2) -- (2A1) node[midway,right] {${\color{blue}b^{\special}_6}$};
		\draw (2A1) -- (A1) node[midway,right] {${\color{blue}e_7}$};
		\draw (A1) -- (0) node[midway,right] {${\color{blue}e_8}$};
		
	\end{scope}
\end{tikzpicture}

\begin{figure}[H]\caption{Minimal Special Degenerations  in $D_4$ with $S_3$ outer action} \label{D4_nilcone_with_S3_action}
	\begin{tikzpicture}
		\hspace*{-0.15\linewidth}
		\begin{scope}[shift={(-4,0)}]
			\matrix (D4special)
			[matrix of math nodes,
			nodes in empty cells,
			nodes={outer sep=0pt,minimum size=0pt},
			column sep={1.8cm,between origins},
			row sep={1cm,between origins}]
			{
				&&& | (1) | [7,1]  \\
				&&& | (2) |[5,3] \\
				&&& | (3) | [5,1^3] \\
				&&& | (5) |[3^2,1^2]  \\
				&&&  | (6) |[3,1^5] \\
				&&& | (8) |[2^2,1^4] \\
				&&& | (9) | [1^{8}] \\	
			};
			

			\draw (1) -- (2) node[midway, right] {${\color{blue}G_2}$};
			\draw (2) -- (3) node[midway, right] {${\color{blue}A_1}$};
			\draw (3) -- (5) node[midway, right, scale=.85]{${\color{blue}[3C_2]^{++}}$};
			\draw (5) -- (6) node[midway, right, scale=.85]{${\color{blue}c^{\special}_2}$};	  
			\draw (6) -- (8) node[midway, right,scale=.85]{${\color{blue}[3A_1]^{++}}$};	  
			\draw (8) -- (9) node[midway, right]{${\color{blue}d_4^{++}}$};
			
		\end{scope}

	\end{tikzpicture}
\end{figure}
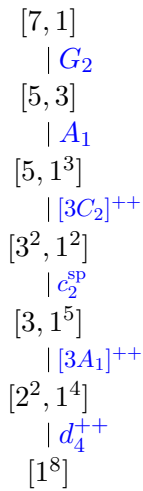

\bibliographystyle{myalpha}
\bibliography{fjls}

\def\cprime{$'$}
\begin{thebibliography}{MBMY23}

\bibitem[Ale79]{Alexeevsky}
A.~V. Alekseevski{\u\i}.
\newblock Component groups of centralizers of unipotent elements in semisimple
  algebraic groups.
\newblock {\em Akad. Nauk Gruzin. SSR Trudy Tbiliss. Mat. Inst. Razmadze},
  62:5--27, 1979.
\newblock Collection of articles on algebra, 2.

\bibitem[Car93]{Carter}
R.~W. Carter.
\newblock {\em Finite groups of {L}ie type}.
\newblock Wiley Classics Library. John Wiley \& Sons Ltd., Chichester, 1993.
\newblock Conjugacy classes and complex characters, Reprint of the 1985
  original, A Wiley-Interscience Publication.

\bibitem[CM93]{C-M}
D.~H. Collingwood and W.~M. McGovern.
\newblock {\em Nilpotent Orbits in Semisimple Lie Algebras}.
\newblock Van Nostrand Reinhold Co, New York, 1993.

\bibitem[FJLS17]{FJLS}
B.~Fu, D.~Juteau, P.~Levy, and E.~Sommers.
\newblock Generic singularities of nilpotent orbit closures.
\newblock {\em Adv. Math.}, 305:1--77, 2017.

\bibitem[FJLS23]{Fu-Juteau-Levy-Sommers:GeomSP}
B.~Fu, D.~Juteau, P.~Levy, and E.~Sommers.
\newblock Local geometry of special pieces of nilpotent orbits.
\newblock arXiv:2308.07398, 2023.

\bibitem[Fu07]{Fu:wreath}
B.~Fu.
\newblock Wreath products, nilpotent orbits and symplectic deformations.
\newblock {\em Internat. J. Math.}, 18(5):473--481, 2007.

\bibitem[HKK23]{Hanany:Coulomb_branches}
A.~Hanany, R.~Kalveks, and G.~Kumaran.
\newblock {$SU(n)$} hyper-{K}\"{a}hler quotients of {$3d$}{ $\mathcal N = 4$}
  {C}oulomb branches and quiver subtraction.
\newblock 2023.
\newblock https://doi.org/10.48550/arXiv.2308.05853.

\bibitem[KP81]{Kraft-Procesi:GLn}
H.~Kraft and C.~Procesi.
\newblock Minimal singularities in {${\rm GL}_{n}$}.
\newblock {\em Invent. Math.}, 62(3):503--515, 1981.

\bibitem[KP82]{Kraft-Procesi:classical}
H.~Kraft and C.~Procesi.
\newblock On the geometry of conjugacy classes in classical groups.
\newblock {\em Comment. Math. Helv.}, 57(4):539--602, 1982.

\bibitem[KP89]{Kraft-Procesi:special}
H.~Kraft and C.~Procesi.
\newblock A special decomposition of the nilpotent cone of a classical lie
  algebra.
\newblock {\em Ast\'erisque}, 173-174:271--279, 1989.

\bibitem[LT11]{Lawther-Testerman}
R.~Lawther and D.~M. Testerman.
\newblock Centres of centralizers of unipotent elements in simple algebraic
  groups.
\newblock {\em Mem. Amer. Math. Soc.}, 210(988):188, 2011.

\bibitem[Lus79]{Lusztig:special}
G.~Lusztig.
\newblock A class of irreducible representations of a {W}eyl group.
\newblock {\em Indag. Math.}, 41:323--335, 1979.

\bibitem[Lus22]{Lusztig:adjacency}
G.~Lusztig.
\newblock Adjacency for special representations of a {W}eyl group.
\newblock {\em Bull. Inst. Math. Acad. Sin. (N.S.)}, 17(2):125--141, 2022.

\bibitem[MBMY23]{MMY:Extended_Sommers_Duality}
L.~Mason-Brown, D.~Matvieievskyi, and S.~Yu.
\newblock Unipotent representations of complex groups and extended sommers
  duality.
\newblock arXiv:2309.14853, 2023.

\bibitem[Slo80]{Slodowy:book}
P.~Slodowy.
\newblock {\em Simple singularities and simple algebraic groups}, volume 815 of
  {\em Lecture Notes in Mathematics}.
\newblock Springer, Berlin, 1980.

\bibitem[Som98]{Sommers:Bala-Carter}
E.~Sommers.
\newblock A generalization of the {B}ala-{C}arter theorem for nilpotent orbits.
\newblock {\em Internat. Math. Res. Notices}, 1998(11):539--562, 1998.

\bibitem[Som01]{Sommers:Duality}
E.~Sommers.
\newblock Lusztig's canonical quotient and generalized duality.
\newblock {\em J. Algebra}, 243:790--812, 2001.

\bibitem[Spa82]{Spaltenstein}
N.~Spaltenstein.
\newblock {\em Classes unipotentes et sous-groupes de {B}orel}, volume 946 of
  {\em Lecture Notes in Mathematics}.
\newblock Springer-Verlag, Berlin, 1982.

\end{thebibliography}

\end{document}